\documentclass[a4paper]{amsart}
\usepackage{amssymb,latexsym}

\usepackage[active]{srcltx}
\usepackage{amsmath}
\usepackage{amsfonts}

\usepackage{amsthm}
\usepackage{amscd}
\usepackage{mathrsfs}
\usepackage{mathtools}
\usepackage{MnSymbol}
\usepackage{fancyhdr}
\usepackage{graphicx}
\usepackage{calc}
\usepackage{url}
\usepackage{verbatim}
\usepackage[cmtip,arrow,matrix,curve,tips,frame]{xy}
\usepackage{hyperref}
\xyoption{matrix}

\usepackage{ifthen}
\usepackage{color}
\usepackage{amsxtra}
\usepackage{amstext}
\usepackage{amssymb}
\usepackage{stmaryrd}
\usepackage{mathrsfs}
\usepackage{pifont}
\usepackage[T1]{fontenc}
\usepackage{textcomp}
\usepackage{bbm}

\usepackage[normalem]{ulem}
\usepackage[new]{old-arrows}

\usepackage{tikz} 
\usetikzlibrary{graphs,matrix,arrows,calc,decorations.pathreplacing,decorations.pathmorphing,decorations.markings,fit}
\usepackage{tikz-cd}

\usepackage{csquotes}

\numberwithin{equation}{subsection}

\theoremstyle{plain}
\newtheorem{thm}[subsection]{Theorem}
\newtheorem*{thm*}{Theorem}

\newtheorem*{exthm*}{Expected Theorem}

\newtheorem{lem}[subsection]{Lemma}

\newtheorem*{lem*}{Lemma}

\newtheorem{prop}[subsection]{Proposition}

\newtheorem*{prop*}{Proposition}

\newtheorem{cor}[subsection]{Corollary}

\newtheorem*{cor*}{Corollary}

\newtheorem*{claim*}{Claim}

\newtheorem*{conj*}{Conjecture}

\theoremstyle{definition}
\newtheorem{defn}[subsection]{Definition}
\newtheorem*{defn*}{Definition}

\newtheorem{notation}[subsection]{Notation and Conventions}

\newtheorem{constr}[subsection]{Construction}

\newtheorem{exa}[subsection]{Example}

\newtheorem*{exa*}{Example}

\theoremstyle{remark}
\newtheorem{rmk}[subsection]{Remark}

\newtheorem*{rmk*}{Remark}

\newtheorem{ass}[subsection]{Assumption}

\newtheorem*{ass*}{Assumption}

\numberwithin{figure}{subsection}
\numberwithin{table}{subsection}

\newcounter{listnum}

\newcounter{asslistcounter}

\newcounter{subenvcounter}

\DeclareMathOperator{\Adm}{Adm}

\DeclareMathOperator{\End}{End}

\DeclareMathOperator{\Frac}{Frac}
\DeclareMathOperator{\Frob}{Frob}
\DeclareMathOperator{\id}{id}

\DeclareMathOperator{\Int}{Int}
\DeclareMathOperator{\inv}{inv}

\DeclareMathOperator{\GL}{GL}
\DeclareMathOperator{\Gr}{Gr}

\DeclareMathOperator{\Isom}{Isom}

\DeclareMathOperator{\pr}{pr}

\newcommand{\Sh}{\mathsf{Sh}}
\DeclareMathOperator{\SL}{SL}

\DeclareMathOperator{\Spec}{Spec}
\DeclareMathOperator{\Spf}{Spf}

\DeclareMathOperator{\Res}{Res}

\DeclareMathOperator{\rank}{rk}

\def \AA {\mathbb{A}}

\newcommand{\CC}{\mathbb{C}}

\def \FF {\mathbb{F}}
\def \GG {\mathbb{G}}

\def \PP {\mathbb{P}}
\def \QQ {\mathbb{Q}}

\def \ZZ {\mathbb{Z}}

\def \Bcal {\mathcal{B}}

\def \Dcal {\mathcal{D}}
\def \Ecal {\mathcal{E}}
\def \Fcal {\mathcal{F}}
\def \Gcal {\mathcal{G}}

\def \Lcal {\mathcal{L}}

\def \Ocal {\mathcal{O}}
\def \Pcal {\mathcal{P}}

\def \Scal {\mathcal{S}}
\def \Tcal {\mathcal{T}}

\def \Vcal {\mathcal{V}}

\def \mfr {\mathfrak{m}}

\def \Sscr {\mathscr{S}}

\def \hbar {\bar{h}}

\def \bbf {\mathbf{b}}

\def \wbf {\mathbf{w}}

\def \ybf {\mathbf{y}}

\def \Ksf  {\mathsf{K}}

\newcommand{\pot}[1]{ [\hspace{-0,17em}[ {#1} ]\hspace{-0,17em}] }
\newcommand{\rpot}[1]{ (\hspace{-0,23em}( {#1} )\hspace{-0,23em}) }

\newcommand{\restr}[2]{{#1}\raise-.5ex\hbox{\ensuremath|}_{#2}}

\def \dom {{\rm dom}}
\def \univ {{\rm univ}}

\newcommand{\cl}[1]{\mkern 1.5mu\overline{\mkern-1.5mu#1\mkern-1.5mu}\mkern 1.5mu}
\newcommand{\scl}[1]{{#1}^s} %Algebraic and separable closure

\newcommand{\red}{\mathrm{red}}

\newcommand{\riso}{\xrightarrow{\sim}}

\DeclareMathOperator{\Sch}{Sch}

\DeclareMathOperator{\coker}{coker}

\DeclareMathOperator{\Sht}{Sht}

\usepackage{bbold}

\usepackage[british]{babel}
\usepackage{xcolor}
\hypersetup{%
  colorlinks=true,
}

\newcommand{\blambda}{\boldsymbol\lambda}

\usepackage{manfnt}

\DeclareMathOperator{\bsc}{bsc}

\DeclareMathOperator{\Bun}{Bun}

\DeclareMathOperator{\Hck}{Hck}
\DeclareMathOperator{\op}{op}

\DeclareMathOperator{\Pic}{Pic}
\DeclareMathOperator{\lcm}{lcm}
\DeclareMathOperator{\Ram}{Ram}

\usepackage{leftindex}

\newcommand{\ltau}{{}^{\tau}}
\newcommand{\nc}{\mathrm{nc}}
\newcommand{\incl}{\mathrm{incl}}
\newcommand{\str}{\mathrm{str}}

\newcommand{\Eucal}{\underline{\mathcal{E}}}
\newcommand{\Fucal}{\underline{\mathcal{F}}}
\newcommand{\Pucal}{\underline{\mathcal{P}}}

\author[Y.-G.~Choi]{Yong-Gyu Choi}
\address{Department of Mathematical Sciences, KAIST
291 Daehak-ro, Yuseong-gu, Daejeon, 34141, Republic of Korea}
\email{yonggyuchoimath@gmail.com}

\author[W.~Kim]{Wansu Kim}
\address{Department of Mathematical Sciences, KAIST
291 Daehak-ro, Yuseong-gu, Daejeon, 34141, Republic of Korea}
\email{wansu.math@kaist.ac.kr}

\author[J.~Park]{Junyeong Park}
\address{Department of Mathematics Education, Chonnam National University, 77, Yongbong-ro, Buk-gu, Gwangju 61186, Republic of Korea}
\email{junyeongp@gmail.com}

\title{On properness of moduli stacks of $\mathcal{D}^{\times}$-shtukas over ramified legs}

\begin{document}
\begin{abstract}
Given a maximal order $\Dcal$ of a central division algebra $D$ over a global function field $F$, we prove an explicit necessary and sufficient condition for moduli stacks of $\Dcal^\times$-shtukas to be proper over a finite field (modulo a suitable central action) in terms of the \emph{local invariants} of $D$ and \emph{bounds}. Our result extends E.~Lau's criterion (Duke Math. J. \textbf{140} (2007)) for properness of the \emph{leg morphism} (or \emph{characteristic morphism}) away from the ramification locus of $D$ to allow legs at ramified places.
We also establish non-emptiness of Newton and Kottwitz--Rapoport strata for moduli stacks of $\Bcal^\times$-shtukas, where $\Bcal$ is a maximal order of a central simple algebra over $F$.
\end{abstract}

\maketitle

\tableofcontents

\section{Introduction}\label{sec-intro}
% Wansu to write later
A Shimura variety $\Sh_K(G,\{h\})_\CC$ associated to a Shimura datum $(G,\{h\})$ is compact if and only if $G$ does not admit any proper parabolic subgroups defined over $\QQ$ (i.e., $G$ is totally anisotropic mod centre). Moreover, if a compact Shimura variety $\Sh_K(G,\{h\})_\CC$ admits a \emph{``canonical'' integral model} $\Sscr_{\Ksf}(G,\{h\})$ over a finite extension of $\ZZ_{(p)}$ (typically under the assumption that the level $\Ksf$ is parahoric at $p$), then we expect $\Sscr_{\Ksf}(G,\{h\})$ to be proper as well. Note that it is extremely delicate to even \emph{define} the ``canonicity'' of integral models (\emph{cf.} \cite{Milne:ModpPointsGoodRed}, \cite{Pappas:Canonical}, \cite{PappasRapoport:ShtukaCanonical}), not to mention the difficulty involved in actually constructing such a model. Nonetheless, in many cases where there is a construction of integral models that are expected to be canonical, the theory of arithmetic compactification confirms that integral models of compact Shimura varieties are indeed proper; see \cite[Theorem~5]{MadapusiPera:Compactification} and the references therein.

The moduli stacks of ``shtukas'' are often introduced as function field analogues of Shimura varieties. To explain the analogy, we fix a global function field $F$ of a smooth projective geometrically integral curve $X$ over a finite field $\FF_q$ with $q$ elements. Let $G$ be a connected reductive group over $F$, and $\Gcal$ a \emph{parahoric Bruhat--Tits integral model} of $G$ over $X$. Given a finite set $I$ parametrising \emph{legs}, a partition $I_\bullet$ of $I$, and additional data $\blambda$ specifying \emph{bounds}, one obtains a moduli stack $\Sht^{\leqslant\blambda}_{\Gcal,I_\bullet}$ of $\Gcal$-shtukas, which is a separated Deligne--Mumford stack locally of finite type over $X^I$. 

If $\Gcal$ is a split reductive group, then the construction of $\Sht^{\leqslant\blambda}_{\Gcal,I_\bullet}$ is due to Varshavsky \cite{Varshavsky:shtuka}, building upon earlier work of Drinfeld. When $\Gcal$ is not necessarily reductive over $X$, the construction is due to Arasteh~Rad and Hartl \cite{ArastehRad-Hartl:LocGlShtuka, ArastehRad-Hartl:Uniformizing}, who considered a much broader class of integral models $\Gcal$ than those treated here. This construction was further refined by Bieker \cite{Bieker:IntModels}. It should be noted that the foundation becomes substantially more elaborate if we allow \emph{any} points as legs (instead of avoiding the bad places where $\Gcal$ is not reductive). Loosely speaking, the moduli stacks of $\Gcal$-shtukas with legs at places where $\Gcal$ is non-reductive (or ramified) are analogous to bad reduction of Shimura varieties with parahoric level structure. 
 
Although the analogy between Shimura varieties and moduli stacks of $\Gcal$-shtukas has been extremely successful, the latter exhibit certain ``pathologies'' with no analogue on the Shimura variety side. For instance, $\Sht^{\leqslant\blambda}_{\Gcal,I_\bullet}$ is typically non quasi-compact. Even in the classical setting with $\Gcal = \GL_d$ considered by Drinfeld and Lafforgue \cite{LafforgueL:GlobalLanglands}, each irreducible component of  $\Sht^{\leqslant\blambda}_{\Gcal,I_\bullet}$ is non quasi-compact.

That said, the aforementioned pathology can be avoided by working with certain ``inner forms'' of $\GL_d$ -- an idea employed, for example, in \cite{LaumonRapoportStuhler}, \cite{LafforgueL:Ramanujan} and \cite{NgoBC:DSht}. To explain, let $D$ be a central division algebra over $F$ with dimension $d^2$, and choose a \emph{maximal} $\Ocal_X$-order $\Dcal$ of $D$ (for simplicity). The group $\Gcal\coloneqq \Dcal^\times$ is a parahoric Bruhat--Tits integral model of $G\coloneqq D^\times$, which is reductive away from $\Ram(D)$, the set of places where $D$ ramifies. Since the reductive $F$-group $D^\times$ is totally anisotropic mod centre, one might expect that the moduli stack of $\Dcal^\times$-shtukas should have some ``properness'' property via the \emph{naive} analogy with Shimura varieties. Surprisingly, the following result was obtained by E.~Lau:
\begin{thm}[{\cite[Theorem~A]{Lau:Degeneration}}]\label{Lau-thm}
    Let $I$ be a non-empty finite set indexing legs, and choose a partition $I_\bullet$ of $I$. We fix $\blambda = (\lambda_i)_{i\in I}$ where $\lambda_i = (\lambda_{i,1},\cdots,\lambda_{i,d})\in \ZZ^d$ such that $\lambda_{i,j}\geq\lambda_{i,j+1}$ for any $i,j$, and we have $\sum_{i\in I}\sum_{j=1}^d\lambda_{i,j}=0$. Set $U\coloneqq X\setminus \Ram(D)$.

    Then $\left.\big(\Sht^{\leqslant\blambda}_{\Dcal^\times,I_\bullet}/a^\ZZ\big)\right|_{U^I}$ is proper over $U^I$ for any id\`ele $a\in\AA^\times_F$ of positive degree \emph{if and only if} the following inequality holds 
\begin{equation}\label{eq:Lau}
    \sum_{y\in \Ram(D)} [m\inv_y(D)]_\QQ > \sum_{i\in I}\sum_{j=1}^{d-m}\lambda_{i,j}
\end{equation}
for any integer $0<m<d$, where $[-]_\QQ\colon\QQ/\ZZ\to\QQ\cap [0,1)$ is a set-theoretic section.
\end{thm}
This result asserts that the moduli stack  $\left.\big(\Sht^{\leqslant\blambda}_{\Dcal^\times,I_\bullet}/a^\ZZ\big)\right|_{U^I}$ may \emph{not} be proper over $U^I$, and the properness follows if $D$ is ``sufficiently ramified'' relative to $\blambda$.

In this paper, we show that the same numerical criterion characterises properness of the entire stack $\Sht^{\leqslant\blambda}_{\Dcal^\times,I_\bullet}/a^\ZZ$ over $X^I$, without restricting to $U^I$:
\begin{thm}[Theorem~\ref{thm:from-inequality-to-properness}]\label{main-thm}
In the same setting as in Theorem~\ref{Lau-thm}, the stack $\Sht^{\leqslant\blambda}_{\Dcal^\times,I_\bullet}/a^\ZZ$ is proper over $X^I$ for any id\`ele $a\in\AA^\times_F$ of positive degree \emph{if and only if} the following inequality holds
        \begin{equation}\label{eq:main-thm}
            \sum_{y\in \Ram(D)} [m\inv_y(D)]_\QQ > \sum_{i\in I}\sum_{j=1}^{d-m}\lambda_{i,j}
        \end{equation}
for any integer $0<m<d$.
\end{thm}

Let us explain why Theorem~\ref{main-thm} (as well as E.~Lau's theorem) is non-trivial. Although $D$ is simple as a (right) $D$-module, $D\otimes_{\FF_q}\cl\FF_q$ admits many proper submodules where $\FF_q$ is the constant field of $F$. Equivalently, although $G\coloneqq D^\times$ has no proper $F$-rational parabolic subgroup, $G$ admits many proper parabolic subgroups defined over $F\cdot\cl\FF_q$. Such proper non-rational parabolic subgroups may a priori contribute to \emph{non-trivial degenerations}. Therefore, to show properness, we show that the inequality \eqref{eq:main-thm} contradicts the existence of such a degeneration. This idea is already present in E.~Lau's proof of the sufficient condition for properness over $U^I$ \cite[Proposition~3.2]{Lau:Degeneration}, and we refine its proof to allow the legs to meet the ramification locus $\Ram(D)$.

Thus Lau's numerical criterion \eqref{eq:Lau} also characterises properness over the whole of $X^I$, even when the legs are allowed to meet the ramification locus $\Ram(D)$. In particular, $\Sht^{\leqslant\blambda}_{\Dcal^\times,I_\bullet}/a^\ZZ$ is proper over $X^I$ if and only if its restriction to $U^I$ is proper over $U^I$. For example, if $d=|I|=2$ and $\blambda = \big((1,0),(0,-1)\big)$, then $\Sht^{\leqslant\blambda}_{\Dcal^\times,I_\bullet}/a^\ZZ$ is proper over $X^2$ if and only if $|\Ram(D)|\geq4$.

Our main result can also be applied to establish a sufficient condition for the properness of moduli stacks of bounded $\Gcal$-shtukas, where $\Gcal$ is a parahoric Bruhat--Tits integral model over $X$ of a totally anisotropic unitary group in odd characteristic. The full details will appear in forthcoming work. We hope that these cases will have useful applications, for example via Mantovan's formula or in the context of arithmetic intersection theory of special cycles.

%Our main result yields many examples of proper moduli stacks. For instance, it provides a sufficient condition for moduli stacks of $\Gcal$-shtukas where $\Gcal$ is a parahoric Bruhat--Tits integral model for a totally anisotropic unitary group in odd characteristics. Also for a parahoric Bruhat--Tits integral model $\Gcal$ of a form of $\GL_d$ that is totally anisotropic mod centre and ``sufficiently ramified'', we can deduce the properness of the moduli stacks of $2$-legged $\Gcal$-shtukas with ``small bound'' at one fixed leg in the sense of \cite[Definition~2.6.11]{HartlXu:UniformizingII}. We hope that these cases will have useful applications, for example via Mantovan's formula or in the context of arithmetic intersection theory of certain cycles.

Along the way, we establish Mazur's inequality for ``local $D_x$-shtukas'' (Proposition~\ref{prop:Adm}), and thereby obtain the \emph{non-emptiness} of Newton and Kottwitz--Rapoport strata in $\Sht^{\leqslant\blambda}_{\Dcal^\times,I_\bullet}$ (\emph{cf.} Corollaries~\ref{cor:nonempty-KR}, \ref{cor:nonempty-Newton}), which is of independent interest. This in particular shows that the fibre of $\Sht^{\leqslant\blambda}_{\Dcal^\times,I_\bullet}$ at any closed point of $X^I$ is \emph{non-empty} and exhibits features analogous to the mod~$p$ reduction of Shimura varieties. Although these results are not logically required for the proof of Theorem~\ref{main-thm}, they are conceptually linked: the key step in the proof of Theorem~\ref{main-thm} can be viewed as understanding the degeneration of Mazur's inequality along the boundary. (See Remark~\ref{rmk:stupid-sanity} for further discussions.)

After setting up the general notation in \S\ref{sec-notation}, we review the basic definitions of moduli stacks of $\Gcal$-shtukas following \cite{ArastehRad-Hartl:LocGlShtuka}, \cite{ArastehRad-Hartl:Uniformizing} and \cite{Bieker:IntModels}. We then digress on Kottwitz--Rapoport and Newton stratifications and establish non-emptiness results in \S\ref{sec-stratifications}. In \S\ref{sec-isosht}, we review the classification of $(D,\varphi)$-spaces and Dieudonn\'e $D_x$-modules, and deduce the key computation for the properness result (\emph{cf.} Lemma~\ref{lemma-Lau-bound-for-degree-of-D-varphi-space-outside-split-places}), and in \S\ref{sec-valcrit} we finally prove Theorem~\ref{main-thm}.

\subsection*{Regarding the previous version and the use of AI}
The first draft of this manuscript was completed in May 2025 without the help of AI, and has undergone minor revisions since then. The main result of the previous version was the following \emph{sufficient condition} for properness.
\begin{thm}\label{main-thm-previous} 
 Let $I^{\nc}\coloneqq \{i\in I\mid\lambda_i\text{ is not central}\}$, and assume that $|\Ram(D)|>|I^{\nc}|$.
  Suppose that for each $0< m <d$ and for any subset $Y\subset \Ram(D)$ with size $|\Ram(D)|-|I^{\nc}|$, we have 
  \begin{equation}  \label{eq:inequality-for-properness-previous}  
      \sum_{x\in Y} [m \inv_x (D)]_{\QQ}  > \sum_{i\in I}\sum_{j=1}^{d-m} \lambda_{i,j} .
  \end{equation}
Then, the stack $\Sht_{\Dcal^\times,I_\bullet}^{\leqslant\blambda} / a^{\ZZ}$ is proper over $\cl\FF_q$ for any $a\in \AA_F^{\times}$ of positive degree.
\end{thm}
Here, we say that $\lambda_i$ (and the corresponding leg $x_i$) is \emph{central} if $\lambda_{i,1}=\lambda_{i,d}$. This theorem remains valid, as it is implied by the main theorem of the current version---namely, Theorem~\ref{main-thm}---since the left side of inequality~\eqref{eq:inequality-for-properness-previous} is at most the left side of inequality~\eqref{eq:main-thm}. However, the sufficient condition in the above theorem is more stringent, as inequality~\eqref{eq:inequality-for-properness-previous} depends on the number of ramified places that can be hit by non-central legs. Namely, the left side of~\eqref{eq:inequality-for-properness-previous} omits the contribution of $|I^{\nc}|$ ramified places, which accounts for the worst case where all the non-central legs hit pairwise distinct ramified places.

Until recently, we were confident that the condition in Theorem~\ref{main-thm-previous} would also be \emph{necessary} for the properness of the stack $\Sht_{\Dcal^\times,I_\bullet}^{\leqslant\blambda} / a^{\ZZ}$. This expectation was challenged by Claude (Fable 5.1, a large language model developed by Anthropic) in October 2026, when the first named author was using Claude to work on a follow-up project on the properness of moduli stacks of hermitian shtukas. More precisely, for moduli stacks of hermitian shtukas, Claude first suggested a sufficient condition for properness that does \emph{not} depend on the number of ramified places hit by non-central legs. When pushed further, Claude suggested the sufficient condition for $\Dcal^\times$-shtukas stated in this version of the manuscript, which is precisely Lau's inequality~\eqref{eq:Lau}. Since Lau's inequality is also a necessary condition by \cite[Theorem~A]{Lau:Degeneration}, this yields the ``if and only if'' criterion in Theorem~\ref{main-thm}. In particular, there exists no non-proper moduli stack $\Sht_{\Dcal^\times,I_\bullet}^{\leqslant\blambda} / a^{\ZZ}$ whose restriction to $U^I$ is proper over $U^I$, contrary to the expectation stated in the previous version.

The \emph{upgrade} suggested by Claude is completely localised in the proof of Theorem~\ref{thm:from-inequality-to-properness}---the very final step of the proof of Theorem~\ref{main-thm}. Apart from this, all statements and proofs are identical to the previous version, which was purely human work, up to minor simplifications in \S\ref{sec-valcrit} omitting some estimates that are no longer needed. We also included in Example~\ref{exa:ramified-legs-quaternion} an explanation, originally given by Claude, of why the argument in the previous version (recalled in Remark~\ref{rmk:previous-properness-criterion}) does not yield an optimal condition when $D$ is a quaternion division algebra and $\blambda=\big((1,0),(0,-1)\big)$. This explanation corrected one of the misconceptions that led us to believe that inequality~\eqref{eq:inequality-for-properness-previous} was optimal.

We have checked and fully understood the arguments suggested by Claude; indeed, the revised proof of Theorem~\ref{thm:from-inequality-to-properness} is conceptually clear. All the discussions added in this version, including the proof of Theorem~\ref{thm:from-inequality-to-properness}, Remark~\ref{rmk:previous-properness-criterion} and Example~\ref{exa:ramified-legs-quaternion}, were written by the authors, and we take full responsibility for the content of this paper.

\section{Notations and Preliminaries}\label{sec-notation}
\subsection{General conventions}\label{ssec:general-convention}

For any field $k$ we let $\cl k$ denote its algebraic closure and $\scl{k}$ its separable closure.

Throughout the paper, we fix a finite field $\FF_q$ of characteristic $p$, and its algebraic closure $\cl\FF_q$. All the rings, schemes, algebraic stacks and morphisms between them in this paper are assumed to be over $\Spec(\FF_q)$ unless stated otherwise. We let $-\otimes-$ denote the tensor product over $\FF_q$ if there is no risk of confusion. For any schemes (or algebraic stacks) $S$ and $T$ over $\FF_q$, we let $S\times T$ denote the fibre product over $\Spec(\FF_q)$. If $T = \Spec R$, then we may write $S\times R$ instead of $S\times\Spec R$.   

We fix a smooth projective geometrically connected curve $X$ over $\Spec(\FF_q)$, and let $\breve X\coloneqq X\times \Spec\cl\FF_q$. For any (open or closed) subscheme $T\subset X$, we let $\breve T\subset \breve X$ denote the preimage of $T$. 
Let $F$ be the function field of $X$, and $\breve F$ the function field of $\breve X$. Note that $\breve F \cong F\otimes\cl\FF_q$. Under this isomorphism, we write $\tau\colon \breve F\to \breve F$ for the map $\id_F\otimes \Frob_q\colon f\otimes a\mapsto f\otimes a^q$.

For any finite extension $E/F$, let $\widetilde X_E$ denote the normalisation of $X$ in $E$, which is a finite flat cover of $X$. 

For any place $x$ of $F$ (i.e., a closed point of $X$), we write $\Ocal_x\coloneqq\widehat\Ocal_{X,x}$ and $F_x\coloneqq\Frac(\Ocal_x)$. A chosen uniformiser of $\Ocal_x$ will be denoted by $\varpi_x$, and we write $\kappa_x\coloneqq\Ocal_x/(\varpi_x)$, so that $\Ocal_x = \kappa_x \pot {\varpi_x}$ and $F_x = \kappa_x \rpot {\varpi_x}$. Let $\breve\Ocal_x$ and $\breve F_x$ denote the completed maximal unramified extensions of $\Ocal_x$ and $F_x$, respectively, so that $\breve\Ocal_x = \cl\kappa_x \pot {\varpi_x}$ and $\breve F_x = \cl\kappa_x \rpot {\varpi_x}$. %Under this isomorphism, we write $\tau_x \colon \breve F_x\to\breve F_x$ for the map $\id_{F_x}\widehat\otimes \Frob_{q}^{\deg(x)}\colon f\otimes a\mapsto f\otimes a^{q^{\deg(x)}}$.

For any $\cl x\in X(\cl\FF_q)$, we write $\Ocal_{\cl x}\coloneqq\widehat\Ocal_{\breve X,\cl x}$ and $F_{\cl x}\coloneqq \Frac(\Ocal_{\cl x})$. If $\cl x$ lies over $x\in X$ then we have $\breve\Ocal_x\cong \Ocal_{\cl x}$ and $\breve F_x\cong F_{\cl x}$.

For an $\FF_q$-scheme $S$, we set
\begin{equation}\label{eq:tau}
   \tau\coloneqq \id_X\times\Frob_{q,S}\colon X\times S\to X\times S,
\end{equation}
where $\Frob_{q,S}\colon S\to S$ denotes the absolute $q$-Frobenius.
Similarly, for a place $x$ of $F$ and a $\kappa_x$-scheme $S$, we define
\begin{equation}\label{eq:tau-x}
   \tau_x\coloneqq \id\times\Frob_{|\kappa_x|,S}\colon \Spf\Ocal_x \times_{\kappa_x} S \to \Spf\Ocal_x \times_{\kappa_x} S.
\end{equation}
In particular, $\tau_x$ is the restriction of $\tau^{[\kappa_x:\FF_q]}\colon X\times S \to X\times S$ under the following map
\[
\Spf\Ocal_x \times_{\kappa_x} S \hookrightarrow \Spf\Ocal_x \times_{\FF_q} S \to X\times_{\FF_q} S \eqqcolon X\times S.
\]

% \noindent\wk{[Sorry, I changed my mind to use $\kappa_x$ instead of $\kappa(x)$ for the residue field...]}

\subsection{Linear algebraic groups and loop groups}\label{ssec:groups}
Throughout the paper, $\Gcal$ denotes a \emph{parahoric Bruhat--Tits group scheme}; i.e., a smooth affine group scheme over $X$ with \emph{reductive} generic fibre denoted by $G$ over $F$, such that its base change over $\Spec(\Ocal_x)$ is a \emph{parahoric group scheme} in the sense of Bruhat--Tits (\cite{BruhatTits:RedGp1}, \cite{BruhatTits:RedGp2}) for any closed point $x\in|X|$. (In particular, each fibre of $\Gcal$ is connected.)

For any scheme $T$ over $X$, we write $\Gcal_T\coloneqq \Gcal\times_X T$. If $T=\Spec R$ then we write $\Gcal_R\coloneqq\Gcal_{\Spec R}$. We similarly define $G_T$ and $G_R$ for any $F$-scheme $T$ and $F$-algebra $R$, respectively.

For a closed point $x\in|X|$, we write $\Gcal_x\coloneqq \Gcal_{\Ocal_x}$ and $G_x\coloneqq G_{F_x}$. For $R\in \mathrm{CRing}_{\kappa_x}$, we regard the rings $R\pot z$ and $R\rpot z$ as $\Ocal_x$-algebras via the $\kappa_x$-algebra maps $\varpi_x \mapsto z$ where $z$ denotes a formal variable, and we define 
\begin{align}
  L^+_x\Gcal \coloneqq L^+\Gcal_{x}:\mathrm{CRing}_{\kappa_x}& \to \mathrm{Grp} \\
  R & \mapsto \Gcal_{\Ocal_x}(R[\![ z ]\!] ), \quad\text{and}\notag\\
  L_xG\coloneqq L G_{F_x} : \mathrm{CRing}_{\kappa_x}& \to \mathrm{Grp} \\
  R & \mapsto G_{x}(R(\!( z )\!) ),\notag
\end{align}
where $\Gcal_{\Ocal_x}(R \pot{z})$ and $G_x(R \rpot{z})$ denote the groups of $R\pot{z}$- and $R\rpot{z}$-points, respectively, with respect to the $\Ocal_x$-algebra structure specified above. We call them the \emph{positive loop group} of $\Gcal$ at $x$ and the \emph{loop group} of $G$ at $x$, respectively. Note that $L^+_x\Gcal$ is a pro-algebraic group over $\kappa_x$ and $L_xG$ is an ind-group scheme over $\kappa_x$. 

We similarly define the \emph{generic positive loop group} $L^+_\eta G$ and \emph{generic loop group} $L_\eta G$ to be the positive loop group and the loop group associated to $G$ over $F$. (We use the subscript $\eta$ to emphasise the different nature of $L_\eta G$ and $L^+_\eta G$ from $L_xG$  and $L^+_x\Gcal$ for $x\in |X|$.) 
%% [Explaining the generic affine grassmannian, but maybe copy to the next section]
%We call $Gr_{G,\eta}\coloneqq L_\eta G/L^+_\eta G$ the \emph{generic affine grassmannian} of $G$, which is representable by an ind-scheme ind-projective over $F$.

\subsection{$\Gcal$-bundles}\label{ssect:G-Bun}
All torsors will be \emph{right} torsors. 
By a \emph{$\Gcal$-bundle $\Pcal$} over $X\times S$, we always mean a (right) $\Gcal$-torsor. %Similarly, all torsors for the positive loop groups and loop groups will be \emph{right} torsors. 

Let $S$ be a $\kappa_x$-scheme for some $x\in |X|$. For any $\Gcal$-bundle $\Pcal$ over $X\times S$, we define its \emph{localisation} at $x$, denoted by $\Lcal^+_x\Pcal$, to be the $L^+_x\Gcal$-torsor over $S$ obtained from the ``completion of $\Pcal$ at $x$''. See \cite[\S5.2]{ArastehRad-Hartl:LocGlShtuka} or \cite[2.3.4, 2.3.8(b)]{HartlXu:UniformizingII} for details. We also attach to $\Pcal$ an $L_xG$-torsor $\Lcal_x\Pcal$ over $S$ by setting \[\Lcal_x\Pcal\coloneqq\Lcal^+_x\Pcal\times^{L^+_x\Gcal}L_xG.\]

%By \emph{stack} $\Mcal$ over a scheme $S$, we always mean a stack fibred in groupoids over the big \emph{fppf} site of $S$. Let $\Bun_\Gcal$ denote the stack over $\Spec(\FF_q)$ whose $S$-points correspond to right $\Gcal$-bundles over $X\times S$. It is known that $\Bun_\Gcal$ is a smooth locally finite-type algebraic stack over $\FF_q$ as $\Gcal$ is assumed to be smooth and affine  over $X$; \emph{cf.} \cite[Proposition~1]{Heinloth:Uniformization}.

\subsection{Central simple algebras}\label{ssec-CSA}
Let $B$ be a central simple algebra over $F$, and choose an order $\Bcal$ over $X$; i.e., a locally free $\Ocal_X$-algebra with generic fibre $B$. By abuse of notation, we let $\Bcal^\times$ also denote the smooth affine group scheme over $X$ representing the following contravariant functor on $\Sch_{/X}$:
\begin{equation}
    T\rightsquigarrow \Gamma(T, (\Bcal_T)^\times),
\end{equation}
where $\Bcal_T$ is the pullback of $\Bcal$. From now on, we suppose that $\Bcal$ is \emph{hereditary}; this is equivalent to requiring that $\Bcal^\times$ is parahoric everywhere (\emph{cf.} \cite[Theorems~(39.14), (40.5)]{Reiner:MaxOrders}). Recall that any maximal order is hereditary (\emph{cf.} \cite[Theorem~(21.4)]{Reiner:MaxOrders}).

For any scheme $T$ over $X$, a \emph{$\Bcal_T$-module} is  a quasi-coherent sheaf on $T$ with a \emph{right} $\Bcal_T$-scalar multiplication (in the obvious sense). We say a $\Bcal_T$-module $\Ecal$ is \emph{locally free of rank~$r$ (as a $\Bcal_T$-module)} if there is a Zariski covering $\{U_\alpha\}$ of $T$ such that $\Ecal|_{U_\alpha}\cong \Bcal_{U_\alpha}^r$. Note that we have natural quasi-inverse equivalences of categories between rank-$1$ locally free $\Bcal_T$-modules $\Ecal$ and $\Bcal^\times$-torsors $\Pcal$ over $T$ as follows:
\begin{equation}\label{eq:torsor-module}
    \Ecal \rightsquigarrow  \Pcal\coloneqq\Isom_{\Bcal_T}(\Bcal_T,\Ecal) \quad \&\quad
    \Pcal\rightsquigarrow \Ecal\coloneqq \Pcal \times^{\Bcal^\times}\Bcal.
\end{equation}

For any $\FF_q$-scheme $S$, let $\ltau(-)$ denote the pullback by $\tau$ for a $\Gcal$-bundle or a $\Bcal_{X\times S}$-module, where $\tau$ is defined in \eqref{eq:tau}.
% For any $\FF_q$-scheme $S$, we set 
% \begin{equation}
% \tau\coloneqq \id_X\times\Frob_q\colon X\times S\to X\times S,
% \end{equation}
% and let $\ltau(-)$ denote the pullback by $\tau$ for a $\Gcal$-bundle or a $\Bcal_{X\times S}$-module. This notation of $\tau$ is compatible with $\tau:\breve F \to \breve F$ introduced in \S\ref{ssec:general-convention}.

\section{Moduli stacks of $\Gcal$-shtukas}\label{sec-shtuka}
\label{section-review-D-shtukas}
% !TEX root = ./degeneration.tex
%
%
%Let $\Gcal$ be a smooth affine group scheme over $X$ with connected fibres, and we suppose its generic fibre $G$ is a reductive group over $F$. Let $U\subset X$ denote a dense open subscheme over which $\Gcal$ is reductive. Built upon the work of Drinfeld,  Varshavsky \cite{Varshavsky:shtuka} defined $\Gcal$-shtukas (with arbitrary number of legs and arbitrary modifications) when $\Gcal$ is a split reductive group, and the definition can easily extend to general $\Gcal$ if we require that all legs should be contained in $U$. 
%
%If $\Gcal$ is not reductive everywhere, then the notion of $\Gcal$-shtukas has been generalised to allow \emph{arbitrary} legs (not just in $U$) by Arasteh~Rad and Hartl \cite{ArastehRad-Hartl:LocGlShtuka, ArastehRad-Hartl:Uniformizing}, and refined by Bieker \cite{Bieker:IntModels}. Let us review the definition under the additional running (and simplifying) assumption that $\Gcal$ is \emph{parahoric} everywhere. 

%Let $I$ be a non-empty finite set, and choose a partition $I_\bullet\coloneqq(I_1,\cdots, I_m)$ of $I$. Let $\Bun_\Gcal$ be a stack over $\FF_q$ whose $S$-points parametrise $\Gcal$-bundles over $X\times S$, which is known to be a smooth algebraic stack locally of finite type over $\FF_q$; \emph{cf.} \cite[Proposition~1]{Heinloth:Uniformization}.

In this section, we review the basic definitions and properties of the moduli stacks of $\Gcal$-shtukas following \cite{ArastehRad-Hartl:LocGlShtuka, ArastehRad-Hartl:Uniformizing}, \cite{Bieker:IntModels} and \cite{HartlXu:UniformizingII}, under the running assumption that $\Gcal$ is \emph{parahoric} everywhere. 
We also obtain a quasi-compactness result on moduli stacks of $\Dcal^\times$-shtukas; \emph{cf.} Proposition~\ref{prop:quasi-compactness-of-Sht-D-mod-a}.

We consider the moduli stack $\Bun_{\Gcal}$ of $\Gcal$-bundles on $X$, the stack over $\FF_q$ whose $S$-points parametrise $\Gcal$-bundles over $X\times S$. It is a smooth algebraic stack locally of finite type over $\FF_q$ (\emph{cf.} \cite[Theorem~2.5]{ArastehRad-Hartl:Uniformizing}). 
% \wk{(I think \cite[Theorem~2.5]{ArastehRad-Hartl:Uniformizing} is a generalisation for \emph{flat} affine $\Gscr$, not necessarily smooth.)} \yg{[I think you are correct so I changed the reference, but I have a simple question: is $\Bun_{\Gcal}$ separated (if not, quasi-separated)? I don't think it is separated, because $[\Spec(\ZZ)/\GG_{m,\Spec(\ZZ)}]$ is not separated so it is not sensible to think the Weil restriction of non-separated Artin stack is separated. But I can't find a reference of quasi-separatedness of $[*/\mathcal G]$ for `flat' affine $\Gcal$ of finite type, \cite[Theorem~2.5]{ArastehRad-Hartl:Uniformizing} doesn't say anything for separation (for smooth separated of finite presentation $\Gcal$, \cite[(4.6.1)]{Laumon-MoretBailly} is a reference as every algebraic stack in \cite{Laumon-MoretBailly} is quasi-separated, even though I don't understand the quasi-separatedness proof there). I know a non-quasi-separated classifying stack for non-quasi-compact $\Gcal$: $[\Spec(\ZZ) / \underline{\ZZ}]$. I at least know a reference that shows quasi-separatedness of $\Bun_{\Gcal}$ from that of $[\Spec(\FF_q)/\Gcal]$.]}

\begin{defn}
  \label{def:Hecke}
  Let $\Gcal$ be as above, and choose a non-empty finite set $I$ equipped with a partition $I_\bullet\coloneqq(I_1,\cdots, I_r)$. The \emph{(unbounded) Hecke stack}, denoted as $\Hck_{\Gcal,  I_{\bullet}}$, is defined to be the stack over $\FF_q$ whose $S$-points parametrise tuples
  \[
   \Pucal\coloneqq ( (x_i)_{i\in I}, (\Pcal_j)_{j=0,\cdots, r}, (\varphi_j)_{j=1,\cdots, r} )
  \]
  where
  \begin{itemize}
    \item $x_i \in X(S)$ for any $i\in I$,
    \item $\Pcal_j \in \Bun_{\Gcal}(S)$ for any $0\leq j\leq r$,
    \item $\varphi_j \colon \Pcal_{j-1} |_{ (X\times S) \smallsetminus \bigcup_{i\in I_j} \Gamma_{x_i} } \to  \Pcal_{j} |_{ (X\times S) \smallsetminus \bigcup_{i\in I_j} \Gamma_{x_i} }$ is an isomorphism of $\Gcal$-torsors on $(X\times S) \smallsetminus \bigcup_{i\in I_j} \Gamma_{x_i} $ for every integer $1\leq j \leq r$.
  \end{itemize} 
For each $j_0=0,\cdots,r$, we define a morphism $\pr_{j_0}$ as follows:
\[\pr_{j_0}\colon \Hck_{\Gcal, I_\bullet}\to \Bun_\Gcal; \quad \Pucal= ( (x_i), (\Pcal_j), (\varphi_j))\mapsto\Pcal_{j_0}.\]
We also define the \emph{leg morphism} as follows:
\[
\wp\colon \Hck_{\Gcal ,I_\bullet}\to X^I; \quad \Pucal= ( (x_i), (\Pcal_j), (\varphi_j))\mapsto (x_i)_{i\in I}.
\]
If $I$ is a singleton then we write $\Hck_{\Gcal,X}\coloneqq\Hck_{\Gcal, (I)}$.
\end{defn}

\begin{defn}\label{def:BD}
     For $\Gcal$, $I$, $I_\bullet$ as above, the \emph{Beilinson--Drinfeld affine grassmannian} (or the \emph{BD grassmannian}), denoted as $\Gr_{\Gcal,  I_\bullet}$, is the fibre of the trivial bundle with respect to $\pr_0\colon\Hck_{\Gcal, I_\bullet}\to\Bun_\Gcal$. More explicitly, $\Gr_{\Gcal,I_\bullet}(S)$ classifies $(\Pucal,\epsilon)$, where $\Pucal\in\Hck_{\Gcal,I_\bullet}(S)$ and $\epsilon\,\colon \Gcal_{X\times S}\riso\Pcal_0$ is a trivialisation of $\Gcal$-torsors.
     %defined to be the stack over $\FF_q$ whose $S$-points correspond to $S$-points of the Hecke stack $\Pucal$ together with a trivialisation $\epsilon \colon \Gcal_{X\times S}\riso\Pcal_0$. In other words, $\Gr_{\Gcal,  I_\bullet}$ is the fibre of the trivial bundle with respect to $\pr_0\colon \Hck_{\Gcal, I_\bullet}\to\Bun_\Gcal$. %\wk{[Switched to $\pr_0$ and redefined trivialisation as this is more natural in the later computation in admissible sets]}

    If $I$ is a singleton then we write  $\Gr_{\Gcal,X}\coloneqq\Gr_{\Gcal, (I)}$.
%Even in the special case when $\Gcal$ is the unit group of an order of a central division algebra over $F$, the definition of $\Gcal$-shtukas with arbitrary legs does not seem to simply in terms of linear algebra.
\end{defn}
As $\Gcal$ is assumed to be parahoric everywhere, the morphism 
\[(\wp, \pr_0)\,\colon 
\begin{tikzcd}
    \Hck_{\Gcal, I_{\bullet}} \arrow[r] & X^I\times\Bun_{\Gcal}
\end{tikzcd}\]
is ind-projective. So in particular $\Gr_{\Gcal,  I_\bullet}$ is representable by an ind-projective ind-scheme. (\emph{Cf.} \cite[Propositions~3.9,~3.12]{ArastehRad-Hartl:Uniformizing}, building upon \cite[Theorem~2.19]{Richarz:AffGr}.)

\begin{rmk}\label{rmk:BL}
    When $S = \Spec R$, we  give an alternative description of $\Gr_{\Gcal,I_\bullet}(R)$ as follows. To set up the notation, given $x\in X(R)$ we let $\widehat\Gamma_x$ denote the $\Spec$ of the formal completion of $\Ocal_{X\times S}$ at $\Gamma_x$. Similarly, given $(x_i)_{i\in I} \in X(R)^I$ we let $\widehat\Gamma_{(x_i)}$ denote the $\Spec$ of the formal completion of $\Ocal_{X\times S}$ at $\bigcup_{i\in I}\Gamma_{x_i}$. Note that $\Gamma_{x_i}$ can be viewed as a closed subscheme of $\widehat\Gamma_{(x_i)}$. 
    
    Then, by the Beauville--Laszlo descent lemma for $\Gcal$-torsors (\emph{cf.} \cite[Lemma~5.1]{ArastehRad-Hartl:LocGlShtuka}), the restriction of $\Gcal$-torsors on $X\times S$ to $\widehat\Gamma_{(x_i)}$ defines a natural isomorphism 
    \begin{equation}\label{eq:BD-BL}
    \begin{tikzcd}
        \Gr_{\Gcal,I_\bullet}(R) \arrow{r}{\cong} & \{( (x_i), (\widehat\Pcal_j), (\widehat\varphi_j), \widehat\epsilon\, )\},
    \end{tikzcd}
    \end{equation}
    where $\widehat\Pcal_j$ is a $\Gcal$-torsor over $\widehat\Gamma_{(x_i)}$,  and $\widehat\varphi_j\colon \widehat\Pcal_{j-1}|_{\widehat\Gamma_{(x_i)}\setminus\bigcup_{i\in I_j}\Gamma_{x_i}} \riso \widehat\Pcal_{j}|_{\widehat\Gamma_{(x_i)}\setminus\bigcup_{i\in I_j}\Gamma_{x_i}}$ and $\widehat\epsilon\,\colon \Gcal_{\widehat\Gamma_{(x_i)}}\riso\widehat\Pcal_0$ are isomorphisms of $\Gcal$-torsors. %In fact, the isomorphism \eqref{eq:BD-BL} is a consequence of the Beauville--Laszlo descent lemma for $\Gcal$-torsors (\emph{cf.} \cite[Lemma~5.1]{ArastehRad-Hartl:LocGlShtuka}). %(See \cite[Corollary~4.7]{HamacherKim:Igusa} for the full details; although \emph{loc.~cit.} was stated for the case when $I$ is a singleton and $x$ is supported on a single closed point of the curve $X$, its proof remains valid in this setting.)
    
    Recall that the group-valued contravariant functor on affine schemes over $X^I$
    \begin{equation}
        \Lcal^+_{X^I}\Gcal\colon
        \begin{tikzcd}
            (\Spec R , (x_i\in X(R))_{i\in I}) \arrow[r, rightsquigarrow] & \Gcal(\widehat\Gamma_{(x_i)})
        \end{tikzcd}
    \end{equation}
    can be represented by a pro-algebraic group that is \emph{flat} over $X^I$, called the \emph{positive global loop group}; \emph{cf.} \cite[Definition~2.1.3]{Bieker:IntModels} and \cite[Lemma~2.11]{Richarz:AffGr}. We have a left action of $\Lcal^+_{X^I}\Gcal$ on $\Gr_{\Gcal,I_\bullet}$ defined as follows: $g\in \Lcal^+_{X^I}\Gcal(R)$ modifies the trivialisation $\widehat\epsilon$ by precomposition with the left multiplication by $g^{-1}$.

    The description of $\Gr_{\Gcal,I_\bullet}$ and $\Lcal^+_{X^I}\Gcal$ can be simplified over the dense open subscheme $W_I\subset X^I$ of pairwise disjoint points $(x_i)_{i\in I}$. Indeed, from \eqref{eq:BD-BL} we get 
    \begin{equation}\label{eq:splitting-UI}
    \Gr_{\Gcal,  I_\bullet}|_{W_I} \riso (\prod_{i\in I}\Gr_{\Gcal,X})|_{W_I}\quad\text{and}\quad \Lcal^+_{X^I}\Gcal|_{W_I} \cong (\prod_{i\in I}\Lcal^+_X\Gcal)|_{W_I},
    \end{equation}
    where the products are taken as (ind-)schemes over $\Spec\FF_q$.     
\end{rmk}

Next, we review the notion of \emph{Beilinson--Drinfeld Schubert variety} (or \emph{BD Schubert variety}), indexed by $I$-tuples of dominant coweights. Let us fix a maximal $\scl F$-torus $T$ in $G_{\scl F}$ along with a Borel subgroup containing $T$,  which together determine the dominant coweights $X_\ast(T)_+$ associated with this choice. Recall that $X_\ast(T)_+$ is in natural bijection with the set of geometric conjugacy classes of $G$-valued cocharacters via the following construction
\[\xymatrix@1{
\lambda\in X_\ast(T)_+ \ar@{~>}[r]& \text{ the conjugacy class }\{\lambda\} \text{ of } \GG_m\xrightarrow{\lambda}T\hookrightarrow G_{\scl F}.
}\]
For any $\lambda\in X_\ast(T)_+$, let $E_\lambda\subset \scl F$ denote the field of definition of the geometric conjugacy class $\{\lambda\}$, called the \emph{reflex field} of $\lambda$. Note that $E_\lambda/F$ is finite as $E_\lambda$ is contained in any separable extension of $F$ splitting $G$. 

Now, observe that the generic fibre of $\wp\colon\Gr_{\Gcal,X}\to X$ is nothing but
\begin{equation}
    \Gr_{G,\eta}\coloneqq L_\eta G/L^+_\eta G,
\end{equation}
which we call the \emph{generic affine grassmannian}. It is an ind-projective scheme over $F$, and its $L^+_\eta G$-orbits over $\scl F$ are exactly \emph{affine Schubert cells} $\{\Scal^\lambda\}_{\lambda\in X_\ast(T)_+}$; namely, by identifying $L_\eta G(\scl F)=G(\scl F \rpot t )$ and $L^+_\eta G(\scl F) = G(\scl F\pot t)$ where $t$ denotes a formal variable, we can write
\begin{equation}\label{eq:t-lambda}
    \Scal^\lambda(\scl F) = G(\scl F\pot t) \cdot t^{-\lambda}\cdot  G(\scl F\pot t)/G(\scl F\pot t),
\end{equation}
where $t^{-\lambda}$ is the image of $t^{-1}\in \scl F\rpot t ^\times$ via $\lambda\colon \GG_m\to G_{\scl F}$. Here, the minus sign in $t^{-\lambda}$ is used so that the map $\varphi_{i}^{-1}$ is bounded by $\lambda_i$ in $\Gr_{\mathcal G,I_{\bullet}}$ and $\Hck_{\mathcal G,I_{\bullet}}$, in accordance with the convention in \cite[\S1]{Lau:Degeneration}.
It turns out that $\Scal^\lambda$ is defined over $E_\lambda$.
%
%The following construction is taken from \cite[Definition~2.4.1]{Bieker:IntModels}.

\begin{defn}[\emph{cf.} {\cite[Definition~2.4.1]{Bieker:IntModels}}] \label{def:BD-Schubert}
    We fix an $I$-tuple $\blambda\coloneqq(\lambda_i)_{i\in I}$ of dominant coweights of $G$. We write $\widetilde X_{\lambda_i}\coloneqq \widetilde X_{E_{\lambda_i}}$ for the normalisation of $X$ in the reflex field $E_{\lambda_i}$, and set $\widetilde X_{\blambda} \coloneqq \prod_{i\in I}\widetilde X_{\lambda_i}$.
    
    Then the \emph{Beilinson--Drinfeld Schubert variety} (or the \emph{BD Schubert variety}) for $(\Gcal,I_\bullet,\blambda)$, denoted as  $\Gr^{\leqslant\blambda}_{\Gcal,I_\bullet}$, is defined as the scheme theoretic closure of the image of
        \begin{equation}\label{eq:BD-Schubert:cell}
            \begin{tikzcd}
            \prod_{i\in I} \Scal^{\lambda_i} \arrow[hook]{r} & \prod_{i\in I} (\Gr_{G,\eta})_{E_{\lambda_i}} \arrow[hook]{r}  & \Gr_{\Gcal, I_\bullet}\times_{X^I} \widetilde X_{\blambda},
        \end{tikzcd}
        \end{equation}
        where $\prod_{i\in I}$ denotes the $I$-fold fibre product over $\FF_q$.
 %       Then $\Zcal^{\blambda}_{I_\bullet}$ is quasi-projective over $\widetilde X_{\blambda}$, and it is projective if $\Gcal$ is parahoric everywhere.
 %
 %   For any finite separable extension $E/F$ containing $E_{\lambda_i}$ for each $i\in I$, we set 
 %     \begin{equation}
 %         \Zcal^{\blambda}_{I_\bullet, E}\coloneqq \Zcal^{\blambda}_{I_\bullet}\times_{\widetilde X_{\blambda}}(\widetilde X_E)^I.
 %     \end{equation}
 %    Equivalently, $\Zcal^{\blambda}_{I_\bullet, E}$ is the closure of $\prod_{i\in I}\Scal^{\lambda_i}_E$ in $\Gr_{\Gcal, I_\bullet}\times_{X^I} \widetilde X_E^I$ as the formation of closure commutes with arbitrary flat base change.
 \end{defn}
 Set $\eta_{\blambda} \coloneqq \prod_{i\in I} \Spec E_{\lambda_i}$, where the product is taken over $\Spec \FF_q$, and set $\eta_{\blambda}'\coloneqq \eta_{\blambda}\cap W_I$. (Here, $W_I\subset X^I$ is the locus of pairwise disjoint points.) Then $ \prod_{i\in I} \Scal^{\lambda_i}|_{\eta'_{\blambda}}$ is an orbit under the action of $\Lcal^+_{X^I}\Gcal |_{\eta_{\blambda}'}$ on $\Gr_{\Gcal,I_\bullet}\times_{X^I}\eta_{\blambda}'\cong\prod_{i\in I} (\Gr_{G,\eta})_{E_{\lambda_i}}|_{\eta'_{\blambda}}$; \emph{cf.}~\eqref{eq:splitting-UI}.
 From flatness of $\Lcal^+_{X^I}\Gcal$ over $X^I$, it now follows that $\Gr^{\leqslant\blambda}_{\Gcal,I_\bullet}\subset \Gr_{\Gcal, I_\bullet}\times_{X^I} \widetilde X_{\blambda}$ is stable under the  action of $\Lcal^+_{X^I}\Gcal$; \emph{cf.} the proof of Lemma~2.2.5 in \cite{Bieker:IntModels}.

 %The immersion \eqref{eq:BD-Schubert:cell} can be seen to factor through a projective subscheme, so 
 Note that $\Gr^{\leqslant\blambda}_{\Gcal,I_\bullet}$ can be constructed as the schematic image inside a projective subscheme of $\Gr_{\Gcal, I_\bullet}\times_{X^I} \widetilde X_{\blambda}$. To see this, choose a faithful representation $\rho\colon \Gcal \hookrightarrow \SL(\Vcal)$ for a vector bundle $\Vcal$ on $X$ satisfying suitable conditions as in \cite[Proposition~2.2(b)]{ArastehRad-Hartl:Uniformizing}. Let $\underline\omega\coloneqq \rho\blambda$ denote the induced $I$-tuple of coweights of $\SL(\Vcal_\eta)$, which defines a closed substack $\Hck^{\underline\omega}_{\Gcal,I_\bullet}$ that is projective over $\Bun_\Gcal$ via $\pr_0$ by \cite[Proposition~3.9]{ArastehRad-Hartl:Uniformizing}. Then, the immersion \eqref{eq:BD-Schubert:cell} is factorised by 
\begin{equation}\label{eq:BD-Schubert:ARH-bound}
    \begin{tikzcd}
        \prod_{i\in I} \Scal^{\lambda_i} \arrow[hook]{r} & (\Hck^{\underline\omega}_{\Gcal,I_\bullet}\times_{\Hck_{\Gcal,I_\bullet}} \Gr_{\Gcal, I_\bullet})\times_{X^I} \widetilde X_{\blambda}
    \end{tikzcd},
\end{equation}
where we identify $\Gr_{\Gcal, I_\bullet}$ as the fibre of $\pr_0\colon \Hck_{\Gcal,I_\bullet}\to\Bun_\Gcal$ at the trivial bundle. By \emph{loc.~cit.} the target of \eqref{eq:BD-Schubert:ARH-bound} is a projective subscheme of $\Gr_{\Gcal, I_\bullet}\times_{X^I} \widetilde X_{\blambda}$. %Similarly, the schematic image in the construction of $\Zcal^{\blambda}_{I_\bullet, E}$ takes place in a projective subscheme of $\Gr_{\Gcal, I_\bullet}\times_{X^I} \widetilde X_E^I$.

The collection $\{\Gr^{\leqslant\blambda}_{\Gcal,I_\bullet}\times_{\widetilde X_{\blambda}} \prod_{i\in I}\widetilde X_{E_i}\}_{(E_i)_{i\in I}}$, where $E_i/E_{\lambda_i}$ runs through all finite separable subextensions of $\scl F$, is an example of a \emph{(generically defined) bound} as in \cite[Definitions~2.2.1, 2.3.1]{Bieker:IntModels}. By abuse of notation, we let $\Gr^{\leqslant\blambda}_{\Gcal,I_\bullet}$ also denote the \emph{bound}, which would not cause any confusion (\emph{cf.} \cite[Lemma~2.2.8]{Bieker:IntModels}).

Let us recall how to bound the Hecke stack $\Hck_{\Gcal,I_\bullet}$ via $\Gr^{\leqslant\blambda}_{\Gcal,I_\bullet}$, following  \cite{Bieker:IntModels}. We will use the fact that $\pr_0\colon \Hck_{\Gcal,I_\bullet}\to\Bun_\Gcal$ is a $\Gr_{\Gcal,I_\bullet}$-fibration that splits \'etale-locally. Indeed, given $\Pucal\in \Hck_{\Gcal,I_\bullet}(S)$ with $\pr_0(\Pucal) = \Pcal_0\in\Bun_\Gcal(S)$, there exists an \'etale cover $\pi\colon S'\to S$ such that the $\Gcal$-torsor  $(\pi^\ast\Pcal_0)|_{\widehat\Gamma_{(x_i)}}$ becomes trivial (\emph{cf.} \cite[Lemma~3.4]{HainesRicharz:TestFtnRes}). Choosing a trivialisation $\widehat\epsilon\,\colon\Gcal_{\widehat\Gamma_{(x_i)}}\riso(\pi^\ast\Pcal_0)|_{\widehat\Gamma_{(x_i)}}$, we may view 
\begin{equation}\label{eq:lifting-to-Grass}
    (\pi^\ast\Pucal,\widehat\epsilon\, ) \in \Gr_{\Gcal,I_\bullet}(S')
\end{equation}
using the description \eqref{eq:BD-BL}. Its $\Lcal^+_{X^I}\Gcal(S')$-orbit is independent of the choice of the trivialisation $\widehat\epsilon$.

% Note that the notion of bound for \[\Pucal\coloneqq ((x_i)_{i\in I},(\Pcal_j)_{j=0,\cdots,r},(\varphi_j)_{j=1,\cdots,r})\in\Hck_{\Gcal,I_\bullet}(S)\] 
%should only depend on the restriction of $\varphi_j$ at the formal neighbourhood $\widehat\Gamma_{x_i}$ of $\Gamma_{x_i}$ at each $i\in I_j$. By \cite[Lemma~3.4]{HainesRicharz:TestFtnRes}, there exists an \'etale covering $\pi\colon S'\to S$ such that $(\pi^\ast\Pcal_m)|_{\widehat\Gamma_{x_i'}}$ is a trivial $\Gcal$-torsor for any $i\in I$, where $x_i'\in X(S')$ is the point induced by $x_i$. We say that such \'etale covering $\pi$ \emph{trivialises} $\Pcal_m|_{\widehat\Gamma_{x_i}}$. Now, by the Beauville-Laszlo descent lemma for $\Gcal$-torsors  (\emph{cf.} \cite[Lemma~5.1]{ArastehRad-Hartl:LocGlShtuka})
%we can produce an $S'$-point of $\Gr_{\Gcal,I_\bullet}$
%\begin{equation}\label{eq:lifting-to-Grass}
%\Pucal'\coloneqq ((x_i')_{i\in I},(\Pcal_j')_{j=0,\cdots,r},(\varphi_j')_{j=1,\cdots,r},\epsilon)\in \Gr_{\Gcal,I_\bullet}(S')
%\end{equation}
%such that for each $i\in I$ and $j=0,\cdots, r$ there is an isomorphism $\Pcal_j'|_{\widehat\Gamma_{x_i'}}\cong (\pi^\ast\Pcal_j)|_{\widehat\Gamma_{x_i'}}$ which is compatible with $\pi^\ast(\varphi_j)$ and $\varphi_j'$. (See \cite[Remark~2.1.2]{Bieker:IntModels} and the references therein for further details.)

\begin{defn}\label{def:bdd-by-mu}
    Let $I_{\bullet}$, $\blambda$ and $\widetilde{X}_{\blambda}$ be as in Definition~\ref{def:BD-Schubert}. We define $\Hck^{\leqslant\blambda}_{\Gcal, I_\bullet}$ to be the closed substack of $\Hck_{\Gcal, I_\bullet}\times_{X^I}\widetilde X_{\blambda}$ given by the following condition: 
        for an $\FF_q$-scheme $S$,  $\Hck_{\Gcal, I_\bullet}^{\leqslant\blambda}(S)$ consists of $\Pucal \in (\Hck_{\Gcal, I_\bullet}\times_{X^I}\widetilde X_{\blambda})(S)$ such that for some \'etale covering $\pi\colon S'\to S$ the construction in \eqref{eq:lifting-to-Grass} yields an $S'$-point of $\Gr^{\leqslant\blambda}_{\Gcal,I_\bullet}$. This definition is independent of the choice of the \'etale covering $S'\to S$ and the trivialisation $\widehat\epsilon$ of $(\pi^\ast\Pcal_0)|_{\widehat\Gamma_{(x_i)}}$ by $\Lcal^+_{X^I}\Gcal$-stability of  $\Gr^{\leqslant\blambda}_{\Gcal,I_\bullet}$ (\emph{cf.} \cite[Definition~3.1.3]{Bieker:IntModels}).
%    Set $\Sht^{\leqslant\blambda}_{\Gcal, I_\bullet}\coloneqq \Sht_{\Gcal, I_\bullet}\times_{\Hck_{\Gcal, I_\bullet}}\Hck^{\leqslant\blambda}_{\Gcal, I_\bullet}$.
%
%    A \emph{$\Gcal$-shtuka over $S$ bounded by $\blambda$} is defined to be an $S$-point of $\Sht^{\leqslant\blambda}_{\Gcal, I_\bullet}$.
\end{defn}

We have focused on the bounds given by the BD Schubert varieties, which are most relevant to us, although there are other natural examples of bounds as indicated in \cite[Example~2.4.2]{Bieker:IntModels} and in the recent work of Hartl and Xu \cite{HartlXu:UniformizingII}. %Note also that the theory of bounds builds upon earlier work of Arasteh~Rad and Hartl \cite{ArastehRad-Hartl:Uniformizing} and Arasteh~Rad and Habibi \cite{ArastehRad-Habibi:LocMod}.

\begin{rmk}\label{rmk:convention-of-modification-and-Frobenius}
    We tried to faithfully follow the convention in \cite{Lau:Degeneration} for signs and indices, which may be slightly different from \cite{ArastehRad-Hartl:LocGlShtuka}, \cite{ArastehRad-Hartl:Uniformizing}, \cite{Bieker:IntModels}.
%  Our definition of $\Hck^{\leqslant\blambda}_{\Gcal,I_\bullet}$ and $\Sht^{\leqslant\blambda}_{\Gcal,I_\bullet}$ is consistent with \cite{Lau:Degeneration} whenever applicable, and our sign convention for the coweights $\blambda$ for the BD Schubert variety is consistent with \emph{loc.~cit}.
%  to avoid confusion. Comparing to \cite{Bieker:IntModels}, Comparing \yg{Convention check is incomplete yet, cf. newly added `bound-convention.txt' file.}
\end{rmk}

The following results are straightforward generalisations of the results in Arasteh~Rad and Hartl \cite{ArastehRad-Hartl:Uniformizing} for the bounds given by BD Schubert varieties.
\begin{prop}[Arasteh~Rad--Hartl {\cite[Propositions~3.9,~3.12]{ArastehRad-Hartl:Uniformizing}}] \label{prop:projectivity}
For any $I_\bullet$ and $\blambda$ as above, the morphism $\pr_0\colon\Hck^{\leqslant\blambda}_{\Gcal,I_\bullet}\to\Bun_\Gcal$ is projective.

Furthermore, for any refinement $I'_\bullet$ of $I_\bullet$ the natural map
\[
\Hck^{\leqslant\blambda}_{\Gcal,I'_\bullet} \to \Hck^{\leqslant\blambda}_{\Gcal,I_\bullet}
\]
defined by ``joining modifications'' is a projective birational morphism.    
\end{prop}
\begin{proof}
    Choose a faithful representation $\rho\colon \Gcal\hookrightarrow\SL(\Vcal)$ for some vector bundle $\Vcal$ on $X$ as in \cite[Proposition~2.2(b)]{ArastehRad-Hartl:Uniformizing}. Then Arasteh~Rad and Hartl  constructed a  closed substack $\Hck^{\underline\omega}_{\Gcal,I_\bullet}$ with $\underline\omega\coloneqq\rho\blambda$ that is projective over $\Bun_\Gcal$; \emph{cf.} \cite[Propositions~3.9,~3.12]{ArastehRad-Hartl:Uniformizing}, built upon \cite[Theorem~2.19]{Richarz:AffGr}. (Recall the assumption from the outset that  $\Gcal$ is parahoric everywhere.) By the construction of $\Gr^{\leqslant\blambda}_{\Gcal,I_\bullet}$ and $\Hck^{\leqslant\blambda}_{\Gcal,I_\bullet}$ as explained in \eqref{eq:BD-Schubert:ARH-bound} and above Definition~\ref{def:bdd-by-mu}, the natural inclusion $\Hck^{\leqslant\blambda}_{\Gcal,I_\bullet}\hookrightarrow\Hck_{\Gcal,I_\bullet}\times_{X^I}\widetilde X_{\blambda}$ can be factored by $\Hck^{\leqslant\blambda}_{\Gcal,I_\bullet}\hookrightarrow\Hck^{\underline\omega}_{\Gcal,I_\bullet}\times_{X^I}\widetilde X_{\blambda}$, which is a closed immersion; to check that its pull back under $S\to\Bun_\Gcal$ is a closed immersion, one may work over an \'etale covering $S'$ instead where the map is essentially    given by the closure of \eqref{eq:BD-Schubert:ARH-bound}. This shows that $\pr_0\colon\Hck^{\leqslant\blambda}_{\Gcal,I_\bullet}\to\Bun_\Gcal$ is projective. 

    The projectivity of the map $\Hck^{\leqslant\blambda}_{\Gcal,I'_\bullet} \to \Hck^{\leqslant\blambda}_{\Gcal,I_\bullet}$ is now clear as the source and the target are both projective over $\Bun_\Gcal$. Furthermore, it restricts to an isomorphism over the preimage of $W_I$ in $\widetilde X_{\blambda}$.
\end{proof}

\begin{defn}\label{def:sht}
    The \emph{moduli stack of $\Gcal$-shtukas} for $I_\bullet$ \emph{bounded by $\blambda$}, denoted by $\Sht^{\leqslant\blambda}_{\Gcal, I_\bullet}$, is defined via the following $2$-cartesian square
  \[
      \xymatrix{
      \Sht^{\leqslant\blambda}_{\Gcal, I_\bullet} \ar[d] \ar[r]&
      \Bun_\Gcal \ar[d]^-{({}^\tau(-),\id)}\\
      \Hck^{\leqslant\blambda}_{\Gcal, I_\bullet} \ar[r] %_-{(\pr_0,\pr_r)}
      & \Bun_\Gcal\times\Bun_\Gcal}
  \]
  where the bottom arrow sends $( (x_i), (\Pcal_j), (\varphi_j))$ to $(\Pcal_0,\Pcal_r)$. An object in  $\Sht^{\leqslant\blambda}_{\Gcal, I_\bullet}(S)$ is called a \emph{$\Gcal$-shtuka for $I_\bullet$ over $S$ bounded by $\blambda$}, which can be understood as a tuple $(\Pucal,\alpha)$ where $\Pucal\in\Hck^{\leqslant\blambda}_{\Gcal,I_\bullet}(S)$ and $\alpha\colon {}^\tau\Pcal_0\riso \Pcal_r$ is an isomorphism of $\Gcal$-bundles. 
\end{defn}
  It is known that $\Sht^{\leqslant\blambda}_{\Gcal, I_\bullet}$ is a Deligne--Mumford stack that is separated and locally of finite type over $\widetilde X_{\blambda}$; \emph{cf.} \cite[Theorem~2.7.8]{HartlXu:UniformizingII}. Furthermore, we have the following immediate corollary of Proposition~\ref{prop:projectivity}.

  \begin{cor}\label{cor:joining-modifications}
      For any refinement $I'_\bullet$ of $I_\bullet$ the natural map
\[
\Sht^{\leqslant\blambda}_{\Gcal,I'_\bullet} \to \Sht^{\leqslant\blambda}_{\Gcal,I_\bullet}
\]
defined by ``joining modifications'' is a projective birational morphism. 
  \end{cor}

  Lastly, we recall the effect of change of parahoric integral model $\Gcal$. Suppose that $\Gcal'$ is a smooth affine group scheme over $X$ equipped with a group homomorphism $\iota\colon\Gcal\to\Gcal'$ that restricts to an isomorphism on some dense open subscheme of $X$. We also assume that $\Gcal'$ is parahoric everywhere. Then we have a natural morphism of stacks
  \[\Bun_\Gcal\to\Bun_{\Gcal'};\quad \Pcal\mapsto\Pcal\times^\Gcal\Gcal'\]
induced by the pushout construction, which in turn induces morphisms $\Gr_{\Gcal,I_\bullet}\to\Gr_{\Gcal',I_\bullet}$ and $\Hck_{\Gcal,I_\bullet}\to \Hck_{\Gcal',I_\bullet}$. 
Fixing $\blambda$ as above, these morphisms restrict to
\[\Gr^{\leqslant\blambda}_{\Gcal,I_\bullet}\to\Gr^{\leqslant\blambda}_{\Gcal',I_\bullet},\quad \Hck^{\leqslant\blambda}_{\Gcal,I_\bullet}\to \Hck^{\leqslant\blambda}_{\Gcal',I_\bullet}\quad \text{and}\quad \Sht^{\leqslant\blambda}_{\Gcal,I_\bullet}\to \Sht^{\leqslant\blambda}_{\Gcal',I_\bullet}.\]
(Note that the reflex field $E_{\lambda_i}$ does not depend on the choice of $\Gcal$ or $\Gcal'$, so the above morphisms are defined over $\widetilde X_{\blambda}$.)

Let us now recall the following result from \cite[Theorem~3.5.5]{Bieker:IntModels}, building upon \cite[Theorem~3.20]{Breutmann:Thesis}.
\begin{thm}[{\emph{Cf.}~\cite[Theorem~3.5.5]{Bieker:IntModels}}]\label{th:change-parahorics}
    In the above setting, the morphism $\Sht^{\leqslant\blambda}_{\Gcal,I_\bullet}\to \Sht^{\leqslant\blambda}_{\Gcal',I_\bullet}$ is proper and surjective. Furthermore, it restricts to a finite \'etale map over $U_{\iota}^I$ where $U_{\iota}\subset X$ is the locus where $\iota|_{U_{\iota}}$ is an isomorphism.
\end{thm}

To conclude the section, let us specialise to the case when $\Gcal$ is an integral model of an inner form of $\GL_d$ that is parahoric everywhere. We will later focus on the case when $\Gcal$ comes from a hereditary order of a central division algebra over $F$.

\begin{exa}\label{exa-typeA}
%Let us first consider the following slightly more general setting. 
Let $B$ be a central simple algebra over $F$ with dimension $d^2$, and we fix a \emph{hereditary} order $\Bcal$ over $X$ so that $\Gcal\coloneqq \Bcal^\times$ is parahoric everywhere. Since its generic fibre $G$ is an inner form of $\GL_d$, dominant coweights $X_\ast(T)_+$ are given by $d$-tuples of decreasing integers and the geometric conjugacy class of $G$ containing any $\lambda\in X_\ast(T)_+$ is defined over $F$. (To show the latter claim, note that there is an isomorphism $G(\breve{F})\cong \GL_d(\breve{F})$, and the natural $q$-Frobenius $\tau_G$ on $G(\breve{F})$ corresponds to $\Int(\gamma)\circ \tau_{\GL_d}$ on $\GL_d(\breve{F})$ for some $\gamma\in\GL_d(\breve{F})$.) In particular, for any $\blambda=(\lambda_i)\in(X_\ast(T)_+)^I$ we have $\widetilde X_{\blambda} = X^I$.
%It follows that the natural map
%\begin{equation}\label{eq-exa-typeA}
%    \Sht^{\leqslant\blambda}_{\Bcal^\times,I_\bullet}\times_{\widetilde X_{\blambda}}\breve X^I \to \Sht_{\Bcal^\times, I_\bullet}\times_{X^I}\breve X^I
%\end{equation}
%is a closed immersion; i.e., the condition of being bounded by $\blambda$ can be ``rationally defined over $\breve X^I$''.
\end{exa}

\begin{rmk}\label{rmk-shtuka-via-module}
Recall that a \emph{rank-$n$ shtuka over $S$} is a tuple $(\Eucal,\alpha)$, where
\begin{equation}\label{eq-shtuka-vb}
    \Eucal \coloneqq \big( (x_i)_{i\in I}, (\Ecal_j)_{j=0,\cdots, r}, (\varphi_j)_{j=1,\cdots, r}\big) \quad \text{and}\quad\alpha\colon {}^\tau\Ecal_0\riso\Ecal_r.
\end{equation}
Here, $\Ecal_j$ is a rank-$n$ vector bundle on $X\times S$ for any $0\leq j\leq r$,  and 
\begin{equation}
     \varphi_j \colon \Ecal_{j-1} |_{ (X\times S) \smallsetminus \bigcup_{i\in I_j} \Gamma_{x_i} } \to  \Ecal_{j} |_{ (X\times S) \smallsetminus \bigcup_{i\in I_j} \Gamma_{x_i} }
\end{equation}
are isomorphisms of vector bundles for any $1\leq j\leq r$.

Via the correspondence \eqref{eq:torsor-module} between $\Bcal^\times$-torsors over $X\times S$ and rank-$1$ locally free right $\Bcal_{X\times S}$-modules, one can interpret a $\Bcal^\times$-shtuka over $S$ as a rank-$d^2$ shtuka $(\Eucal,\alpha)$ over $S$ given by the tuple as in \eqref{eq-shtuka-vb}, such that each $\Ecal_j$ is a rank-$1$ locally free right $\Bcal_{X\times S}$-module, and $\varphi_j$'s and $\alpha$ are $\Bcal$-equivariant.  We can give a similar interpretation for $\Eucal\in\Hck_{\Bcal^\times,I_\bullet}(S)$ and $(\Eucal,\epsilon)\in \Gr_{\Bcal^\times,I_\bullet}(S)$.

It is worth noting that if the legs are allowed to meet the ramification locus of $B$, then Definition~\ref{def:bdd-by-mu} for $\Hck^{\leqslant\blambda}_{\Bcal^\times,I_\bullet}$ and $\Sht^{\leqslant\blambda}_{\Bcal^\times,I_\bullet}$ may not admit a nice linear algebraic interpretation in general. Indeed, the BD Schubert variety $\Gr^{\leqslant\blambda}_{\Bcal^\times,I_\bullet}$ may not admit a nice moduli description in general. 
\end{rmk}

Returning to moduli stacks of $\Bcal^\times$-shtukas, let us record the following condition for the non-emptiness of $\Sht^{\leqslant\blambda}_{\Bcal^\times,I_\bullet}$. For $\lambda_i = (\lambda_{i,1},\cdots,\lambda_{i,d})$, we define
\begin{equation}
    \deg(\lambda_i)\coloneqq\sum_{j=1}^d\lambda_{i,j}.
\end{equation}
\begin{lem}\label{lem:non-empty-Sht}
    The stack $\Sht^{\leqslant\blambda}_{\Bcal^\times,I_\bullet}$ is non-empty if and only if $\sum_{i\in I}\deg(\lambda_i) = 0$.
\end{lem}
%The converse of the lemma also holds, which we show in Proposition~\ref{prop:non-empty-basic-stratum}.
\begin{proof}
    Let $T\subset B^\times_{\scl F}\cong \GL_{d,\scl F}$ be a maximal torus, and  $T_{\GL}\subset \GL_{d^2}$ the subgroup of diagonal matrices. Given $\lambda_i = (\lambda_{i,1},\cdots,\lambda_{i,d})\in X_\ast(T)_+$ we set 
    \begin{equation}\label{eq:lambda-GL}
        \lambda_{i,\GL}\coloneqq (\underbrace{\lambda_{i,1},\cdots,\lambda_{i,1}}_{d\text{ times}}, \cdots, \underbrace{\lambda_{i,d},\cdots,\lambda_{i,d}}_{d\text{ times}}) \in X_\ast(T_{\GL})_+ .
    \end{equation}
    Given $\blambda\in X_\ast(T)_+^I$, we let $\blambda_{\GL} = (\lambda_{i,\GL})_{i\in I}\in X_\ast(T_{\GL})_+^I$. Then the natural map
   $\Gr_{\Bcal^\times,I_\bullet} \hookrightarrow \Gr_{\GL_{d^2},I_\bullet}$,  defined by viewing a rank-$1$ locally free right $\Bcal_{X\times S}$-module as a rank-$d^2$ vector bundle, induces 
   \begin{equation}\label{eq:functoriality-BD-Gr}
       \begin{tikzcd}
           \Gr^{\leqslant\blambda}_{\Bcal^\times,I_\bullet} \arrow[r,hookrightarrow] &  \Gr^{\leqslant\blambda_{\GL}}_{\GL_{d^2},I_\bullet}
       \end{tikzcd};
   \end{equation}
   indeed, the generic affine Schubert cell $\prod_{i\in I}\Scal^{\lambda_i}$ in \eqref{eq:BD-Schubert:cell} maps to the cell $\prod_{i\in I} \Scal^{\lambda_{i,\GL}}\subset \prod_{i\in I}\Gr_{\GL_{d^2},\eta}$. Thus, we have a morphism
    \begin{equation}
        \begin{tikzcd}
    \Sht^{\leqslant\blambda}_{\Bcal^\times,I_\bullet} \arrow[r]& 
    \Sht^{\leqslant\blambda_{\GL}}_{\GL_{d^2},I_\bullet}
    \end{tikzcd},
    \end{equation}
    and the non-emptiness of $\Sht^{\leqslant\blambda}_{\Bcal^\times,I_\bullet}$ implies the non-emptiness of $ \Sht^{\leqslant\blambda_{\GL}}_{\GL_{d^2},I_\bullet}$, which in turn implies the following equality by \cite[Proposition~2.16(d)]{Varshavsky:shtuka}:
    \[
    0 = \sum_{i\in I}\deg(\lambda_{i,\GL}) = d\cdot\sum_{i\in I}\deg(\lambda_{i}).
    \]
    
    Conversely, suppose we have $\sum_{i\in I}\deg(\lambda_{i})=0$. By Corollary~\ref{cor:joining-modifications}, it suffices to show $\Sht^{\leqslant\blambda}_{\Bcal^\times,(I)}$ is non-empty. Now for any geometric point $\cl y\in X(\cl \kappa)$ in the split locus for $B$, we have 
    \[
    \big((\cl y,\cdots,\cl y), (\Ecal_j = \Bcal_{X\times\cl\kappa})_{j=0,1}, \varphi_1=\id, \alpha = \id \big) \in \Sht^{\leqslant\blambda}_{\Bcal^\times,(I)}(\cl\kappa),
    \]
    since we have $\underline 0\leq \sum_i\lambda_i$. 
\end{proof}

Now, recall the following action of the id\`ele group $\AA_F^\times$ on moduli stacks of $\Bcal^\times$-shtukas.
\begin{defn}\label{defn-idele-action-on-Bun-B}
    Recall that we have a natural degree-preserving map $\AA^\times_F\twoheadrightarrow \Pic(X)$ with kernel $F^\times\cdot\prod_v\Ocal_v^{\times}$. For $a\in\AA^\times_F$, let $\Lcal(a)$ denote the line bundle on $X$ corresponding to $a$. Then the operation $\Lcal(a)\otimes_{\Ocal_X}(-)$ on rank-$1$ locally free $\Bcal_{X\times S}$-modules defines an action of $\AA^\times_F$ on $\Bun_{\Bcal^\times}$; \emph{cf.} Remark~\ref{rmk-shtuka-via-module}.
%     We now describe the natural action of $\AA^\times_F$ on $\Bun_{\Bcal^\times}$ as follows: we let $a\in\AA^\times_F$ act on any rank-$1$ locally free $\Bcal_{X\times S}$-module $\Ecal$ by twisting with the line bundle associated to $a$. Clearly, the action of $a\in F^\times\cdot \prod_v\Ocal_v$ is trivial on $\Bun_{\Bcal^\times}$.

    Similarly, we can let $\AA^\times_F$ act on $\Hck_{\Bcal^\times,I_\bullet}$ and $\Sht_{\Bcal^\times,I_\bullet}$ via the natural action on each $\Ecal_i$. Finally, the $\AA^\times_F$-action preserves the substacks $\Hck_{\Bcal^\times,I_\bullet}^{\leqslant\blambda}$ and $\Sht_{\Bcal^\times,I_\bullet}^{\leqslant\blambda}$, as the construction \eqref{eq:lifting-to-Grass} is invariant under the $\AA^\times_F$-action. Thus, if $a\in \AA^\times_F$ is an id\`ele of positive degree, this action induces an isomorphism
    \begin{equation}\label{eq:sht-mod-a-disjoint-union-of-degree}
      \Sht^{\leqslant\blambda}_{\Bcal^\times,I_\bullet}/a^{\ZZ}\simeq \coprod_{e \in [0, d^2\deg(a)-1]} \Sht^{\leqslant\blambda,\deg=e}_{\Bcal^\times,I_\bullet}.
    \end{equation}
    Here, $\Sht^{\leqslant\blambda,\deg=e}_{\Bcal^\times,I_\bullet}$ denotes the open and closed substack of $\Sht^{\leqslant\blambda}_{\Bcal^\times,I_\bullet}$ whose $S$-points $(\Eucal,\alpha)$ are characterised as follows:  the rank-$d^2$ vector bundle $\Ecal_{0}$ over $X\times S$ has constant degree $e$ at each geometric point of $S$. 
\end{defn}

%In particular, $\Sht^{\leqslant\blambda}_{\Bcal^\times,I_\bullet}$ is never quasi-compact as it has infinitely many connected components. 

\begin{rmk}\label{rmk:non-qc}
    The above definition also shows that $\Sht^{\leqslant\blambda}_{\Bcal^\times,I_\bullet}$  has infinitely many connected components, and thus it is never quasi-compact. Furthermore, passing to the quotient $\Sht^{\leqslant\blambda}_{\Bcal^\times,I_\bullet}/a^{\ZZ}$ may not guarantee quasi-compactness in general, and we only get a covering by quasi-compact open substacks given by some ``Harder--Narasimhan truncation'' (\emph{cf.} \cite[Theorem~2.1]{Schieder:HNstr}). There is no reason to expect that Harder--Narasimhan parameters appearing in $\Sht^{\leqslant\blambda}_{\Bcal^\times,I_\bullet}/a^{\ZZ}$ are bounded, especially when $B = \Bcal\otimes_{\Ocal_X}F$ is \emph{not} a division algebra. 

    Let $\Dcal$ be a hereditary $\Ocal_X$-order of a central division algebra $D$. Then it is known that the restriction $(\Sht^{\leqslant\blambda}_{\Dcal^\times,I_\bullet}/a^\ZZ)|_{(X\setminus\Ram(D))^I}$ is quasi-compact; this was proved in \cite[Proposition~1.11]{Lau:Degeneration} when $\Dcal$ is maximal, and the general case reduces to this case by Theorem~\ref{th:change-parahorics}.  
    Perhaps surprisingly, it is not known whether the entire stack $\Sht^{\leqslant\blambda}_{\Dcal^\times,I_\bullet}/a^\ZZ$ is quasi-compact. To explain the subtlety, note that $\Sht^{\leqslant\blambda}_{\Dcal^\times,I_\bullet}/a^\ZZ$ is of finite type over $\Bun_{\Dcal^\times}/a^\ZZ$, but $\Bun_{\Dcal^\times}/a^\ZZ$ is \emph{not} quasi-compact. (To see this, note that we have
\[
\Bun_{\Dcal^\times}\times_{\Spec\FF_q}\Spec\cl\FF_q \cong \Bun_{(\Dcal\otimes_{\FF_q}\cl\FF_q)^\times},
\]
but $(\Dcal\otimes_{\FF_q}\cl\FF_q)^\times$ is a parahoric Bruhat--Tits integral model of $(D\otimes_{\FF_q}\cl\FF_q)^\times \cong \GL_d$ over $\breve F$. In particular, proper parabolic subgroups of $D^\times$ defined over $\breve F$ prohibit the stack $\Bun_{\Dcal^\times}$ from being quasi-compact.) Therefore, to show that $\Sht^{\leqslant\blambda}_{\Dcal^\times,I_\bullet}/a^\ZZ$ is quasi-compact, we need to show that its image in $\Bun_{\Dcal^\times}/a^\ZZ$ is quasi-compact. To achieve this, we follow the standard strategy by showing the \emph{irreducibility} of $\Dcal^\times$-shtukas in the sense of Definition~\ref{def:D-subshtukas} below, but we are able to obtain  quasi-compactness results by this approach only under some additional assumption. (See Proposition~\ref{prop:quasi-compactness-of-Sht-D-mod-a} and Theorem~\ref{thm:from-inequality-to-properness}.)
    
%    In fact, we impose an additional assumption to  show the quasi-compactness of  $\Sht^{\leqslant\blambda}_{\Dcal^\times,I_\bullet}$; \emph{cf.} Proposition~\ref{prop:quasi-compactness-of-Sht-D-mod-a}. The main reason why the quasi-compactness is non-trivial is that there may be many non-trivial $\Dcal_{X\times S}$-submodules of $\Dcal_{X\times S}$ with $\Ocal_{X\times S}$-rank less than $d^2$  if $S$ is a scheme over $\cl\FF_{q}$. (A related fact is that the reductive $F$-group $D^\times$ has many proper parabolics defined over $\breve F$.) Exactly for the same reason, there may be a degenerating family of $\Dcal^\times$-shtukas as asserted in \cite[Theorem~A]{Lau:Degeneration}.
\end{rmk}

%We will work with the following notion mainly in \S\ref{sec-valcrit}, but we introduce it now to motivate the generality of \S\ref{sec-isosht}.
\begin{defn}\label{def:D-subshtukas}
    Let $(\Eucal,\alpha)$ be a rank-$n$ shtuka over $S$  given by the tuple  in \eqref{eq-shtuka-vb}. A \emph{subshtuka} of $(\Eucal,\alpha)$ is defined to be a rank-$m$ shtuka over $S$
    \[( (x_i)_{i\in I}, (\Fcal_j)_{j=0,\cdots, r}, (\varphi_j')_{j=1,\cdots, r} , \alpha')\]
    with the same legs $x_i$, where each $\Fcal_j$ is a rank-$m$ subbundle of $\Ecal_j$ for each $j = 0,\cdots,r$ (which by definition implies that $\Ecal_j/\Fcal_j$ is also a vector bundle on $X\times S$), and we have $\varphi'_j = \varphi_j|_{\Fcal_{j-1}}$ for each $j=1,\cdots,r$ and $\alpha' = \alpha|_{^{\tau}\Fcal_0}$. For simplicity, we denote the above subshtuka by $(\Fucal,\alpha)$ without strictly distinguishing $\alpha|_{\ltau\Fcal_0}$ and $\alpha$.

    If $(\Eucal,\alpha)$ is a $\Bcal^\times$-shtuka (i.e., each $\Ecal_j$ is rank-$1$ locally free as a right $\Bcal_{X\times S}$-module), then we say that a subshtuka $(\Fucal,\alpha)$ is \emph{$\Bcal$-stable} if each $\Fcal_j$ is a $\Bcal_{X\times S}$-submodule of $\Ecal_j$. We say that $(\Eucal,\alpha)$ is \emph{irreducible} if it does not admit any non-zero proper $\Bcal$-stable subshtuka.
\end{defn}

Even when $\Bcal = \Dcal$ is a maximal order of a central division algebra $D$ over $F$, a $\Dcal^\times$-shtuka $(\Eucal,\alpha)$ over $\cl\FF_q$ may \emph{not} be irreducible as explained in Remark~\ref{rmk:non-qc}.

From now on, let $D$ be a central division algebra over $F$ with dimension $d^2$, and let $\Ram(D)\subset|X|$ denote the set of places where $D$ is ramified.
We fix a hereditary order $\Dcal$ over $X$.

Let us consider the following condition on $D$ and a finite subset $Y\subseteq \Ram(D)$.
\begin{ass}\label{ass:quasi-compact-Sht-D-setting}
    Given $D$ and a finite subset $Y\subseteq \Ram(D)$, suppose that the lowest common denominator of $\{\inv_{y}(D)\}_{y\in Y}$ is exactly $d$, the index of $D$.
\end{ass}

%Perhaps surprisingly, even the quasi-compactness of $\Sht^{\leqslant\blambda}_{\Dcal^\times,I_\bullet}/a^\ZZ$ is non-trivial; indeed, although the natural map $\pr_0\colon\Sht^{\leqslant\blambda}_{\Dcal^\times,I_\bullet}/a^\ZZ \to \Bun_{\Dcal^\times}$ is quasi-compact, the stack $\Bun_{\Dcal^\times}$ is \emph{not} quasi-compact. (To see this, note that $\Bun_{\Dcal^\times}\times\Spec\cl\FF_q \cong \Bun_{(\Dcal\otimes\cl\FF_q)^\times}$, where $\Dcal\otimes\cl\FF_q$ is a hereditary order of $D\otimes\cl\FF_q \cong M_d(\breve F)$.)
%
%We do not know if $\Sht^{\leqslant\blambda}_{\Dcal^\times,I_\bullet}/a^\ZZ$ should be quasi-compact in general, but we show the quasi-compactness under a mild extra assumption.

We now obtain a quasi-compactness result, generalising \cite[Proposition~1.11]{Lau:Degeneration}.

\begin{prop}\label{prop:quasi-compactness-of-Sht-D-mod-a}
    Fix an id\`ele $a\in \AA_F^\times$ of positive degree. 
    \begin{enumerate}
        \item\label{prop:quasi-compactness-of-Sht-D-mod-a:leg-restriction} If a non-empty finite subset $Y\subseteq \Ram(D)$ satisfies Assumption~\ref{ass:quasi-compact-Sht-D-setting} for $D$, then $(\Sht^{\leqslant\blambda}_{\Dcal^\times,I_\bullet}/a^{\ZZ})|_{(X\setminus Y)^I}$ is of finite type over $\FF_q$. %(so in particular, it is quasi-compact).
        \item\label{prop:quasi-compactness-of-Sht-D-mod-a:no-restriction} Suppose that $ c \coloneqq |\Ram(D)| - |I| >0$, and every subset $Y\subset \Ram(D)$ of size $ c $ satisfies Assumption~\ref{ass:quasi-compact-Sht-D-setting} for $D$. Then $\Sht^{\leqslant\blambda}_{\Dcal^\times,I_\bullet}/a^{\ZZ}$ is of finite type over $\FF_q$. %(so in particular, it is quasi-compact).
    \end{enumerate}    
\end{prop}

Before we prove the proposition, let us record the following immediate corollary.
\begin{cor}\label{cor:quasi-compactness-of-Sht-D-mod-a}
    Suppose that $D_y$ is a division algebra for any $y\in \Ram(D)$, which is automatic if $d$ is a prime. If $|\Ram(D)|>|I|$ then $\Sht^{\leqslant\blambda}_{\Dcal^\times,I_\bullet}/a^{\ZZ}$ is of finite type over $\FF_q$. %In particular, if $d$ is a prime number then $\Sht^{\leqslant\blambda}_{\Dcal^\times,I_\bullet}/a^{\ZZ}$ is of finite type over $\FF_q$ if we have $|I|>|X\setminus U|$.
\end{cor}

   Our proof follows the same strategy as \cite[Proposition~1.11]{Lau:Degeneration}, and the additional hypothesis is needed to ensure the irreducibility (see Lemma~\ref{lem:irreducible-shtuka}). Note that Proposition~\ref{prop:quasi-compactness-of-Sht-D-mod-a}(\ref{prop:quasi-compactness-of-Sht-D-mod-a:no-restriction}) is \emph{not} an optimal quasi-compactness result; see Remark~\ref{rmk:quasi-compactness-optimality} for further discussions.

We now state a few lemmas needed for the proof of the proposition.
Given a vector bundle $\Ecal$ over $X\times k$, let $\mu^+(\Ecal)$ and $\mu^-(\Ecal)$ respectively denote the largest and smallest Harder--Narasimhan slopes. %Given a rank-$n$ shtuka $(\Eucal,\alpha)$ over $k$ given as in \eqref{eq-shtuka-vb}, then we set $\mu^\pm((\Eucal,\alpha))\coloneqq\mu^\pm(\Ecal_0)$. %The following lemma asserts that if the Harder--Narasimhan parameter for a rank-$n$ shtuka $(\Eucal,\alpha)$ is \emph{sufficiently convex} when viewed as a rational dominant coweight in $\GL_n$, then the shtuka $(\Eucal,\alpha)$ is \emph{reducible} (in the sense that it admits a non-trivial proper subshtuka).

%\yg{[Change to $\Ecal_0$. ]}
\begin{lem}[{\emph{Cf.}~\cite[\S4.2, Lemma]{Laumon:DrinfeldShtukas}}]\label{lem:bounding-HN-parameter-of-Sht-D}
  Let $I_{\bullet}$ and $\blambda\in (X_\ast(T)_+)^I$ be as in Definition~\ref{def:BD-Schubert}. Then there exists a constant $R > 0$ depending on $\blambda$ and $\Dcal$, such that the following property holds. For any $\Dcal^\times$-shtuka 
    \[
        (\Eucal,\alpha)\coloneqq ( (x_i)_{i\in I}, (\Ecal_j)_{j=0,\cdots, r}, (\varphi_j)_{j=1,\cdots, r} , \alpha ) \in \Sht^{\leqslant\blambda}_{\Dcal^\times,I_\bullet}(k)
    \]
    over an algebraically closed extension $k$ of $\FF_q$, if there exists a subbundle $0\subsetneq\Fcal\subsetneq \Ecal_0$ satisfying
    \[
        \mu^{-}(\Fcal) > \mu^+(\Ecal_0/\Fcal) + R,
    \]
    then $\Fcal$ is a $\Dcal_{X\times k}$-submodule of $\Ecal_0$ and there exists a $\Dcal$-stable subshtuka $(\Fucal,\alpha)$ of $(\Eucal,\alpha)$ such that $\Fcal_0=\Fcal$. (Here, we use the notation from Definition~\ref{def:D-subshtukas}.)
\end{lem}

\begin{proof}
    In the special case where $I_\bullet = (\{1\},\{2\})$ and $\blambda = (\lambda_1,\lambda_2)$ with $\lambda_1 = (0,\cdots,0,-1)$ and $\lambda_2 = (1,0,\cdots,0)$, the lemma was proved in \cite[\S4.2, Lemma]{Laumon:DrinfeldShtukas}. However, as noted in the proof of Proposition~1.11 in \cite{Lau:Degeneration}, the same argument extends to any choice of $I_\bullet$ and $\blambda$. %\wk{[I believe we can suppress further details.]}
\end{proof}

%\begin{proof}
%        Let $\Lcal$ be a line bundle on $X$ such that the $\Ocal_X$-module $\Dcal\otimes_{\Ocal_X}\Lcal$ is generated by global sections. For a decreasing sequence $\lambda = (\lambda^1,\cdots, \lambda^d)$ of integers, $|\lambda|$ denotes the sum $\sum_{j=1}^d|\lambda^j| \geq 0$. We define $R$ to be the supremum of $\deg(\mathcal L)$ and $ |\lambda_i|$ for all $i\in I$. Let $\mathcal F\subseteq \mathcal E_0$ be an $\Ocal_{X\otimes k}$-submodule such that $\mathcal E_0/\mathcal F$ is a torsionfree $\Ocal_{X\otimes k}$-module. 
%\end{proof}

\begin{lem}\label{lem:irreducible-shtuka}
%  Under Assumption \ref{ass:quasi-compact-Sht-D-setting}, any $\Dcal^{\times}$-shtuka over $\cl{\FF}_q$ in $\Sht^{\leqslant\blambda}_{\Dcal^\times,I_\bullet}$ is irreducible.
%Suppose that we have $ c \coloneqq |X\smallsetminus U| - |I| >0$, and for any $y_1,\cdots, y_ c \in X\setminus U$ the lowest common multiple of the torsion orders of $\{\inv_{y_l}(D)\}_{l=1,\cdots, c }$ is exactly $d$. 
Let  $(\Eucal,\alpha)\in\Sht^{\leqslant\blambda}_{\Dcal^\times,I_\bullet}(k)$ be a $\Dcal^{\times}$-shtuka over an algebraically closed extension $k$ of $\FF_{q}$. If there exists a non-empty finite subset $Y\subseteq \Ram(D)$, disjoint from the legs of $(\Eucal,\alpha)$, that satisfies Assumption~\ref{ass:quasi-compact-Sht-D-setting} for $D$, then $(\Eucal,\alpha)$ is \emph{irreducible} in the sense of Definition~\ref{def:D-subshtukas}.
\end{lem}

\begin{proof}
  Given a place $y\in |X|$ disjoint from the legs $(x_i)_{i\in I}$ of $(\Eucal,\alpha)$, let us recall the construction of the \'etale local shtuka at $y$ associated to $(\Eucal,\alpha)$ obtained via the \emph{global-local functor} at $y$ (\emph{cf.} \cite[\S5]{ArastehRad-Hartl:LocGlShtuka}). Choose a $k$-point $\cl y\in X(k)$ lying over $y$, and let
$\Ocal_{\cl y}$ and $\Ecal_{0,\cl y}$ denote the completed stalks of $\Ocal_{X\times k}$ and $\Ecal_0$ at $\cl y$, respectively. Since $\tau^{\deg(y)}$ fixes $\cl y$, it induces an endomorphism $\tau_y\,\colon\, \Ocal_{\cl y}\to\Ocal_{\cl y}$. 

Consider the modification
    \[\Phi\coloneqq(\varphi_1)^{-1}\circ\cdots\circ (\varphi_r)^{-1} \circ \alpha\,  \colon \,\xymatrix@1{{}^{\tau}\Ecal_0|_{(X\times k) \smallsetminus \bigcup_i\Gamma_{x_i}}\ar[r]^-{\cong} & \Ecal_0|_{(X\times k) \smallsetminus \bigcup_i\Gamma_{x_i}}},
  \] 
and define $\varphi_{\cl y}\,\colon\, \Ecal_{0,\cl y}\to \Ecal_{0,\cl y}$ to be the $\tau_y$-linear endomorphism induced by $\Phi\circ{}^\tau\Phi\circ\cdots\circ{}^{\tau^{\deg(y)-1}}\Phi$. Since $y$ is disjoint from any of the $x_i$'s, $\varphi_{\cl y}$ is well defined and its $\Ocal_{\cl y}$-linearisation is an isomorphism. Thus, $(\Ecal_{0,\cl y},\varphi_{\cl y})$ defines a rank-$d^2$ \'etale local shtuka over $k$ in the sense of \cite[Definition~3.1]{ArastehRad-Hartl:LocGlShtuka}. Furthermore, $\varphi_{\cl y}$ commutes with the natural right action of $\Dcal_y$ on $\Ecal_{0,\cl y}$.

  Let $(\Fucal,\alpha)$ be a non-zero $\Dcal$-stable subshtuka of $(\Eucal,\alpha)$. Then clearly, the completed stalk $\Fcal_{0,\cl y}$ of $\Fcal_0$ at $\cl y$ is a $\varphi_{\cl y}$-stable $\Dcal_y\otimes_{\Ocal_y}\Ocal_{\cl y}$-submodule, so in particular $(\Fcal_{0,\cl y},\varphi_{\cl y})$ defines an \'etale local subshtuka of $(\Ecal_{0,\cl y},\varphi_{\cl y})$.

  It now follows from \cite[Proposition~3.4]{ArastehRad-Hartl:LocGlShtuka} that we have 
  \[
  \rank_{\Ocal_y}(\Ecal_{0,\cl y})^{\varphi_{\cl y}=1} = \rank \Ecal_0 \quad \text{and}\quad \rank_{\Ocal_y}(\Fcal_{0,\cl y})^{\varphi_{\cl y}=1} = \rank \Fcal_0 ,\]
  Furthermore, $(\Ecal_{0,\cl y})^{\varphi_{\cl y}=1}\otimes_{\Ocal_y}F_y$ is a rank-$1$ free right $D_y$-module and $(\Fcal_{0,\cl y})^{\varphi_{\cl y}=1}\otimes_{\Ocal_y}F_y$ is its $D_y$-submodule. Therefore, from the description of the simple $D_y$-module we see that $\rank\Fcal_0$ is divisible by $d\cdot e_y$ where $e_y$ is the torsion order of $\inv_y(D)$.

  Now if we choose $Y\subseteq \Ram(D)$ as in the statement, then it follows from the previous paragraph that  $\rank\Fcal_0$ is divisible by $d^2 = \lcm \{d\cdot e_{y}\}_{y\in Y}$. Since $\Fcal_0$ is non-zero,  the rank consideration forces $(\Fucal,\alpha) = (\Eucal,\alpha)$, as desired.
%  
%  Let $(\Fucal,\alpha)$ be a $\Dcal$-stable subshtuka of $(\Eucal,\alpha)$, and we use the notation as in Definition~\ref{def:D-subshtukas}. By assumption, there exist $ c $ ramified places $y_1,\cdots,y_ c \in X\setminus U$  disjoint from the legs $(x_i)_{i\in I}$ of $(\Eucal,\alpha)$. For each $y = y_l$ we choose a $k$-point $\cl y\in X(k)$ lying over $y$, and let $\Ocal_{\cl y}$ denote the completed stalk of $\Ocal_{X\times k}$ at $\cl y$.  $\tau_y\colon \Ocal$
%  
%  the isomorphisms of $\Dcal_{y'}$-modules 
%  \[
%    (\varphi_1)^{-1}_{y'}\circ\cdots\circ (\varphi_r)^{-1}_{y'} \circ \alpha_{y'} \colon \xymatrix@1{\ltau (\mathcal F_0\otimes_{\Ocal_X} \Ocal_{y'}) \ar[r] & \mathcal F_0\otimes_{\Ocal_X} \Ocal_{y'}},
%  \] 
%  for all $y'\in X(k)$ lying over $y$, induce an isomorphism of $(\Dcal_{y}\otimes k)\simeq \prod_{y'}\Dcal_{y'}$-modules 
%  \[
%    \xymatrix@1{\ltau (\mathcal F_0\otimes_{\Ocal_X} \Ocal_{y} \otimes k) \ar[r] & \mathcal F_0\otimes_{\Ocal_X} \Ocal_{y} \otimes k}.
%  \] 
%  % \[ 
%  %   \ltau(\mathcal F_0\otimes_{\Ocal_X}\Ocal_x^{\operatorname{ur}}) \to \mathcal F_0\otimes_{\Ocal_X}\Ocal_x^{\operatorname{ur}}.
%  % \]
%  By \cite[Proposition~3.4]{ArastehRad-Hartl:LocGlShtuka}, we obtain a right $\Dcal_{y}$-module $\mathcal F_{0,y}$ which is finite free over $\Ocal_{y}$ such that $\mathcal F_{0,y}\otimes k\simeq \mathcal F_0 \otimes_{\Ocal_X}\Ocal_y$. Since $\mathcal F_{0,x}\otimes_{\Ocal_x}F_x \simeq D_x$, we see that $\mathcal F_0\otimes_{\Ocal_X} \breve \Ocal_{x}$
\end{proof}

\begin{proof}[Proof of Proposition~\ref{prop:quasi-compactness-of-Sht-D-mod-a}]
Note that (\ref{prop:quasi-compactness-of-Sht-D-mod-a:leg-restriction}) implies (\ref{prop:quasi-compactness-of-Sht-D-mod-a:no-restriction}), since we have 
\[\Sht^{\leqslant\blambda}_{\Dcal^\times,I_\bullet}/a^\ZZ=\bigcup(\Sht^{\leqslant\blambda}_{\Dcal^\times,I_\bullet}/a^\ZZ)|_{(X\setminus Y)^I},\] 
where the union is over all the (finitely many) choices of $Y \subseteq \Ram(D)$ with size $c$ for some fixed $1\leq c \leq|\Ram(D)|$. In the setting of  (\ref{prop:quasi-compactness-of-Sht-D-mod-a:no-restriction}), we may set $ c \coloneqq |\Ram(D)| - |I|$. 

It remains to prove (\ref{prop:quasi-compactness-of-Sht-D-mod-a:leg-restriction}).
We identify $\Bun_{\Dcal^\times}(S)$ as the groupoid of rank-$1$ locally free right $\Dcal_{X\times S}$-modules, and $\Bun_{\GL_{d^2}}(S)$ as the groupoid of rank-$d^2$ vector bundles over $X\times S$. Then the forgetful map $\Bun_{\Dcal^\times} \to \Bun_{\GL_{d^2}}$ is representable by a quasi-affine morphism of finite type \cite[\S{I.2},~Lemme~2]{LafforgueL:Ramanujan}. Therefore, it suffices to show that the composition of finite-type morphisms
\[
\begin{tikzcd}
    (\Sht^{\leqslant\blambda}_{\Dcal^\times,I_\bullet}/a^\ZZ)|_{(X\setminus Y)^I} \arrow[r, "\pr_0"] &
    \Bun_{\Dcal^\times}/a^\ZZ \arrow{r} & \Bun_{\GL_{d^2}}/a^\ZZ
\end{tikzcd}
\]
factors through a quasi-compact open substack of $\Bun_{\GL_{d^2}}/a^\ZZ$, where $Y\subseteq \Ram(D)$ is a subset satisfying Assumption~\ref{ass:quasi-compact-Sht-D-setting} for $D$. Or equivalently, for any integer $0\leq e < d^2\deg(a)$ the image of $(\Sht^{\leqslant\blambda, \deg=e}_{\Dcal^\times,I_\bullet})|_{(X\setminus Y)^I}$ factors through some quasi-compact open substack of $\Bun_{\GL_{d^2}}^{\deg = e}$. 

Let $(\Eucal,\alpha)\in\Sht^{\leqslant\blambda,\deg = e}_{\Dcal^\times,I_\bullet}(k)$ be a $\Dcal^\times$-shtuka over an algebraically closed extension $k$ of $\FF_q$, and let $\mu_1(\Ecal_0)>\cdots>\,\mu_t(\Ecal_0)$ be the Harder--Narasimhan slopes for the underlying vector bundle of $\Ecal_0$. If we have $\mu_i(\Ecal_0)-\mu_{i+1}(\Ecal_0)>R$ for some $i$ (where $R$ is the constant defined in Lemma~\ref{lem:bounding-HN-parameter-of-Sht-D}), then Lemma~\ref{lem:bounding-HN-parameter-of-Sht-D} implies that the $i$th Harder--Narasimhan filtration defines a $\Dcal$-stable subshtuka of $(\Eucal,\alpha)$. Therefore, if the legs of $(\Eucal,\alpha)$ are disjoint from $Y\subset\Ram(D)$ satisfying Assumption~\ref{ass:quasi-compact-Sht-D-setting}, then the gaps between consecutive Harder--Narasimhan slopes of $\Ecal_0$ are bounded by $R$ by irreducibility of $(\Eucal,\alpha)$ (\emph{cf.} Lemma~\ref{lem:irreducible-shtuka}). In particular, the image of $\Ecal_0$ in $\Bun_{\GL_{d^2}}^{\deg = e}$ lies in some \emph{Harder--Narasimhan truncation}, which is a quasi-compact open substack of $\Bun_{\GL_{d^2}}^{\deg = e}$; \emph{cf.} \cite[Theorem~2.1]{Schieder:HNstr}.
\end{proof}
\begin{rmk}\label{rmk:quasi-compactness-optimality}
    Our proof of Proposition~\ref{prop:quasi-compactness-of-Sht-D-mod-a}(\ref{prop:quasi-compactness-of-Sht-D-mod-a:leg-restriction}) with $Y = \Ram(D)$ is alternative to \cite[Proposition~1.11]{Lau:Degeneration}. The proof in \emph{loc.~cit.} is based on the special case of Lemma~\ref{lem:irreducible-shtuka} (namely, \cite[\S{I.3},~Proposition~7]{LafforgueL:Ramanujan}), which was deduced from a \emph{global} Weil descent result \cite[\S{I.3},~Lemme~3]{LafforgueL:Ramanujan}. In contrast, our proof of Lemma~\ref{lem:irreducible-shtuka} is based on the analogue of \emph{loc.~cit.} for \'etale local shtukas -- namely, \cite[Proposition~3.4]{ArastehRad-Hartl:LocGlShtuka} -- thus it yields more general irreducibility and quasi-compactness results.

    That said, Lemma~\ref{lem:irreducible-shtuka} is not the optimal irreducibility result. In fact, in the proof of Theorem~\ref{thm:from-inequality-to-properness}, we show irreducibility for $\Dcal^\times$-shtukas over $\cl\FF_q$ under a different condition from Lemma~\ref{lem:irreducible-shtuka}, and thereby obtain a quasi-compactness result for $\Sht^{\leqslant\blambda}_{\Dcal^\times,I_\bullet}/a^\ZZ$ that does not seem to be covered by Proposition~\ref{prop:quasi-compactness-of-Sht-D-mod-a}.
%    Specifically, the proof of Lemma~\ref{lem:irreducible-shtuka} can be adapted to show the irreducibility of a $\Dcal^\times$-shtuka $\Eucal$ over $k = \cl k$ whose legs are disjoint from $X\setminus U$. This irreducibility result was originally proven in \cite[\S{I.3},~Proposition~7]{LafforgueL:Ramanujan}, which relies on a global descent lemma \cite[\S{I.3},~Lemme~3]{LafforgueL:Ramanujan}. In contrast, the proof of Lemma~\ref{lem:irreducible-shtuka} relies on the analogue of \emph{loc.~cit.} for \'etale local shtukas; namely, \cite[Proposition~3.4]{ArastehRad-Hartl:LocGlShtuka}.
\end{rmk}

\section{Kottwitz--Rapoport and Newton stratifications}\label{sec-stratifications}
We digress to review the Newton stratification for the moduli of $\Gcal$-shtukas and establish Mazur's inequality  under a mild assumption (Proposition~\ref{prop:Adm}). We also discuss the Kottwitz--Rapoport stratification and prove non-emptiness results for both stratifications when $G$ is an inner form of $\GL_d$ (Corollaries~\ref{cor:nonempty-KR}, \ref{cor:nonempty-Newton}), which are of independent interest. This in particular shows that each fibre of the leg morphism $\wp\colon\Sht^{\leqslant\blambda}_{\Gcal,I_\bullet}\to X^I$ is non-empty if $\Sht^{\leqslant\blambda}_{\Gcal,I_\bullet}\ne\emptyset$ and $\Gcal$ is a parahoric integral model of an inner form of $\GL_d$.
Although these results are not used directly in the proof of our main theorem (Theorem~\ref{thm:from-inequality-to-properness}), we note a parallel between the proofs of Proposition~\ref{prop:Adm} and Lemma~\ref{lem:adm-lattice}, suggesting that Lemma~\ref{lem:adm-lattice} may be viewed as a ``degeneration'' of Mazur's inequality.

\subsection{``Special fibres'' of Beilinson--Drinfeld Schubert varieties}\label{ssec:fixed-legs}
    If a geometric point $\cl x\in X(\cl\FF_q)$ lies over a closed point $x\in X$, then we write $\cl x \mid x$.

    Let $Y\subset X$ be a finite set of closed points. We choose pairwise disjoint non-empty subsets $I_y\subset I$ for each $y\in Y$, and partition each $I_y$ as $I_y = \bigsqcup_{\cl y \mid y}I_{\cl y}$, where $I_{\cl y}$ is allowed to be empty. Given such $(I_{\cl y})_{\cl y\mid y\in Y}$, we obtain %the following tuple of geometric points
    \begin{equation}\label{eq:ybf}
        \cl \ybf \coloneqq (x_i)_{i\in \bigsqcup_{y\in Y}I_y} \in X(\cl\FF_q)^{\bigsqcup_{y\in Y}I_y},\qquad \text{where }x_i = \cl y \text{ if }i\in I_{\cl y}.
    \end{equation}

    Set $J\coloneqq I\setminus \big(\bigsqcup_{y\in Y}I_y\big)$, and write $(X\setminus Y)^J\times \cl\ybf$ for the following scheme over $X^I$
    \begin{equation}
        (X\setminus Y)^J\times \cl\ybf \coloneqq\begin{tikzcd}[column sep = large]
         (X\setminus Y)^J\times\Spec\cl\FF_q \arrow[r, "{(\incl,\cl\ybf)}" ] & X^J \times  X^{\bigsqcup_{y\in Y}I_y} \cong  X^I
    \end{tikzcd}.
    \end{equation}
    Now, define $\Gr_{\Gcal,I_\bullet,\cl\ybf}^{\leqslant\blambda}$ by the following $2$-cartesian diagram
    \begin{equation}\label{eq:Gr-ybf}
        \begin{tikzcd}
            \Gr_{\Gcal,I_\bullet,\cl\ybf}^{\leqslant\blambda} \arrow[r] \arrow[d]& \Gr_{\Gcal,I_\bullet}^{\leqslant\blambda} \arrow[d] \\
            % (X\setminus Y)^{J}
            (X\setminus Y)^{J}
            \times\cl\ybf \arrow[r] & X^I
        \end{tikzcd},
    \end{equation}
    and similarly define $\Sht^{\leqslant\blambda}_{\Gcal,I_\bullet,\cl\ybf}$ and  $\Hck^{\leqslant\blambda}_{\Gcal,I_\bullet,\cl\ybf}$.

    Given $I_\bullet$, let $I_{\cl y,\bullet}$ and $J_\bullet$ be the partitions induced by $I_\bullet$; namely, $I_{\cl y,\bullet}$ is the partition defined by $I_{\cl y}\cap I_j$, skipping the index $j$ if $I_{\cl y}\cap I_j = \emptyset$. We also define $\blambda_J \in X_\ast(T)^J$ and $\blambda_{\cl y}\in X_\ast(T)^{I_{\cl y}}$ to be the sub-tuples of $\blambda$ corresponding to $J$ and $I_{\cl y}$, respectively.
    Then we can decompose $\Gr^{\leqslant\blambda}_{\Gcal,I_\bullet,\cl\ybf}$ and $\Lcal^+_{X^I}\Gcal|_{(X\setminus Y)^J\times\cl\ybf}$ as follows:
    \begin{align}
        \label{eq:Gr-fixed-legs}
        \Gr^{\leqslant\blambda}_{\Gcal,I_\bullet,\cl\ybf} &\cong  \Gr^{\leqslant\blambda_J}_{\Gcal,J_\bullet}|_{(X\setminus Y)^{J}} \times \prod_{\cl y \text{ s.t. } I_{\cl y}\ne\emptyset}(\Gr^{\leqslant \blambda_{\cl y}}_{\Gcal,I_{\cl y,\bullet}})_{\Delta(\cl y)}, \quad\text{and}\\
        \label{eq:loop-gp-fixed-legs}
        \Lcal^+_{X^I}\Gcal|_{(X\setminus Y)^J\times\cl\ybf} & \cong  \Lcal^+_{X^J}\Gcal|_{(X\setminus Y)^J} \times \prod_{\cl y \text{ s.t. } I_{\cl y}\ne\emptyset} (L^+_y\Gcal)_{\cl\FF_q},
    \end{align}
    where $\Delta(\cl y)\coloneqq (\cl y,\cdots,\cl y)\in X(\cl\FF_q)^{I_{\cl y}}$. 
    In particular, the pro-smooth morphism of stacks $\Sht^{\leqslant\blambda}_{\Gcal, I_\bullet} \to [\Lcal^+_{X^I}\Gcal \backslash \Gr^{\leqslant\blambda}_{\Gcal,I_\bullet}]$ (\emph{cf.} \cite[Proposition~3.4.2]{Bieker:IntModels}) restricts to 
    \begin{equation}\label{eq:loc-mod}
     \begin{tikzcd}[column sep = tiny]
       {\Sht^{\leqslant\blambda}_{\Gcal, I_\bullet,\cl \ybf}} \arrow[r] & {\left.\big[\Lcal^+_{X^J}\Gcal \backslash \Gr^{\leqslant\blambda_J}_{\Gcal,J_\bullet} \big]\right|_{(X\setminus Y)^J} \times \prod\limits_{\cl y \text{ s.t. } I_{\cl y}\ne\emptyset} \big[L^+_y\Gcal\backslash (\Gr^{\leqslant \blambda_{\cl y}}_{\Gcal,I_{\cl y,\bullet}})_{\Delta(\cl y)}\big].}
   \end{tikzcd}
    \end{equation}
    Therefore, to analyse $\Gr^{\leqslant\blambda}_{\Gcal,I_\bullet,\cl\ybf}$, we may focus on each individual $(\Gr^{\leqslant \blambda_{\cl y}}_{\Gcal,I_{\cl y,\bullet}})_{\Delta(\cl y)}$ endowed with $L^+_y\Gcal$-action.
    
    Suppose that $I = I_{\cl y}$ and $\cl\ybf=\Delta(\cl y)$ for the moment.
    If $|I|=1$, then we have 
    \begin{equation}
         (\Gr^{\leqslant \lambda}_{\Gcal,X})_{\cl y}\subset(\Gr_{\Gcal,X})_{\cl y} \cong (L_y G/L^+_y\Gcal)_{\cl\FF_q}.
    \end{equation}
    Let us make its underlying reduced scheme more explicit.  %For any geometric point $\cl y\in X(\cl\FF_q)$, write $\Ocal_{\cl y}\coloneqq \Ocal_{\breve X,\cl y}$ and $F_{\cl y}\coloneqq \Frac\Ocal_{\cl y}$.  And recall the Iwahori--Cartan decomposition (\emph{cf.} \cite[Proposition~8]{HainesRapoport:Parahoric})
    By the Iwahori--Cartan decomposition (\emph{cf.} \cite[Proposition~8]{HainesRapoport:Parahoric}), we have
    \begin{equation}\label{eq:Iwahori-Cartan}
        \Gcal(\Ocal_{\cl y}) \backslash G(F_{\cl y}) / \Gcal(\Ocal_{\cl y}) \cong W_{\Gcal_{\cl y}} \backslash \widetilde W_{G_{\cl y}} /W_{\Gcal_{\cl y}} ,
    \end{equation}
    where $\widetilde W_{G_{\cl y}}$ is the Iwahori--Weyl group for $G_{\cl y}\coloneqq G_{F_{\cl y}}$ and $W_{\Gcal_{\cl y}}\subset \widetilde W_{G_{\cl y}}$ is the finite subgroup corresponding to  $\Gcal_{\cl y}\coloneqq \Gcal_{\Ocal_{\cl y}}$. For $\widetilde w\in \widetilde W_{G_{\cl y}}$, we let $\Gcal(\Ocal_{\cl y})\widetilde w\Gcal(\Ocal_{\cl y})$ denote the double coset corresponding to the $W_{\Gcal_{\cl y}}$-double coset of $\widetilde w$ via \eqref{eq:Iwahori-Cartan}. Then, by the description of the reduced fibre $(\Gr^{\leqslant\lambda}_{\Gcal,X})_{\cl y,\red}$ (\emph{cf.} \cite[Theorem~6.12]{HainesRicharz:TestFtnParahoric}, which is valid for any $G_y$), we have
     \begin{equation}\label{eq:gl-Sch-var}
        (\Gr^{\leqslant\lambda}_{\Gcal,X})_{\cl y}(\cl\FF_q)=\bigcup_{\widetilde w \in \Adm_{\Gcal,\cl y}(\lambda)} \Gcal(\Ocal_{\cl y}) \widetilde w^{-1} \Gcal(\Ocal_{\cl y})/\Gcal(\Ocal_{\cl y});
    \end{equation}
    where $\Adm_{\Gcal,\cl y}(\lambda)\subset W_{\Gcal_{\cl y}} \backslash \widetilde W_{G_{\cl y}} /W_{\Gcal_{\cl y}}$ is the double coset of  $\lambda$-admissible elements \cite[(3.6)]{KottwitzRapoport:Existence}. (On the right side of \eqref{eq:gl-Sch-var}, we take the double coset corresponding to $\widetilde w^{-1}$ for $\widetilde w\in\Adm_{\Gcal,\cl y}(\lambda)$, due to our convention on $\Hck_{\Gcal, X}$ and $\Gr^{\leqslant\lambda}_{\Gcal,X}$; \emph{cf.} \eqref{eq:t-lambda}. Though $\widetilde w^{-1}\in\Adm_{\Gcal,\cl y}(-\lambda)$, we would like $\Adm_{\Gcal,\cl y}(\lambda)$ to be the index set; see Proposition~\ref{prop:Adm} and its proof to justify this sign convention.)
%    Here, we abusively let $-\lambda$ also denote its dominant representative; note our convention on $\Gr^{\leqslant\lambda}_{G,X}$ in \eqref{eq:t-lambda}.) %See \cite[Theorem~6.12]{HainesRicharz:TestFtnParahoric}, which describes the reduced fibre $(\Gr^{\leqslant\lambda}_{\Gcal,X})_{\cl y,\red}$ with no assumption on $G_y$. 

    More generally, if $|I| \geq 1$ and $I_\bullet = (I_j)_{j=1}^r$, then we have 
    \begin{equation}\label{eq:convol-prod-unbdd}
       \Gr_{\Gcal,I_\bullet,\Delta(\cl y)} \cong (\underbrace{L_yG\times^{L^+_y\Gcal}  \cdots \times^{L^+_y\Gcal}   L_yG \times^{L^+_y\Gcal} L_yG}_{r \text{ times} } /L^+_y \Gcal)_{\cl\FF_q}.
    \end{equation}
    To see this, by definition we have $\Gr_{\Gcal,(I),\Delta(\cl y)} \cong (L_yG/L^+_y\Gcal)_{\cl\FF_q}$, and now we iterate the \emph{convolution products} of $(\Gr_{\Gcal,X})_{\cl y}$; see \cite[Proof of Lemma~6.3]{Zhu:CoherenceConj} for the case when $r=2$. If $I_\bullet = (\{1\},\cdots,\{r\})$, then we can deduce the following description of $(\Gr_{\Gcal,I_\bullet}^{\leqslant\blambda})_{\Delta(\cl y)}$ from  \eqref{eq:gl-Sch-var} and \eqref{eq:convol-prod-unbdd}:
    \begin{equation}\label{eq:convol-prod-bdd}
       (\Gr_{\Gcal,I_\bullet}^{\leqslant\blambda})_{\Delta(\cl y)} \cong (\widetilde\Gr^{\leqslant\lambda_r}_{\Gcal,X})_{\cl y} \times^{L^+_y\Gcal}  \cdots \times^{L^+_y\Gcal}   (\widetilde\Gr^{\leqslant\lambda_2}_{\Gcal,X})_{\cl y} \times^{L^+_y\Gcal} (\Gr^{\leqslant\lambda_1}_{\Gcal,X})_{\cl y},
    \end{equation}
    where $(\widetilde\Gr^{\leqslant\lambda_i}_{\Gcal,X})_{\cl y}\subset L_yG$ is the preimage of $(\Gr^{\leqslant\lambda_i}_{\Gcal,X})_{\cl y}$. Note that we have
    \begin{equation}\label{eq:gl-Sch-var-lift}
        (\widetilde\Gr^{\leqslant\lambda_i}_{\Gcal,X})_{\cl y}(\cl\FF_q) = \bigcup_{\widetilde w \in \Adm_{\Gcal,\cl y}(\lambda_i)} \Gcal(\Ocal_{\cl y}) \widetilde w^{-1} \Gcal(\Ocal_{\cl y}) \subset G(F_{\cl y}).
    \end{equation}
    This description can be generalised for any $I_\bullet$ as follows.
    \begin{prop}\label{prop:Adm-additive}
        Choose $I_\bullet = (I_j)_{j=1}^r$, and let $I'_\bullet$ denote the refinement of $I_\bullet$ into singletons.
        For each $j=1,\cdots,r$, set $\lambda^{(j)}\coloneqq\sum_{i\in I_j}\lambda_i$. \emph{Suppose} that $G_y$ splits over a tame extension of $F_y$. Then the underlying reduced scheme of $(\Gr_{\Gcal,I_\bullet}^{\leqslant\blambda})_{\Delta(\cl y)}$ coincides with the reduced image of $(\Gr_{\Gcal,I'_\bullet}^{\leqslant\blambda})_{\Delta(\cl y)}$ under the ``joining modifications'' map (\emph{cf.} Corollary~\ref{cor:joining-modifications}), which in turn coincides with the underlying reduced scheme of
        \[(\widetilde\Gr^{\leqslant\lambda^{(r)}}_{\Gcal,X})_{\cl y} \times^{L^+_y\Gcal}  \cdots \times^{L^+_y\Gcal}   (\widetilde\Gr^{\leqslant\lambda^{(2)}}_{\Gcal,X})_{\cl y} \times^{L^+_y\Gcal} (\Gr^{\leqslant\lambda^{(1)}}_{\Gcal,X})_{\cl y}.\]
    \end{prop}
    \begin{proof}
        Since the ``joining modification'' map $\Gr^{\leqslant\blambda}_{\Gcal,I'_\bullet}\to\Gr^{\leqslant\blambda}_{\Gcal,I_\bullet}$ is projective and restricts to the identity map on the dense subset $\prod_{i\in I}\Scal^{\lambda_i}$, it is surjective. Thus, it suffices to compare  $(\Gr_{\Gcal,I_\bullet}^{\leqslant\blambda})_{\Delta(\cl y),\red}$ and the reduced image of $(\Gr_{\Gcal,I'_\bullet}^{\leqslant\blambda})_{\Delta(\cl y)}$ in $(\Gr_{\Gcal,I_\bullet})_{\Delta(\cl y)}$. To show this, it suffices to handle the case where  $I_\bullet = (\{1,2\})$ and $I'_\bullet = (\{1\},\{2\})$, which is implicitly done in the proof of Proposition~6.4 in \cite{Zhu:CoherenceConj} when $G_y$ splits over a tame extension of $F_y$, as noted in \cite[Proof of Theorem~5.1]{He:KottwitzRapoportConj}.
    \end{proof}
    
   \begin{rmk}
        If $|I|=2$, Proposition~\ref{prop:Adm-additive} reduces to the following equality
        \begin{equation}\label{eq:Adm-additive}
        \bigcup_{\substack{\widetilde w_1\in\Adm_{\Gcal,\cl y}(\lambda_1)\\\widetilde w_2\in\Adm_{\Gcal,\cl y}(\lambda_2)}}\frac{\Gcal(\Ocal_{\cl y}) \widetilde w_1^{-1} \Gcal(\Ocal_{\cl y})\widetilde w_2^{-1} \Gcal(\Ocal_{\cl y})}{\Gcal(\Ocal_{\cl y})} =  \bigcup_{\widetilde w \in \Adm_{\Gcal,\cl y}(\lambda)} \frac{\Gcal(\Ocal_{\cl y}) \widetilde w^{-1} \Gcal(\Ocal_{\cl y})}{\Gcal(\Ocal_{\cl y})},
    \end{equation}
    where $\lambda = \lambda_1+\lambda_2$.
       We expect Proposition~\ref{prop:Adm-additive} to hold without any condition on $G_y$, but we do not attempt to optimise Proposition~\ref{prop:Adm} in this paper. % since we will apply it to $G_y = B_y^\times$, which splits over an unramified extension of $F_y$.
   \end{rmk}

    Proposition~\ref{prop:Adm-additive} suggests the following definition of Kottwitz--Rapoport strata.
    \begin{defn}
        Suppose that $\Gcal_y$ is an Iwahori integral model for any $y\in Y$, and we have $I_{\cl y,\bullet} = (I_{\cl y})$ for each geometric point $\cl y$ over $y$ with $I_{\cl y}\ne\emptyset$. This implies that the $L^+_y\Gcal$-orbits in $(\Gr^{\leqslant\blambda_{\cl y}}_{\Gcal, (I_{\cl y})})_{\Delta(\cl y),\red}$ are indexed by $\Adm_{\Gcal,\cl y}(\lambda_{\cl y})\subset \widetilde W_{G_{\cl y}}$ where $\lambda_{\cl y} \coloneqq \sum_{i\in I_{\overline{y}}} \lambda_i$; namely, $\widetilde w_{\cl y}\in \Adm_{\Gcal,\cl y}(\lambda_{\cl y})$ corresponds to the locally closed (reduced) orbit $\Scal^{\widetilde w_{\cl y}}$ such that
        \begin{equation}
            \Scal^{\widetilde w_{\cl y}}(\cl\FF_q) = \Gcal(\Ocal_{\cl y})\widetilde w_{\cl y}^{-1}\Gcal(\Ocal_{\cl y})/\Gcal(\Ocal_{\cl y}).
        \end{equation}
        Given $\widetilde\wbf\coloneqq (\widetilde w_{\cl y})$ where $\widetilde w_{\cl y}\in\Adm_{\Gcal,\cl y}(\lambda_{\cl y})$ for each $\cl y$ with $I_{\cl y}\ne\emptyset$, %we define the \emph{Kottwitz--Rapoport stratum} $\Sht^{\leqslant\blambda_J,\widetilde\wbf}_{\Gcal,I_\bullet,\cl\ybf}$ by 
        we define $\Sht^{\leqslant\blambda_J,\widetilde\wbf}_{\Gcal,I_\bullet,\cl\ybf}$ to be the preimage of 
        \[
        \left.\big[\Lcal^+_{X^J}\Gcal \backslash \Gr^{\leqslant\blambda_J}_{\Gcal,J_\bullet} \big]\right|_{(X\setminus Y)^J} \times \prod_{\cl y \text{ s.t. } I_{\cl y}\ne\emptyset}\big[L^+_y\Gcal\backslash \Scal^{\widetilde w_{\cl y}}\big]
        \]
        via the local model diagram \eqref{eq:loc-mod}. The closure relation among strata $\left(\Sht^{\leqslant\blambda_J,\widetilde\wbf}_{\Gcal,I_\bullet,\cl\ybf}\right)_{\widetilde\wbf}$ is governed by the Bruhat order on $\prod_{\cl y}\Adm_{\Gcal,\cl y}(\lambda_{\cl y})$ in the usual way. In particular, the tuple $\widetilde\wbf_{\bsc} = (\widetilde w_{\cl y,\bsc})$ of the unique (minimal) length-$0$ elements $\widetilde w_{\cl y,\bsc}\in\Adm_{\Gcal,\cl y}(\lambda_{\cl y})$
        corresponds to the unique closed cell $\Sht^{\leqslant\blambda_J,\widetilde\wbf_{\bsc} }_{\Gcal,I_\bullet,\cl\ybf}\subset\Sht^{\leqslant\blambda}_{\Gcal,I_\bullet,\cl\ybf}$.%; namely, each $\widetilde w_{\cl y,\bsc}\in\Adm_{\Gcal,\cl y}(\lambda_{\cl y})$ is the unique (minimal) length-$0$ element.

        If $J=\emptyset$, then we write $\Sht^{\widetilde\wbf}_{\Gcal,I_\bullet,\cl\ybf}$ for $\Sht^{\leqslant\blambda_J,\widetilde\wbf}_{\Gcal,I_\bullet,\cl\ybf}$, and call it the \emph{Kottwitz--Rapoport stratum} for $\widetilde\wbf$. %%By the theory of local models, $\Sht^{\leqslant\blambda_J,\widetilde\wbf}_{\Gcal,I_\bullet,\cl\ybf}$ is smooth of dimension $\sum_{\cl y}\ell(\widetilde w_{\cl y})$. 
        The unique closed stratum $\Sht^{\widetilde\wbf_{\bsc}}_{\Gcal,I_\bullet,\cl\ybf}$ is called the \emph{basic Kottwitz--Rapoport stratum}.
%        
%        we associate a locally closed substack $\Sht^{\leqslant\blambda_J,\widetilde\wbf}_{\Gcal,I_\bullet,\cl\ybf} \subset  \Sht_{\Gcal,I_\bullet,\cl\ybf}^{\leqslant\blambda}$ via the following 2-cartesian diagram
%        \begin{equation}\label{eq:KR}
%        \begin{tikzcd}[column sep = large]
%            \Sht^{\leqslant\blambda_J,\widetilde\wbf}_{\Gcal,I_\bullet,\cl\ybf}\arrow[r] \arrow[d, hookrightarrow]&  \prod\limits_{\cl y \text{ s.t. } I_{\cl y}\ne\emptyset} {\big[L^+_y\Gcal\backslash \Scal^{\widetilde w_{\cl y}}\big]} \arrow[d, hookrightarrow] \\
%            % (X\setminus Y)^{J}
%            \Sht_{\Gcal,I_\bullet,\cl\ybf}^{\leqslant\blambda} \arrow[r] & \prod\limits_{\cl y \text{ s.t. } I_{\cl y}\ne\emptyset} {\big[L^+_y\Gcal\backslash (\Gr^{\leqslant \blambda_{\cl y}}_{\Gcal,(I_{\cl y})})_{\Delta(\cl y)}\big]}
%        \end{tikzcd},
%        \end{equation}
%        where the horizontal arrows are induced from \eqref{eq:loc-mod}.
%
%        Recall that each $\Adm_{\Gcal,\cl y}(-\lambda_{\cl y})$ has a unique length-$0$ element, which we denote as $\tau_{\cl y}$. For $\widetilde\wbf_{\bsc}\coloneqq (\tau_{\cl y})$, we call $\Sht^{\leqslant\blambda_J,\widetilde\wbf}_{\Gcal,I_\bullet,\cl\ybf}$ the \emph{basic Kottwitz--Rapoport stratum}.
    \end{defn}

    \begin{rmk}
        It is often convenient to view $\widetilde\wbf$ as a tuple $(\widetilde w_{\cl y})$ over all geometric points $\cl y$ over $y\in Y$, by setting $\widetilde w_{\cl y}=1$ whenever $I_{\cl y}=\emptyset$.
    \end{rmk}

    Next, we review  \emph{(basic) Newton strata}, and discuss the relation with $\Sht^{\leqslant\blambda_J,\widetilde\wbf}_{\Gcal,I_\bullet,\cl\ybf}$.
\begin{constr}\label{constr:loc-G-isosht}
We use the notation in \S\ref{ssec:fixed-legs}.
Let $S$ be an $\cl\FF_q$-scheme, and fix %a $\Gcal$-shtuka
\[    ((x_i)_{i\in I}, (\Pcal_0,\Pcal_1), \varphi_1,\alpha)\in \Sht_{\Gcal,(I),\cl\ybf}^{\leqslant\blambda}(S).\]
Using the notation from \S\ref{ssect:G-Bun} and the proof of Lemma~\ref{lem:irreducible-shtuka}, the construction of \cite[\S3.6,\S5.2]{HamacherKim:Igusa} yields, for each $y\in Y$, a local $G_y$-isoshtuka $(\Lcal_y\Pcal_0,\varphi_y)$ in the sense of \cite[Definition~2.2]{HamacherKim:Igusa}; here, $\Lcal_y\Pcal_0$ is the $L_yG$-torsor  associated to $\Pcal_0|_{\Spf\Ocal_y\widehat\times_{\kappa_y}S}$, where $S$ is viewed as a $\kappa_y$-scheme via some chosen geometric point $\cl y$ over $y$, and
\begin{equation}\label{eq:loc-G-isosht}
    \varphi_y\coloneqq ( (\varphi_1^{-1}\circ \alpha)^{\deg(y)} )_{\cl y} \colon
    \begin{tikzcd}
        {}^{\tau_y}(\Lcal_y\Pcal_0) \arrow[r,"\cong"] &\Lcal_y\Pcal_0
    \end{tikzcd}.
\end{equation}
Note that $(\Lcal_y\Pcal_0,\varphi_y)$ is independent of the choice of $\cl y$ up to isomorphism, and it only depends on $G_y$ but not on the choice of the integral model $\Gcal_y$.

    Now for any $y\in Y$ and $I_\bullet$, we define the universal local $G_y$-isoshtuka
    \begin{equation}\label{eq:univ-loc-isosht}
        (\Lcal_y\Pcal_0^{\univ},\varphi_y^{\univ})
    \end{equation}
     on $\Sht_{\Gcal,I_\bullet,\cl\ybf}^{\leqslant\blambda}$, as follows. If $I_\bullet = (I)$, then the universal $\Gcal$-shtuka $(\Pucal^{\univ},\alpha^{\univ})$ induces a local $G_y$-isoshtuka  on $\Sht_{\Gcal,(I),\cl\ybf}^{\leqslant\blambda}$ by the recipe in \eqref{eq:loc-G-isosht}. %(This construction can be applied since the universal $i$th leg $\wp_i\colon\Sht_{\Gcal,I_\bullet,\cl\ybf}^{\leqslant\blambda}\to \breve X\to X$ factors through $y$ when $i\in I_y$, and it is disjoint from $y$ otherwise.) 
     In general, we define $ (\Lcal_y\Pcal_0^{\univ},\varphi_y^{\univ})$ to be the pullback of the universal $G_y$-isoshtuka by the ``joining modification'' map $\Sht_{\Gcal,I_\bullet,\cl\ybf}^{\leqslant\blambda} \to \Sht_{\Gcal,(I),\cl\ybf}^{\leqslant\blambda}$ in Corollary~\ref{cor:joining-modifications}.
\end{constr}

Fix a closed point $y$ of $X$. Given $b_y\in G(\breve F_y)$, we define a local $G_y$-isoshtuka $(P_y^0,\varphi_{b_y})$ over $\cl\FF_q$, where $P_y^0 = (L_yG)_{\cl\FF_q}$ and 
\begin{equation}
    \varphi_{b_y}\colon \begin{tikzcd}
        {}^{\tau_y}P_y^0 \cong (L_yG)_{\cl\FF_q} \arrow[r, "b_y"]& (L_yG)_{\cl\FF_q} = P_y^0
    \end{tikzcd},
\end{equation}
where the map is the left multiplication by $b_y$. Clearly, $(P_y^0,\varphi_{b_y})$ and $(P_y^0,\varphi_{b'_y})$ are isomorphic if and only if $b_y$ and $b'_y$ are $\tau_y$-conjugate in $G(\breve F_y)$. 

As usual, let $B(G_y)$ denote the set of $\tau_y$-conjugacy classes in $G(\breve F_y)$. Given $b_y\in G(\breve F_y)$, let $[b_y]\in  B(G_y)$ denote its $\tau_y$-conjugacy class.  For any algebraically closed extension $k$ of $\cl\FF_q$, any local $G_y$-isoshtuka over $k$ is isomorphic to the pullback of $(P_y^0,\varphi_{b_y})$ for a unique $b_y$ up to $\tau_y$-conjugacy; \emph{cf.} \cite[Theorem~1.1]{RapoportRichartz:Gisoc}.

Given a local $G_y$-isoshtuka $(P_y, \varphi_{y})$ over an $\FF_q$-scheme $S$, we get a (possibly empty) locally closed reduced subscheme $S^{[b_y]}\subset S$, called the \emph{Newton stratum} for $[b_y]$, which is characterised as follows: a geometric point $s \in S(\cl k)$ lies in $S^{[b_y]}$ if and only if the fibre of $(P_y,\varphi_y)$ at $s$ is isomorphic to $(P_y^0,\varphi_{b_y})$. (The proof of \cite[Theorem~3.6]{RapoportRichartz:Gisoc} works in the equi-characteristic setting without change.) One can define Newton strata even when the base is an algebraic stack by descending the Newton strata on charts.

\begin{defn}\label{def:Newton-Strata}
Fix a finite subset $Y\subset X$, $(I_y)_{y\in Y}$ and $\cl\ybf$ as in \S\ref{ssec:fixed-legs}, and let $x_i$ be as in \eqref{eq:ybf}. Choose $[b_y]\in B(G_y)$ for each $y\in Y$, and set $\bbf_Y\coloneqq ([b_y])_{y\in Y}$. Then, the \emph{Newton stratum} for $\bbf_Y$ is a (possibly empty) reduced intersection of the Newton strata for $[b_y]$; namely,
\[
    \Sht_{\Gcal,I_\bullet,\cl\ybf}^{\leqslant\blambda,\bbf_Y} \coloneqq\bigcap_{y\in Y} \Sht_{\Gcal,I_\bullet,\cl\ybf}^{\leqslant\blambda,[b_y]}.
\]
\end{defn}

We want to show that the Newton strata are indexed by certain tuples  of \emph{admissible $G_y$-isocrystals}. To explain, for $\lambda\in X_\ast(T)_+$ and $a\in\ZZ$, we let ${}^{\tau^a}\lambda\in X_\ast(T)_+$ denote the unique dominant representative of the conjugacy class of the $G$-valued cocharacter  ${}^{\tau^a}\lambda$. If the reflex field of $\lambda$ is $F$, then we have ${}^{\tau^a}\lambda = \lambda$ for any $a\in \ZZ$.
%The following results give a necessary condition for the non-emptiness of $\Sht_{\Gcal,I_\bullet,\cl\ybf}^{\leqslant\blambda,\bbf_Y}$.
\begin{prop}\label{prop:Adm}
Suppose that $G_y$ splits over a \emph{tame} extension of $F_y$ for each $y\in Y$. Assume further that $\Sht_{\Gcal,I_\bullet,\cl\ybf}^{\leqslant\blambda,\bbf_Y}$ is non-empty. 
Fix a geometric point $\cl y\in X(\cl{\FF}_q)$ over each $y\in Y$, and set
    \[\lambda_y\coloneqq\sum_{a=0}^{\deg(y)-1}\sum_{i\text{ st }\tau^{-a}(\cl y) = x_i}{}^{\tau^{a}}\lambda_i\in X_\ast(T)_+.\] 
Then we have $[b_y]\in B(G_y,\lambda_y)$ for every $y\in Y$, where $B(G_y,\lambda_y)$ is the set of $\lambda_y$-admissible elements in $B(G_y)$ defined by Kottwitz \cite[\S6]{Kottwitz:Gisoc2}.
\end{prop}
When $|I_y|=1$ for each $y\in Y$, we may take $\cl y = x_i$ and $\lambda_y = \lambda_i$ for the unique $i\in I_y$. Then, the proposition is a simple consequence of the group-theoretic formulation of Mazur's inequality  \cite[Theorem~A]{He:KottwitzRapoportConj}. We generalise this observation to the case when there are two distinct legs supported on the same closed point.

We assume that $G_y$ splits over a tame extension of $F_y$ so that we can apply Proposition~\ref{prop:Adm-additive}, but we expect this statement to be valid for any $G_y$.
\begin{proof}
%    Note that the Newton stratum $\Sht_{\Gcal,I_\bullet,\cl\ybf}^{\leqslant\blambda,\bbf_Y}$ is non-empty for \emph{any} choice of $I_\bullet$ if it is non-empty for \emph{one} choice of $I_\bullet$, so without loss of generality we may may choose $I_\bullet$ so that $I_j = \{j\}$ for any $j$. %, and assume that $\Sht_{\Gcal,I_\bullet,\cl\ybf}^{\leqslant\blambda,\bbf_Y}$ is non-empty for this choice of $I_\bullet$.  Being locally of finite type over $\cl\FF_q$, we may choose \[(\Pucal,\alpha) = \big( (\bar x_i), (\Pcal_j), (\varphi_j), \alpha \big)  \in\Sht_{\Gcal,I_\bullet,\cl\ybf}^{\leqslant\blambda,\bbf_Y}(\cl\FF_q).\]
%    
%    %For any geometric point $\cl x\in\breve X(\cl\FF_q)$, we write $\Ocal_{\cl x}\coloneqq \Ocal_{\breve X,\cl x}$ and $F_{\cl x}\coloneqq \Frac\Ocal_{\cl x}$. 
%    We use the notation from \S\ref{ssec:fixed-legs}. Choose $y\in Y$ and a geometric point $\cl y$ over $y$. %For any $\tau_y^{-a}(\cl y)$ for $a\in \ZZ/\deg(y)\ZZ$, we get an identification 
%    Recall we have the following natural isomorphism
%    \begin{equation}
%        \breve F\otimes_F F_y \cong \prod_{a\in \ZZ/\deg(y)\ZZ}F_{\tau^{-a}(\cl y)}.
%    \end{equation}

    We may assume without loss of generality that $I_\bullet = (I)$.
    Given a $\Gcal$-shtuka
    \[(\Pucal,\alpha) = \big( (\bar x_i), 
    \begin{tikzcd}%[column sep = small]
        \ltau\Pcal_0 \arrow[r, "\cong"', "\alpha"] & \Pcal_1 & \arrow[l, dashrightarrow,"\varphi_1"']\Pcal_0
    \end{tikzcd}
    \big)  \in\Sht_{\Gcal,(I),\cl\ybf}^{\leqslant\blambda}(\cl\FF_q),\]
    its localisation at $y$ in the sense of \cite[\S3.6]{HamacherKim:Igusa} is given by
    \begin{equation}\label{eq:global-local}
        \widetilde\varphi_y\coloneqq \varphi_1^{-1}\circ\alpha\ \colon \begin{tikzcd}
            (\ltau\Pcal_0)|_{\Spf(\cl\FF_q\widehat\otimes\Ocal_y)} \arrow[r, dashrightarrow ]& \Pcal_0|_{\Spf(\cl\FF_q\widehat\otimes\Ocal_y)}
        \end{tikzcd},
    \end{equation}
    which we view as a local $\Res_{\Ocal_y/\FF_q\pot{\varpi_y}}\Gcal_y$-shtuka over $\cl\FF_q$ (where $\varpi_y\in\Ocal_y$ is a uniformiser). Set $\widetilde\Gcal_y\coloneqq\Res_{\Ocal_y/\FF_q\pot{\varpi_y}}\Gcal_y$ and $\widetilde G_y\coloneqq \Res_{F_y/\FF_q\rpot{\varpi_y}}G_y$.

    We fix a geometric point $\cl y$ over $y$, and write $\cl y_a\coloneqq \tau^{-a}(\cl y)$ for $a=0,\cdots,[\kappa_y:\FF_q]-1$. Then the Iwahori--Cartan decomposition for $\widetilde\Gcal_y$ can be written as follows
    \begin{multline}\label{eq:Iwahori-Cartan-tilde}
        %\breve F\otimes_F F_y \cong \prod_{a=1}^{[\kappa_y:\FF_q]}F_{\cl y_a} \quad\text{and}\quad G(\breve F\otimes_F F_y) \cong \prod_{a=1}^{[\kappa_y:\FF_q]}G(F_{\cl y_a}).
        \widetilde\Gcal_y(\mathcal O_{\cl y})\backslash\widetilde G_y(F_{\cl y})/ \widetilde\Gcal_y(\mathcal O_{\cl y}) \cong\Gcal(\cl\FF_q\widehat\otimes\Ocal_y)\backslash G(\breve F\widehat\otimes_F F_y)/\Gcal(\cl\FF_q\widehat\otimes\Ocal_y) \\
        \cong \prod_{a=0}^{[\kappa_y:\FF_q]-1}\Gcal(\Ocal_{\cl y_a})\backslash G(F_{\cl y_a})/ \Gcal(\Ocal_{\cl y_a}) 
        \cong \prod_{a=0}^{[\kappa_y:\FF_q]-1} W_{\Gcal_{\cl y_a}}\backslash \widetilde W_{G_{\cl y_a}}/W_{\Gcal_{\cl y_a}},
    \end{multline}
    and dominant coweights for $\widetilde G_y$ are given by tuples of dominant coweights for $G_{\cl y_a}$. We set 
    \begin{equation}\label{eq:lambda-tilde}
        \widetilde\lambda_{y}\coloneqq(\lambda_{\cl y_a})_{a=0}^{[\kappa_y:\FF_q]-1} \in \big(X_\ast(T)_+\big)^{[\kappa_y:\FF_q]},
    \end{equation}
    where $\lambda_{\cl y_a}\coloneqq \sum_{i\in I_{\cl y_a}}\lambda_i$ if $I_{\cl y_a}\ne\emptyset$, and $\lambda_{\cl y_a} = \underline0$ if $I_{\cl y_a}=\emptyset$.
    
    Now, choose a trivialisation of $\Pcal_{j}|_{\Spf(\cl\FF_q\widehat\otimes\Ocal_y)}$ for $j=0,1$, and express $\varphi_1$ and $\widetilde\varphi_y$ as the left multiplication by $g_1$ and $\widetilde b_y$  in $G(\breve F\widehat\otimes_F F_y)$, respectively. Then by \eqref{eq:gl-Sch-var} and Proposition~\ref{prop:Adm-additive} (which is valid as $G_y$ splits over a tame extension of $F_y$), the $\Gcal(\cl\FF_q\widehat\otimes\Ocal_y)$-double coset of $g_1^{-1}$ corresponds to an element in $\prod_{a=0}^{[\kappa_y:\FF_q]-1}\Adm_{\Gcal,\cl y_a}(\lambda_{\cl y_a})$. Since $\widetilde b_y\in g_1^{-1}\Gcal(\cl\FF_q\widehat\otimes\Ocal_y)$, the $\tau$-conjugacy class of $\widetilde b_y$ is $\widetilde \lambda_y$-admissible; i.e.,
    \begin{equation}\label{eq:tilde-admissibility}
        [\widetilde b_y] \in B(\widetilde G_y,\widetilde \lambda_y).
    \end{equation}
    Lastly, the Shapiro isomorphism $N\colon B(\widetilde G_y)\riso B(G_y)$ \cite[\S2.3, (2.5)]{HamacherKim:Gisoc} induces
    \begin{equation}\label{eq:Shapiro}
        \begin{tikzcd}
            B(\widetilde G_y,\widetilde\lambda_y) \arrow[r, hookrightarrow, "N"] & B(G_y,\lambda_y)
        \end{tikzcd} ,
    \end{equation}
    where $\lambda_y$ is as in the statement.

    By construction of the local $G_y$-isoshtuka of $(\Pucal,\alpha)$, we have $(\Pucal,\alpha)\in \Sht^{\leqslant\blambda,\bbf_Y}_{\Gcal,(I),\cl\ybf}(\cl\FF_q)$
    if and only if $N([\widetilde b_y]) = [b_y]$ for each $y\in Y$ where $\bbf_Y = ([b_y])_{y\in Y}$. The proposition now follows since $N([\widetilde b_y])\in B(G_y,\lambda_y)$.
\end{proof}
%\begin{rmk}
%    The only step where we need to assume that $G_y$ splits over a tame extension of $F_y$ is Proposition~\ref{prop:Adm-additive}, which is expected to hold for any $G_y$.  
%\end{rmk}

The proof of Proposition~\ref{prop:Adm} can be refined as follows:
\begin{cor}\label{cor:straight-KR}
    In the setting of Proposition~\ref{prop:Adm}, we additionally assume that $\Gcal_y$ is an Iwahori integral model and $I_\bullet = (I)$. 
    We use the notation in the proof of Proposition~\ref{prop:Adm}, and choose $\bbf_Y = ([b_y])_{y\in Y}$ so that for each $y\in Y$ we have $[b_y] = N([\widetilde b_y])$ for some $[\widetilde b_y]\in B(\widetilde G_y,\widetilde\lambda_y)$, where $N$ is the Shapiro map \eqref{eq:Shapiro}. 

    Then the Newton stratum $\Sht^{\leqslant\blambda,\bbf_Y}_{\Gcal,(I),\cl\ybf}$ contains $\Sht^{\leqslant\blambda_J,\widetilde\wbf}_{\Gcal,(I),\cl \ybf}$ for some $\widetilde\wbf$. Furthermore, the basic Kottwitz--Rapoport stratum $\Sht^{\widetilde\wbf_{\bsc}}_{\Gcal,(I),\cl\ybf}$ is contained in the basic Newton stratum $\Sht^{\leqslant\blambda,\bsc}_{\Gcal,(I),\cl\ybf}$.
\end{cor}
\begin{proof}
Recall that the natural $\tau$-action on $G(\breve F\widehat\otimes_F F_y)$ induces an automorphism $\delta$ of $\prod_{a=0}^{[\kappa_y:\FF_q]-1}\widetilde W_{G_{\cl y_a}}$. Let $B(\prod_{a}\widetilde W_{G_{\cl y_a}},\delta)_{\delta-\str}$ denote the set of $\delta$-conjugacy classes of ``$\delta$-straight elements'' in the sense of \cite[\S2.4]{He:GeomADL}. Then we have a natural bijection 
   \begin{equation}\label{eq:He-map}
       \begin{tikzcd}
        B(\prod_{a}\widetilde W_{G_{\cl y_a}},\delta)_{\delta-\str} \arrow[r,"\cong"] & B(\widetilde G_y),
    \end{tikzcd}
   \end{equation}
   sending a $\delta$-straight representative $\widetilde\wbf_y\coloneqq(\widetilde w_{\cl y_a})_a \in \prod_{a=0}^{[\kappa_y:\FF_q]-1}\widetilde W_{G_{\cl y_a}}$ to the \emph{unique} $\tau$-conjugacy class in $\widetilde G_y(F_{\cl y}) = G(\breve F\widehat\otimes_F F_y)$ containing the $\widetilde \Gcal_y(\Ocal_{\cl y})$-double coset corresponding to $\widetilde\wbf_y$ via \eqref{eq:Iwahori-Cartan-tilde}. (See \cite[\S3]{He:GeomADL}; especially, Proposition~3.6 and Theorem~3.7.) Furthermore, for any $[\widetilde b_y]\in B(\widetilde G_y,\widetilde\lambda_y)$, there exists a \emph{$\delta$-straight} $\widetilde\lambda_y$-admissible element 
   \begin{equation}
       \widetilde\wbf_y\in \prod_{a=0}^{[\kappa_y:\FF_q]-1}\Adm_{\Gcal,\cl y_a}(\lambda_{\cl y_a})
   \end{equation}
    which maps to $[\widetilde b_y]$ via \eqref{eq:He-map}; \emph{cf.} \cite[Proposition~4.1]{He:KottwitzRapoportConj}. Note that if $I_{\cl y_a}=\emptyset$ then $\lambda_{\cl y_a}= \underline 0$, so we have $\widetilde w_{\cl y_a}=1$. Thus, choosing a $\delta$-straight representative $ \widetilde\wbf_y$ of $[\widetilde b_y]$ for each $y\in Y$ as above, we obtain a tuple $\widetilde\wbf$ that defines $\Sht^{\leqslant\blambda_J,\widetilde\wbf}_{\Gcal,(I),\cl\ybf}$, and the proof of Proposition~\ref{prop:Adm} also shows that $\Sht^{\leqslant\blambda_J,\widetilde\wbf}_{\Gcal,(I),\cl\ybf}$ is contained in $\Sht^{\leqslant\blambda, \bbf_Y}_{\Gcal,(I),\cl\ybf}$. Lastly, the second claim also follows since a length-$0$ element is always $\delta$-straight, and it maps to a basic element in $B(\widetilde G_y)$ via \eqref{eq:He-map}. (Note that the Shapiro map \eqref{eq:Shapiro} preserves the unique basic elements on both sides.)
\end{proof}

\begin{defn}\label{def:basic-str}
    If we choose $\bbf_Y\coloneqq ([b_y])_{y\in Y}$ so that each $[b_y]$ is the basic element in $B(G_y,\lambda_y)$, then we write $\Sht^{\leqslant\blambda,\bsc}_{\Gcal, I_\bullet, \cl\ybf}$ in place of $\Sht^{\leqslant\blambda,\bbf_Y}_{\Gcal, I_\bullet, \cl\ybf}$. Note that $\Sht^{\leqslant\blambda,\bsc}_{\Gcal, I_\bullet, \cl\ybf}$ is a \emph{closed} substack of $\Sht^{\leqslant\blambda}_{\Gcal, I_\bullet, \cl\ybf}$, by the same proof as \cite[Theorem~3.6(iii)]{RapoportRichartz:Gisoc}.
\end{defn}

When $\Gcal = \Bcal^\times$ where $\Bcal$ is a hereditary $\Ocal_X$-order of a central simple algebra $B$ over $F$, we can show the non-emptiness of Kottwitz--Rapoport and Newton strata. Let us first show that the basic Newton strata are non-empty.

\begin{prop}\label{prop:non-empty-basic-stratum}
    If $\sum_{i\in I}\deg(\lambda_i) = 0$ then $\Sht^{\leqslant\blambda,\bsc}_{\Bcal^\times, I_\bullet, \cl\ybf}$ is non-empty for any $\cl\ybf$. %In particular, the converse of Lemma~\ref{lem:non-empty-Sht} holds.
\end{prop}

The proof follows the same line of thought as the ``incomplete idea of proof'' for the \emph{weak} basic non-emptiness of Kottwitz--Rapoport strata in \cite[\S5.2]{Breutmann:Thesis}. Furthermore, for $\Gcal = \Bcal^\times$ one can ``complete'' the argument in \emph{loc.~cit.} -- especially, \emph{(Step~2)} -- by explicit construction.

\begin{proof}
    By Theorem~\ref{th:change-parahorics}, we may assume that $\Bcal$ is a \emph{maximal} order. Without loss of generality, we may assume that $I = \bigsqcup I_{\cl y}$ (so $J=\emptyset$) and $I_\bullet = (I)$ using Corollary~\ref{cor:joining-modifications}.

    By weak approximation and Krasner's lemma, there is a finite separable extension $E/F$ of degree $d$ which is unramified and totally inert at each $y\in \Ram(B)\cup Y$. Then there is an $F$-algebra embedding $E\hookrightarrow B$; indeed, by choosing an element $\gamma_0^\ast\in E^\times$ so that $E = F(\gamma_0^\ast)$, we can find an element $\gamma_0\in B^\times$ whose characteristic polynomial equals the minimal polynomial of $\gamma_0^\ast$ (\emph{cf.} \cite[Lemma~8.14]{HamacherKim:Igusa}), and the $F$-subalgebra $F[\gamma_0]\subset B$ is isomorphic to $E$. 

    Let $\pi_E\colon \widetilde X_E\to X$ denote the finite cover of smooth curves corresponding to $E/F$. Then for some finite subset $Z\subset X$, the intersection $E_{x'}\cap\Bcal_x$ is the valuation ring $\Ocal_{E,x'}$ if $\pi_E(x')\notin Z$. We claim that if $x'|x$ for $x\in Z$, then there is a maximal $\Ocal_x$-order $\Bcal'_x$ of $B_x$ that contains the valuation ring $\Ocal_{E,x'}$. This is clear if $x\notin \Ram(B)$, so we may assume that $x\in \Ram(B)$, and thus $E_{x'}/F_x$ is an unramified extension of degree $d$ by assumption. Writing $B_x\cong M_e(\Delta_x)$ for some central division algebra $\Delta_x$ over $F_x$ of index $d/e$, we may assume that the unique maximal order $\Ocal_{\Delta_x}$ contains the valuation ring of a subfield $E'\subset E_{x'}$ of degree $d/e$ over $F_x$, by modifying the embedding $E_{x'}\hookrightarrow B_x$ if necessary. Then since we have
    \[\Ocal_{E, x'} \subset M_{e}(\Ocal_{E'}) \subset M_e(\Ocal_{\Delta_x})\subset M_e(\Delta_x)\cong B_x,\]
    we may take $\Bcal'_x\subset B_x$ to be the maximal $\Ocal_x$-order corresponding to $M_e(\Ocal_{\Delta_x})$.

    Recall that $\Bcal'_x$ is a $B_x^\times$-conjugate of $\Bcal_x$, so by the Beauville--Laszlo descent lemma \cite{Beauville-Laszlo:DescentLemma} we obtain a maximal $\Ocal_X$-order $\Bcal'$ of $B$ by glueing $\Bcal|_{X\setminus Z}$ and $\{\Bcal'_x\}_{x\in Z}$. Then $\Bcal'$ clearly contains $\pi_{E,\ast}\Ocal_{\widetilde X_E}$, and the Beauville--Laszlo descent lemma induces an isomorphism $\Bun_{\Bcal^\times}\cong \Bun_{\Bcal'^\times}$. Thus, we may rename $\Bcal'$ as $\Bcal$, and assume that $\pi_{E,\ast}\Ocal_{\widetilde X_E}$ is contained in $\Bcal$ without loss of generality. 

    Set $\Tcal\coloneqq \Res_{\widetilde X_E/X}\GG_m$, which is a smooth affine group scheme over $X$ with connected fibres.  Moreover, the inclusion $\pi_{E,\ast}\Ocal_{\widetilde X_E}\to \Bcal$ induces a group homomorphism $\iota\colon\Tcal\to\Bcal^\times$, 
%    \[
%       \iota\colon \begin{tikzcd}[column sep = small]
%            \Tcal \arrow[r] & \Bcal^\times
%        \end{tikzcd}.
%    \]
   and thus we get a morphism
    \begin{equation}
        \iota_\ast\colon \begin{tikzcd}
         \Bun_\Tcal \arrow[r]& \Bun_{\Bcal^\times}
        \end{tikzcd}.
    \end{equation}
    In terms of vector bundles, $\Bun_{\Res_{\widetilde X_E/X}\GG_m}(S)$ classifies rank-$d$ vector bundles $\Fcal$ on $X\times S$ that are rank-$1$ locally free $(\pi_{E,\ast}\Ocal_{\widetilde X_E}\boxtimes \Ocal_S)$-modules, and the above morphism sends $\Fcal$ to $\Fcal\otimes_{\pi_{E,\ast}\Ocal_{\widetilde X_E}}\Bcal$.

    Note that the generic fibre of $\Tcal$ is a torus $T\coloneqq\Res_{E/F}\GG_m$, and $\iota$ maps $T$ to a maximal torus of $B^\times$. So given $\blambda' = (\lambda_i')_{i\in I}\in X_\ast(T)^I$ we get a (possibly empty) moduli stack $\Sht_{\Tcal,I_\bullet}^{\blambda'}$ of $\Tcal$-shtukas bounded by $\blambda'$, and $\iota_\ast$ induces a morphism
    \begin{equation}\label{eq:special-points}
        \iota_\ast\colon \begin{tikzcd}
            \Sht_{\Tcal,I_\bullet}^{\blambda'} \arrow[r]& \Sht_{\Bcal^\times,I_\bullet}^{\leqslant\blambda'_{\dom}}
        \end{tikzcd},
    \end{equation}
    where $\blambda'_{\dom} = (\lambda'_{i,\dom})_{i\in I}$ is the collection of the dominant representatives of the Weyl group orbit of $\lambda'_i$'s. (Although $\Tcal$ may not be parahoric at places where $E/F$ ramifies, the construction of moduli stacks in \cite{ArastehRad-Hartl:LocGlShtuka, ArastehRad-Hartl:Uniformizing, Bieker:IntModels} is general enough to handle such a group.)

    Recall that we chose $E/F$ to be unramified and \emph{totally inert} at each place $y\in Y$, so $\iota$ sends $T_y$ to an \emph{elliptic} maximal torus of $B_y^\times$ for each $y\in Y$. Therefore, the image of $B(T_y)\to B(B_y^\times)$ is contained in the set of basic elements by \cite[Proposition~5.3]{Kottwitz:Gisoc1}, and $\iota_\ast$ induces 
    \begin{equation}
        \begin{tikzcd}
            \Sht_{\Tcal,I_\bullet,\cl\ybf}^{\blambda'} \arrow[r, "\iota_\ast"]& \Sht_{\Bcal^\times,I_\bullet,\cl\ybf}^{\leqslant\blambda'_{\dom},\bsc}
        \end{tikzcd}.
    \end{equation}
    Hence, to prove the proposition it suffices to find $\blambda' = (\lambda'_i)_{i\in I}\in X_\ast(T)^I$ such that $\lambda'_{i,\dom}\leqslant\lambda_i$ for each $i$ and $\Sht_{\Tcal,I_\bullet,\cl\ybf}^{\blambda'}$ is \emph{non-empty}.
         
    We first claim that the fibre $\Sht_{\Tcal,I_\bullet,\cl\ybf}^{\blambda'}$ is non-empty if $\sum_{i\in I}\lambda'_i = \underline{0}$. This claim is clear if $\cl y_i = \cl y$ for all $i\in I$ and $I_\bullet = (I)$; in fact, 
    \[
    \big ((\cl y)^I, \Fcal_0 = \Fcal_1 = (\pi_{E,\ast}\Ocal_{\widetilde X_E})\otimes \cl {\mathbb F}_q, \varphi_1 = \id, \alpha=\id)
    \]
    defines an $\cl\FF_q$-point of $\Sht_{\Tcal,I_\bullet,(\cl y)^I}^{\blambda'}$. In general, if $\Ram(E)\subset X$ denotes the ramification locus for $E/F$, then $\Gr_{\Tcal,I_\bullet}\to X^I$ is \emph{ind-finite} over $(X\setminus \Ram(E))^I$ (\emph{cf.} \cite[Lemma~2.20]{Richarz:AffGr}). Since $\Sht_{\Tcal,I_\bullet}^{\blambda'}\to \widetilde{X}_{E}^I$ is a countable union of finite \'etale maps (\emph{cf.} \cite[Lemma~5.13]{YunZhang:ShtukaTaylorI}) and $\widetilde{X}_{E}\to X$ is finite flat, the local model theorem \cite[Proposition~3.4.3]{Bieker:IntModels} shows that $\Sht_{\Tcal,I_\bullet}^{\blambda'}\to X^I$ is a countable union of finite flat maps over $(X\setminus \Ram(E))^I$. Hence, each fibre of 
    \[
    \begin{tikzcd}
        \Sht_{\Tcal,I_\bullet}^{\blambda'} |_{(X\setminus\Ram(E))^I} \arrow[r]& (X\setminus\Ram(E))^I
    \end{tikzcd}
    \]
    is non-empty if and only if $\Sht_{\Tcal,I_\bullet}^{\blambda'} |_{(X\setminus\Ram(E))^I}$ is non-empty. Since we showed that $ \Sht_{\Tcal,I_\bullet,(\cl y)^I}^{\blambda'}\ne\emptyset$ if $\sum_{i\in I}\lambda'_i = \underline{0}$, the desired claim follows as $Y\subset X\setminus \Ram(E)$.
    
    It now remains to construct $\blambda'\in X_\ast(T)^I$ such that $\sum_{i\in I}\lambda'_i = \underline{0}$ and $\lambda'_{i,\dom}\leqslant \lambda_i$ for each $i\in I$. Let $\lambda_{i}^{\min}$ denote the unique central or minuscule element that is less than or equal to $\lambda_i$; in other words, if $\lambda_i = (\lambda_{i,1},\cdots,\lambda_{i,d})$ then $\lambda_i^{\min} = (\lambda^{\min}_{i,1},\cdots,\lambda^{\min}_{i,d})$ satisfies $\lambda^{\min}_{i,1}-\lambda^{\min}_{i,d}\leqslant 1$ and $\sum_{j=1}^d\lambda^{\min}_{i,j} = \sum_{j=1}^d\lambda_{i,j}$. Set $\blambda^{\min}\coloneqq (\lambda_i^{\min})_{i\in I}$. 
    Then, since replacing $\blambda$ with $\blambda^{\min}$ does not modify the unique basic element in $B(B^\times_y,\lambda_y)$, $\Sht^{\leqslant\blambda,\bsc}_{\Bcal^\times, I_\bullet, \cl\ybf}$ contains $\Sht^{\leqslant\blambda^{\min},\bsc}_{\Bcal^\times, I_\bullet, \cl\ybf}$. So we may assume $\blambda = \blambda^{\min}$.

    For each $i\in I$, let $e_i$ denote the number of $j$ such that $\lambda_{i,j} = \lambda_{i,1} = \lambda_{i,d}+1$. Then  $d$ divides $\sum_{i\in I}e_i$
    since we have 
    \[0 =\sum_{i\in I}\deg(\lambda_i) =  \sum_{i\in I}\sum_{j=1}^d \lambda_{i,j}  = \sum_{i\in I}e_i + d\sum_{i\in I} \lambda_{i,d}.\] 
    It remains to show that for each $i\in I$ there is $\lambda'_i\in X_\ast(T)$ in the Weyl group orbit of $\lambda_i$ such that $\sum_{i\in I}\lambda'_i = \underline 0$. Equivalently, for each
    \begin{equation}\label{eq:cowt-combinatorics-delta}
        \delta_i\coloneqq \lambda_i - \lambda_{i,d}\cdot(1,\cdots,1) = (\underbrace{1,\cdots,1}_{e_i\text{ times}},0,\cdots,0)
    \end{equation}
    with $0\leqslant e_i<d$ such that $\sum_{i\in I}e_i = d\cdot s$ for some non-negative integer $s$, we claim 
    \begin{equation}\label{eq:cowt-combinatorics-claim}
        \forall i\in I,\ \exists\ \epsilon_i\in X_\ast(T) \text{ such that }\epsilon_{i,\dom} = \delta_i \ \&\ \sum_{i\in I}\epsilon_i = s\cdot(1,\cdots,1).
    \end{equation}
    
   We reduce Claim \eqref{eq:cowt-combinatorics-claim} for $(\delta_i)_{i \in I}$ to that for another sequence $(\delta'_i)_{i \in I'}$ with $|I'| < |I|$, following the process outlined below:
    \begin{description}
        \item[Step 0] Remove any $\delta_i$ such that $\delta_i = \underline{0}$, because in this case $\epsilon_i = \underline{0}$ is the only possible choice.
        \item[Step 1] Let $\delta_i'$ be the dominant representative of the Weyl group orbit of $ (1,\cdots,1)-\delta_i$. Since Claim \eqref{eq:cowt-combinatorics-claim} for $(\delta_i)_{i \in I}$ and $(\delta_i')_{i \in I}$ are equivalent by taking $\epsilon_i' = (1,\cdots, 1) - \epsilon_i$, we may assume $s \leq |I|/2$ by replacing $(\delta_i)_{i \in I}$ with $(\delta_i')_{i \in I}$ if necessary.
        \item[Step 2] If $e_{i_1} + e_{i_2}<d$  for some distinct $i_1,i_2\in I$, then we define $I' \coloneqq I\setminus\{i_2\}$, and $(\delta'_i)_{i\in I'}$ by 
        $\delta'_i = \delta_i$ for $i\ne i_1,i_2$ and 
        \[
        \delta'_{i_1}\coloneqq  (\underbrace{1,\cdots,1}_{e_{i_1}\text{ times}},\underbrace{1,\cdots,1}_{e_{i_2}\text{ times}},0,\cdots,0).
        \]
        Since Claim \eqref{eq:cowt-combinatorics-claim} for $(\delta'_i)_{i\in I\setminus\{i_2\}}$  implies Claim \eqref{eq:cowt-combinatorics-claim} for $(\delta_i)_{i\in I}$, we replace $(\delta_i)_{i\in I}$ with $(\delta'_i)_{i\in I\setminus\{i_2\}}$.
         \item[Step 3] If $e_{i_1}+e_{i_2}= d$ for some distinct $i_1,i_2\in I$, then Claim \eqref{eq:cowt-combinatorics-claim} for $(\delta_i)_{i\in I\setminus\{i_1,i_2\}}$ implies Claim \eqref{eq:cowt-combinatorics-claim} for $(\delta_i)_{i\in I}$, since we may set $\epsilon_{i_1}=\delta_{i_1}$ and $\epsilon_{i_2} = (1,\cdots,1)-\epsilon_{i_1}$. Thus, we remove $\delta_{i_1}$ and $\delta_{i_2}$.
        \item[Step 4] Repeat \textbf{Steps 2} and \textbf{3}, and terminate the process if no non-zero $\delta_i$ remains. Otherwise, we may assume $e_i+e_{i'}>d$ for all distinct $i,i'\in I$, which clearly implies $s>|I|/2$. In this case, return to \textbf{Step 1}.
    \end{description}
    This process terminates because in steps 2 and 3 the cardinality $|I|$ becomes smaller, and thus the proof of Claim \eqref{eq:cowt-combinatorics-claim} is reduced to the trivial case where $\delta_i=\underline0$ for all $i$. 
    
    Returning to the original setting where $(\delta_i)_{i\in I}$ are given as in \eqref{eq:cowt-combinatorics-delta}, we set 
    \[\lambda_i'\coloneqq \lambda_{i,d}\cdot(1,\cdots,1) + \epsilon_i,\]
    where $(\epsilon_i)_{i\in I}$ are given by Claim~\ref{eq:cowt-combinatorics-claim}. Clearly, $\lambda_i'$ is in the Weyl group orbit of $\lambda_i$ for each $i\in I$, and we have $\sum_{i\in I}\lambda_i' = \underline 0$ as desired.
\end{proof}

\begin{cor}\label{cor:nonempty-KR}
    Suppose that $\sum_{i\in I}\deg(\lambda_i) = 0$. Choose a finite non-empty subset $Y\subset X$ such that $\Bcal_y$ is an Iwahori order for any $y\in Y$, and we choose a partition $\bigsqcup_{y\in Y}\bigsqcup_{\cl y|y}I_{\cl y} \subset I$, which defines $\cl\ybf$. 
    
    Then $\Sht^{\leqslant\blambda_J,\widetilde\wbf}_{\Bcal^\times,(I),\cl\ybf}$ is non-empty for any $\widetilde\wbf\in\prod_{\cl y, \text{ s.t. }I_{\cl y}\ne\emptyset}\Adm_{\Gcal,\cl y}(\lambda_{\cl y})$ where $\lambda_{\cl y} \coloneqq \sum_{i\in I_{\overline{y}}} \lambda_i$.
\end{cor}

\begin{proof}
    Without loss of generality, we may assume $J=\emptyset$. Suppose for a moment that the \textit{basic} Kottwitz--Rapoport stratum $\Sht^{\widetilde\wbf_{\bsc}}_{\Bcal^\times,(I),\cl\ybf}$ is non-empty, and choose a point $v \in |\Sht^{\widetilde\wbf_{\bsc}}_{\Bcal^\times,(I),\cl\ybf}|$. By the local model theorem \cite[Proposition~3.4.3]{Bieker:IntModels}, there exists an \'etale neighbourhood $(U,u)$ of $v$ together with an \'etale morphism
    \begin{equation}\label{eq:loc-mod-nbhd}
      U\to \prod_{\cl y \text{ s.t. } I_{\cl y}\ne\emptyset} (\Gr^{\leqslant \blambda_{\cl y}}_{\Bcal^{\times},(I_{\cl y})})_{\Delta(\cl y)}
    \end{equation}
    which maps $u$ to the unique point in the $0$-dimensional cell $\prod_{\cl y} \Scal^{\widetilde w_{\cl y,\bsc}}$. The (set theoretic) image of $U$ under \eqref{eq:loc-mod-nbhd} is open by \'etaleness, so it intersects with $\prod_{\cl y} \Scal^{\widetilde w_{\cl y}}$ for any $\widetilde w_{\cl y}\in \Adm_{\Gcal,\cl y}(\lambda_{\cl y})$, since the closure of $\prod_{\cl y}\Scal^{\widetilde w_{\cl y}}$ contains $\prod_{\cl y} \Scal^{\widetilde w_{\cl y,\bsc}}$. This shows the non-emptiness of $\Sht^{\widetilde\wbf}_{\Bcal^\times,(I),\cl\ybf}$ for any $\widetilde\wbf$ as in the statement.
    
    Now, we prove that $\Sht^{\widetilde\wbf_{\bsc}}_{\Bcal^\times,(I),\cl\ybf}$ is non-empty following \cite[\S5.2, Axiom~5, (Step~3)]{Breutmann:Thesis}. Choose $(\Pucal,\alpha)\in \Sht^{\leqslant\blambda,\bsc}_{\Bcal^\times,(I),\cl\ybf}(\cl\FF_q)$, which is possible by Proposition~\ref{prop:non-empty-basic-stratum}. We want to show that $(\Pucal,\alpha)$ can be ``modified'' at each $y\in Y$ (via the Beauville--Laszlo descent) to lie in $\Sht^{\widetilde\wbf_{\bsc}}_{\Bcal^\times,(I),\cl\ybf}$. 
    Let us use the notation from the proof of Proposition~\ref{prop:Adm}. Let $\widetilde b_y\in G(\breve F\widehat\otimes_F F_y)$ be the element associated to the localisation of $(\Pucal,\alpha)$ at $y$, and let $\widetilde\wbf_{y,\bsc}$ be the unique length-$0$ $\widetilde\lambda_y$-admissible element in the Iwahori--Weyl group for $G(\breve F\widehat\otimes_F F_y)$. %Then, for the non-emptiness of $\Sht^{\widetilde\wbf_{\bsc}}_{\Bcal^\times,(I),\cl\ybf}$ is non-empty, it suffices to show that the affine Deligne--Lusztig variety $X_{\widetilde\wbf_{y,\bsc}}(\widetilde b_y)$ is non-empty, where $\widetilde\wbf_{y,\bsc}$ is the unique length-$0$ $\widetilde\lambda_y$-admissible element in the Iwahori--Weyl group for $G(\breve F\widehat\otimes_F F_y)$.
    Since length-$0$ elements are $\delta$-straight in the sense of \cite[\S2.4]{He:GeomADL}, the non-emptiness of the affine Deligne--Lusztig variety $X_{\widetilde\wbf_{y,\bsc}}(\widetilde b_y)$ follows from \cite[Theorem~4.5]{He:GeomADL}. We may therefore choose $g_y \in X_{\widetilde\wbf_{y,\bsc}}(\widetilde b_y)$ for each $y\in Y$. Modifying $(\Pucal,\alpha)$ along these $g_y$ via the Beauville--Laszlo descent yields the desired element in $\Sht^{\widetilde\wbf_{\bsc}}_{\Bcal^\times,(I),\cl\ybf}(\cl\FF_q)$.
\end{proof}

Now, we can show the non-emptiness of Newton strata.

\begin{cor}\label{cor:nonempty-Newton}
    Fix $Y$ and $(I_{\cl y})$ as before, and suppose that we have $\sum_{i\in I}\deg(\lambda_i) = 0$, and $\bbf_Y=\big(N([\widetilde b_y])\big)_{y\in Y}$ for some $[\widetilde b_y]\in B(\widetilde G_y,\widetilde \lambda_y)$, using the notation in the proof of Proposition~\ref{prop:Adm}. Then the Newton stratum $\Sht^{\leqslant\blambda,\bbf_Y}_{\Bcal^\times,I_\bullet,\cl\ybf}$ is non-empty.
\end{cor}
\begin{proof}
    We may assume that $J=\emptyset$ without loss of generality.
    Since the Newton stratification is invariant under ``change of parahorics'' maps and ``joining modification'' maps, which are surjective by Theorem~\ref{th:change-parahorics} and Corollary~\ref{cor:joining-modifications}, we may also assume that $\Bcal_y$ is an Iwahori order for any $y\in Y$ and $I_\bullet = (I)$. Now the corollary follows from Corollaries~\ref{cor:nonempty-KR} and \ref{cor:straight-KR}.
\end{proof}
\begin{rmk}
    If $\Bcal$ is a hereditary $\Ocal_X$-order that is Iwahori at each $y\in Y$ and $I_\bullet = (I)$, then it is conjectured that the basic Kottwitz--Rapoport stratum $\Sht^{\widetilde\wbf_{\bsc}}_{\Bcal^\times,(I),\cl\ybf}$ intersects with each connected component of $\Sht^{\leqslant\blambda}_{\Bcal^\times,(I),\cl\ybf}$. See the strong version of Axiom~5 in \cite[\S5.1]{Breutmann:Thesis}, which is inspired by \cite[\S3, Axiom~5]{HeRapoport:EKOR}. For this, we may need to explicitly describe the connected components of $\Sht^{\leqslant\blambda}_{\Bcal^\times,I_\bullet}$, and we are not aware if such a result is known for arbitrary choice of $\blambda$.
\end{rmk}

\section{Review of $(D,\varphi)$-spaces and Dieudonn\'e $D_x$-modules}\label{sec-isosht}
We fix a central division algebra $D$ over $F$ with dimension $d^2$ for the rest of the paper. In this section, we review basic properties of \emph{($D,\varphi)$-spaces} (respectively, \emph{Dieudonn\'e $D_x$-modules} for any $x\in |X|$), which are isoshtukas with right $D$-action (respectively,  local isoshtukas with right $D_x$-action) over $\cl\FF_q$, which can be found in \cite{LaumonRapoportStuhler}. Then we prove the key numerical inequality (Lemma~\ref{lemma-Lau-bound-for-degree-of-D-varphi-space-outside-split-places}), which plays an important role in the proof of the main result in \S\ref{sec-valcrit}.

As explained in \S\ref{ssec-CSA} and Remark~\ref{rmk-shtuka-via-module}, we will freely switch between $\Dcal^\times$-bundles over $S$ and rank-$1$ locally free $\Dcal_{X\times S}$-modules in defining $\Dcal^\times$-shtukas. However, we prefer the module perspective, as we should work with $\Dcal_{X\times S}$-submodules and $\Dcal$-stable subshtukas (Definition~\ref{def:D-subshtukas}).
%From the latter viewpoint, a $\Dcal^\times$-shtuka $\Eucal$ with legs $(x_1,\cdots,x_r)\in X^I(S)$ can be written as
%\begin{equation}
%\xymatrix@1{
%\Ecal_0\ar@{..>}[r]^-{\varphi_1} & 
%\Ecal_1 \ar@{..>}[r]^-{\varphi_2} &\ \cdots\ \ar@{..>}[r]^-{\varphi_r} & 
%\Ecal_r & 
%\ltau\Ecal_0 \ar[l]^-{\cong}_-{\alpha} }
%\end{equation}
%where $\Ecal_j$ is a rank-$1$ locally free $\Dcal_{X\times S}$-modules and $\varphi_j\colon \Ecal_{j-1}|_{(X\times S)\smallsetminus\Gamma_{x_j}}\riso \Ecal_j|_{(X\times S)\smallsetminus\Gamma_{x_j}}$ for any $j=1,\cdots,r$. Though we will have to switch to the $\Dcal^\times$-torsor viewpoint to define the boundedness by $\blambda$ for $\Eucal$, we will prefer to work with the module viewpoint of $\Dcal^\times$-shtukas for what follows.

\subsection{Isogeny classes and localisations for $\Gcal$-shtukas}\label{ssec-isoshtukas}
We start with the general setup. For a connected reductive group $G$ over $F$, a \emph{$G$-isoshtuka} over an $\FF_q$-scheme $S$ is defined to be a pair
$(P,\varphi)$ of a $G$-torsor $P$ over $\Spec F\times S$ and an isomorphism $\varphi\colon \ltau P \to P$ of $G$-torsors, where $\ltau(-)$ is the pullback by $\tau = \id_F\times\Frob_{q}$; see \cite[p1]{HamacherKim:Gisoc} when $S=\Spec\cl\FF_q$. Now, to a $\Gcal$-shtuka $(\Pucal,\alpha)\in\Sht_{\Gcal,I_\bullet}(S)$ whose legs are disjoint from the generic point $\eta$ (with $\Gcal$ as in \S\ref{ssec:groups} and $\Pucal$ as in Definition~\ref{def:Hecke}), we associate a $G$-isoshtuka $(P,\varphi)$ over $S$ by setting
\begin{equation}\label{eq:varphi}
     P \coloneqq \Pcal_0|_{\eta\times S}\quad\text{and}\quad
     \varphi\coloneqq (\varphi_1|_{\eta\times S})^{-1}\circ\cdots\circ (\varphi_r|_{\eta\times S})^{-1} \circ (\alpha|_{\eta\times S}) ,
\end{equation}
where $\eta=\Spec F$ is the generic point of $X$. One can view the associated $G$-isoshtuka $(P,\varphi)$ as the ``(global) isogeny class'' of the $\Gcal$-shtuka $(\Pucal,\alpha)$. Clearly, the construction $(P,\varphi)$ depends only on 
\begin{equation}
    ((x_i)_{i\in I}, (\Pcal_0,\Pcal_r), \varphi_r\circ\cdots\circ\varphi_1,\alpha)\in \Sht_{\Gcal,(I)}(S);
\end{equation}
i.e., the image of $(\Pucal,\alpha)$ by the ``joining modification'' map in Corollary~\ref{cor:joining-modifications}.

Now, let $S$ be an $\cl\FF_q$-scheme, and fix $Y$ and $\cl\ybf$ as in \eqref{eq:ybf}. Then to a $\Gcal$-shtuka $(\Pucal,\alpha)\in\Sht_{\Gcal,(I),\cl\ybf}(S)$, defined below \eqref{eq:Gr-ybf}, we associate a $G_y$-isoshtuka $(\Lcal_y\Pcal_0,\varphi_y)$ on $S$ by pulling back the universal $G_y$-isoshtuka \eqref{eq:univ-loc-isosht}. Then clearly, $(\Lcal_y\Pcal_0,\varphi_y)$ only depends on the $G$-isoshtuka $(P,\varphi)$ associated to $(\Pucal,\alpha)$.

Let us now consider $\Gcal = \Dcal^\times$ for a hereditary $\Ocal_X$-order of $D$, where $D$ is a central division algebra over $F$ of dimension $d^2$. Set $\breve D\coloneqq D\otimes_F \breve F$. Then using \eqref{eq:torsor-module}, giving a $D^\times$-isoshtuka over $\cl{\FF}_q$ is equivalent to giving a rank-$1$ free right $\breve D$-module  $V$ together with a bijective $\tau$-linear and $D$-linear endomorphism $\varphi\colon V\to V$. Similarly, by setting $\breve D_x\coloneqq D\otimes_F \breve F_x$, giving a local $D_x^\times$-isoshtuka over $\cl{\FF}_q$ is equivalent to giving a rank-$1$ free right $\breve D_x$-module $V_x$ with a bijective $\tau_x$-linear and $D_x$-linear endomorphism $\varphi_x\colon V_x\to V_x$. These naturally motivate the definitions of $(D,\varphi)$-spaces (Definition~\ref{def:D-varphi-space}), Dieudonn\'e $D_x$
-modules (Definition~\ref{def:Dieudonne-D-x-module}) and localisations (Definition~\ref{def:localisation}).
%, so each $\Ecal_j$ is a rank-$1$ locally free $\Dcal_{\breve X}$-module. Let $V$ denote the generic fibre of $\Ecal_0$, which is a rank-$1$ free module over $\breve D\coloneqq D\otimes_F \breve F$. Furthermore, each $\varphi_j$ induces an isomorphism of the generic fibres, so we get a $\breve D$-isomorphism 
%\[
%    (\varphi_1)^{-1}_{\breve\eta}\circ\cdots\circ (\varphi_r)^{-1}_{\breve\eta} \circ \alpha_{\breve\eta} \colon \xymatrix@1{\ltau V \ar[r] & V}
%\]
%Let $\varphi\colon V\to V$ denote the $\tau$-linear endomorphism induced by this isomorphism, which is equivariant under the (right) action of $D$. This motivates the following definition.

\begin{defn}[{\emph{cf.} \cite[p~363]{Lau:Degeneration}}]\label{def:D-varphi-space}
    A \emph{$(D,\varphi)$-space} is a tuple $(V,\varphi,\iota)$ where 
\begin{enumerate}
  \item $V$ is a finite-dimensional $\breve F$-vector space,
  \item $\varphi : V\to V$ is a bijective $\tau$-linear homomorphism,
  \item $\iota: D^{\op}\to \End(V,\varphi)$ is a morphism of $F$-algebras (where elements of $\End(V,\varphi)$ are $\breve F$-linear maps $V\to V$ respecting $\varphi$).
\end{enumerate}
\end{defn}
%In other words, a $(D,\varphi)$-space is a finitely generated right $\breve D$-module $V$ equipped with a $\tau$-linear automorphism $\varphi$, where the $\breve D$-action on $V$ is induced by $\iota$. 
The construction outlined in \S\ref{ssec-isoshtukas} associates to a $\Dcal^\times$-shtuka over $\cl\FF_q$ a $(D,\varphi)$-space whose underlying $\breve D$-module is free of rank~$1$, though we also need to consider $(D,\varphi)$-spaces whose underlying $\breve D$-module is not free. (Note that $\breve D\cong M_d(\breve F)$.) For example, if $(\Fucal,\alpha)$ is a proper $\Dcal$-stable subshtuka of a $\Dcal^\times$-shtuka $(\Eucal,\alpha)$, then the generic fibre $W$ of $\Fcal_0$ is a $\varphi$-stable right $\breve D$-submodule so $(W,\varphi|_W)$ is a proper $(D,\varphi)$-subspace of $(V,\varphi)$.

When $D=F$,  Drinfeld \cite{Drinfeld:PeterssonsConjecture} showed that the category of $(D,\varphi)$-spaces turns out to be semisimple, and he classified simple $(F,\varphi)$-modules. Let us recall the classification result, following \cite[Theorem~(A.6)]{LaumonRapoportStuhler}. %The classification of simple objects uses the construction $(V,\varphi)\rightsquigarrow(L_{(V,\varphi)},\Pi_{(V,\varphi)})$, which we recall.
\begin{defn}[{\cite[p.~32]{Drinfeld:PeterssonsConjecture}}]
  \label{defn:Pi-V-phi-space}
  Let $(V,\varphi)$ be an $(F,\varphi)$-space. Then, for some integer $n\geq1$, there exists a $\varphi$-stable finite-dimensional $F\otimes \FF_{q^n}$-subspace $V_0$ of $V$ containing an $\breve F$-basis. Set $\varphi_0 \coloneqq\varphi|_{V_0}$, and observe that $\varphi_0^n$ is an $F\otimes \FF_{q^n}$-linear endomorphism of $V_0$ that commutes with $\varphi_0$. Let $\pi_0\in\End_{(F,\varphi)}(V,\varphi)$ be the endomorphism obtained by extending scalars for  $\varphi_0^n$, and define the following commutative $F$-subalgebra
  \begin{equation}
       L_{(V,\varphi)} \coloneqq \bigcap_{N=1}^\infty F[\pi_0^N] \subseteq \End_{(F,\varphi)}(V,\varphi).
  \end{equation}
  Since each term in the above intersection is finite-dimensional over $F$, there exists $N_0\gg1$ such that $L_{(V,\varphi)} = F[\pi_0^{N_0}]$. Furthermore, $L_{(V,\varphi)}$ is a finite \'etale $F$-algebra since $L_{(V,\varphi)} = F[\pi_0^{N_0p^i}]$ for any $i\geq0$.
  
%  There is an integer $N\geq1$ such that $L_{(V,\varphi)} = F\left[ ((\varphi^{\prime})^{n}\otimes_{\mathbb{F}_{q^n}}\id_{\cl{\mathbb{F}}_q})^N \right]$, since $\dim_F  F[ ((\varphi^{\prime})^{n}\otimes_{\mathbb{F}_{q^n}}\id_{\cl{\mathbb{F}}_q})^N ] < \infty$ for every $N\geq1$. %Then $((\varphi^{\prime})^n \otimes \id_{\cl{\mathbb{F}}_q})^N \in (L_{(V,\varphi)})^{\times}$.
Now, set
  \begin{equation}
      \Pi_{(V,\varphi)} \coloneqq \pi_0^{N_0}\otimes (1/nN_0) \in (L_{(V,\varphi)})^{\times}\otimes_{\ZZ}\QQ.
  \end{equation}
  Conceptually, $\Pi_{(V,\varphi)}$ can be thought of as $(\pi_0^{N_0})^{1/nN_0} = \pi_0^{1/n}$. Clearly, both $L_{(V,\varphi)}$ and $\Pi_{(V,\varphi)}$ are independent of the choices of $N_0$, $n$ and $V_0$.
\end{defn}

For any finite extension $L/F$ and a place $y$ of $L$, we define a $\QQ$-linear function
\begin{equation}\label{eq:deg-y}
    \deg_y\colon L^\times\otimes_\ZZ\QQ \to \QQ
\end{equation}
by $\QQ$-linearly extending $[\kappa_y:\FF_q]\cdot v_y$, where $v_y$ is the normalised valuation sending a uniformiser at $y$ to $1$.
Let us now recall the description of simple $(F,\varphi)$-spaces.
\begin{thm}[Drinfeld; {\emph{cf.} \cite[Theorem~(A.6)]{LaumonRapoportStuhler}}]\label{th:classif-simple-isosht}
    For each $(L,\Pi)$ where $L$ is a finite separable field extension of $F$ and $\Pi\in L^\times\otimes_\ZZ\QQ$ such that $\Pi \not \in (L')^{\times}\otimes_\ZZ\QQ$ for every proper $F$-subalgebra $L'$ of $L$, there exists a unique simple $(F,\varphi)$-space $(V,\varphi)$ up to isomorphism such that $(L_{(V,\varphi)},\Pi_{(V,\varphi)}) \cong (L,\Pi)$. 
    
%    For a place $y$ of $L$, let $\deg_y\colon L^\times\to\ZZ$ denote the valuation at $y$ normalised so that a uniformiser maps to $[\kappa_y:\FF_q]$, and we extend it to a $\QQ$-linear map $\deg_y\colon L^\times\otimes_\ZZ\QQ\to\QQ$.
%    Then, for a simple $(D,\varphi)$-module $(V,\varphi)$ corresponding to $(L,\Pi)$, we have
    Furthermore, for such $(V,\varphi)$ we have
    \[
          \dim_{\breve F}(V) = [L:F]\cdot d(\Pi)
    \]
    where $d(\Pi)$ is the lowest common denominator of $\{\deg_y(\Pi)\}_y$ with $y$ running through all places of $L$. Lastly, $\End_{(F,\varphi)}(V,\varphi)$ is a central division algebra over $L$ with index $d(\Pi)$, and its local invariant at a place $y$ of $L$ is $-\deg_y(\Pi)\bmod\ZZ$.
\end{thm}

%Drinfel$'$d  defined and studied $(D,\varphi)$-spaces (of any rank) when $D=F$, called \emph{$\varphi$-spaces}. He showed that the category of $\varphi$-spaces is semisimple, and classified simple objects; \emph{cf.} \cite[Appendix~A]{LaumonRapoportStuhler}. 
%\wk{Need to recall the classification of simple $\varphi$-modules to introduce the notation $\Pi_{(V,\varphi)}$.} \yg{I put it in Definition \ref{defn:Pi-V-phi-space} but I think it should be re-placed.}
%From this one can deduce that the category of $(D,\varphi)$-spaces is also \emph{semisimple}, and classify simple objects as follows:

The semisimplicity of $(F,\varphi)$-spaces immediately implies the semisimplicity of $(D,\varphi)$-spaces for any central division algebra $D$ over $F$. Furthermore, one can obtain the following result on simple $(D,\varphi)$-spaces.
\begin{cor}[{\emph{Cf.} \cite[Corollary~6.2.3]{Lau:Thesis}}]
  \label{cor:properties-D-varphi-space}
  Let $(V,\varphi,\iota)$ be a simple $(D,\varphi)$-space. Then $L_{(V,\varphi)}$ is a field, and as an $(F,\varphi)$-space  $(V,\varphi)$ is a direct sum of copies of the simple $(F,\varphi)$-space $(W,\varphi_W)$ corresponding to $(L_{(V,\varphi)},\Pi_{(V,\varphi)})$ via Theorem~\ref{th:classif-simple-isosht}. 
  
  Furthermore, the following properties hold.
  \begin{enumerate}
    % \item the $F$-pair $(L_{(V,\varphi,\iota)},\Pi_{(V,\varphi,\iota)})$ (cf. \ref{construction-D-varphi-space-to-F-pair}) is indecomposable, i.e., $L_{(V,\varphi,\iota)}$ is a field,
    \item The endomorphism ring $\Delta = \End_{(D,\varphi)} ( V,\varphi,\iota)$ is a central division algebra over $L_{(V,\varphi)}$, and for each place $y$ of $L_{(V,\varphi)}$ we have
    \[
    \inv_y(\Delta) =  \inv_y (D\otimes_F L_{(V,\varphi)}) -\deg_y(\Pi_{(V,\varphi)}) \in \QQ/\ZZ
    \]
    \item The multiplicity of $(W,\varphi_W)$ in $(V,\varphi)$ is $d\cdot d(\Delta) / d(\Pi_{(V,\varphi)})$, where $d(\Delta)$ is the index of $\Delta$ over $L_{(V,\varphi)}$. In particular, we have 
    \[
    \dim_{\breve F} V = d\cdot [L_{(V,\varphi)} : F] \cdot d(\Delta).
    \]
  \end{enumerate}
\end{cor}

Let us now move on to the local analogue of $(D,\varphi)$-spaces.
\begin{defn}[{\emph{Cf.} \cite[p~363]{Lau:Degeneration}}]\label{def:Dieudonne-D-x-module}
    Let $x$ be a closed point of $X$, and set $D_x\coloneqq D\otimes_F F_x$.
    
     A \emph{Dieudonn\'e $D_x$-module}  is a tuple $(V,\varphi,\iota)$ where 
\begin{enumerate}
  \item $V$ is a finite-dimensional $\breve F_x$-vector space,
  \item $\varphi:V\to V$ is a bijective $\tau^{\deg(x)}$-linear homomorphism,
  \item $\iota: (D_x)^{\op}\to \End(V,\varphi)$ is a morphism of $F_x$-algebras (where elements of $\End(V,\varphi)$ are $\breve F_x$-linear maps respecting $\varphi$). 
\end{enumerate}
\end{defn}
%Let $\breve D_x\coloneqq D_x\otimes_{F_x}\breve F_x$, and we abusively let $\tau_x\colon\breve D_x\to\breve D_x$ denote the map $\id_{D_x}\otimes\tau_x$. Then, a Dieudonn\'e $D_x$-module can be understood as a finitely generated (right) $\breve D_x$-module $V$ equipped with a $\tau_x$-linear and $D_x$-linear automorphism $\varphi$, where the $\breve D_x$-action on $V$ is induced by $\iota$.

To a $(D,\varphi)$-space, we can associate a Dieudonn\'e $D_x$-module as follows. 
\begin{defn}\label{def:localisation}
Given a closed point $x$ of $X$, choose a geometric point $\bar x\in X(\cl\FF_q)$ over $x$, which gives an $F$-algebra embedding $\breve F\hookrightarrow \breve F_x$. The \emph{localisation} of a $(D,\varphi)$-space $(V,\varphi,\iota)$ (with respect to $\bar x$) is a Dieudonn\'e $D_x$-module $(V_x,\varphi_x,\iota_x)$ constructed as follows:
\begin{enumerate}
    \item $V_x\coloneqq V\otimes_{\breve F} \breve F_x$
    \item $\varphi_x$ is the $\tau_x$-linear extension of $\varphi^{\deg(x)}:V\to V$
    \item $\iota_x$ is the composition $D^{\op}\otimes_F F_x \xrightarrow{\iota} \End(V,\varphi)\otimes_F F_x\hookrightarrow \End(V_x,\varphi_x)$.
\end{enumerate}
%In other words, the underlying $\breve D_x$-module is induced by $V\otimes_{\breve F}\breve F_x$ and $\varphi_x$ is induced by $\varphi^{\deg(x)}$.
\end{defn}
Note that the localisation at $x$ is independent of the choice of $\bar x$ up to isomorphism. Indeed, if we choose another geometric point $\tau^a(\bar x)$ for some $a\in\ZZ$, then $\varphi^a$ induces an isomorphism from the localisation of $(V,\varphi,\iota)$ with respect to $\tau^a(\bar x)$ to the localisation with respect to $\bar x$.

A $D^\times$-isoshtuka (respectively, a local $D^\times_x$-isoshtuka) over $\cl\FF_q$ naturally gives rise to a $(D,\varphi)$-space (respectively, a Dieudonn\'e $D_x$-module); \emph{cf.} \S\ref{ssec-isoshtukas}. Furthermore, given a $\Dcal^\times$-shtuka $(\Eucal,\alpha)$, the Dieudonn\'e $D_x$-module attached to the local $D_x^{\times}$-isoshtuka of $(\Eucal,\alpha)$ coincides with the localisation at $x$ of the  $(D,\varphi)$-space associated to $(\Eucal,\alpha)$
%If $(V,\varphi,\iota)$ is the $(D,\varphi)$-space associated to a $\Dcal^\times$-shtuka $(\Eucal,\alpha)$, then its localisation $(V_x,\varphi_x,\iota_x)$ is compatible with the local $D^\times_x$-isoshtuka attached to $(\Eucal,\alpha)$ constructed in \eqref{eq:loc-G-isosht}.

Recall that the category of Dieudonn\'e $F_x$-modules is semisimple, with simple objects classified by their Newton slope; \emph{cf.} \cite[Theorem~B.3]{LaumonRapoportStuhler}. Furthermore, the simple Dieudonn\'e $F_x$-module $(W_{\mu,x},\psi_{\mu,x})$  with Newton slope $\mu\in\QQ$ has rank equal to the denominator of (the reduced fraction of) $\mu$, and its endomorphism ring $\End(W_{\mu,x},\psi_{\mu,x})$ is a central simple $F_x$-algebra with invariant $-\mu\bmod\ZZ$.
From this, one can show the semisimplicity of Dieudonn\'e $D_x$-modules, and obtain a simpler analogue of Corollary~\ref{cor:properties-D-varphi-space} for simple objects. We omit the details. 

Let $(V,\varphi,\iota)$ be a simple $(D,\varphi)$-space, and set $(L,\Pi)\coloneqq (L_{(V,\varphi)},\Pi_{(V,\varphi)})$. Since $L$ is an $F$-subalgebra of $\Delta\coloneqq\End_{(D,\varphi)}(V,\varphi,\iota)$, the localisation of $(V,\varphi,\iota)$ at each place $x\in |X|$ can be decomposed over the places $y$ of $L$ above $x$ as
\begin{equation}\label{eq:V-y}
    (V_x,\varphi_x,\iota_x) \simeq \bigoplus_{y|x} (V_y,\varphi_y,\iota_y),
\end{equation}
via the irreducible idempotents of $L\otimes_FF_x\cong \prod_{y|x}L_y$. Clearly, each summand $(V_y,\varphi_y,\iota_y)$ is a Dieudonn\'e $D_x$-module.

\begin{rmk}
  Let $(V,\varphi)$ be a $(D,\varphi)$-space with $\rank_{\breve D}(V) = 1$, and let $P$ denote the $D^\times$-torsor over $\breve F$ associated with $V$. Then strictly speaking, the $D^\times$-isoshtuka structure on $P$, which was denoted as $\varphi\colon\ltau P\riso P$ in \S\ref{sec-stratifications}, corresponds to the \emph{linearisation} of the $\tau$-linear map $\varphi\colon V\to V$. However, we prefer to  use $\varphi\colon V\to V$ to denote the $\tau$-linear map itself instead of its linearisation, since one can more directly define $\varphi^n\colon V\to V$ for any $n\geqslant1$. The same applies to local $D_x^\times$-isoshtukas and Dieudonn\'e $D_x$-modules.

  From now on, $\varphi$ on a $\breve D$-module (respectively, $\varphi_x$ on a $\breve D_x$-module) denotes a $\tau$-linear bijection (respectively, a $\tau_x$-linear bijection).
  %Strictly speaking, it is the \emph{linearisation} of a bijective $\tau$-linear and $D$-linear endomorphism $\varphi\colon V\to V$ that corresponds to the isomorphism $\ltau P\riso P$ of the $D^\times$-torsor $P$ associated with $V$, and likewise for $\varphi_x\colon V_x\to V_x$. However, we prefer to use $\varphi\colon V\to V$ and $\varphi_x\colon V_x\to V_x$ to denote the $\tau$-semilinear $\tau_x$-semilinear bijections, respectively, rather than their linearisations.
\end{rmk}

\begin{prop}[\emph{Cf.} {\cite[Proposition~2.4]{Lau:Degeneration}}]
  \label{proposition-properties-Dieudonne-D-x-module}
    In the above setting, the Dieudonn\'e $F_x$-module $(V_y,\varphi_y)$ underlying $(V_y,\varphi_y,\iota_y)$ is pure of (Newton) slope 
    \[\mu_y\coloneqq \frac{\deg_y(\Pi)}{[L_y: F_x]}.\] 
    and we have $\dim_{\breve F_x}(V_y) = d\cdot [L_y:F_x]\cdot d(\Delta)$, where $d(\Delta)$ is the index of the central division $L$-algebra $\Delta\coloneqq\End_{(D,\varphi)}(V,\varphi,\iota)$.
\end{prop}

The following congruence on the degree of $(V_x,\varphi_x)$ at ramified place $x$, which generalises \cite[Proposition~2.5]{Lau:Degeneration}, plays an important role in controlling the degeneration of $\Dcal^{\times}$-shtukas over ramified legs. (Recall that $\deg(V_x,\varphi_x)$ is the Newton slope of the determinant of $(V_x,\varphi_x)$ as a Dieudonn\'e $F_x$-module.)
%We need the following generalisation of \cite[Proposition~2.5]{Lau:Degeneration} in order to get full control on degeneration of $\Dcal^{\times}$-shtukas over ramified legs. Recall that the degree $\deg(V_x,\varphi)$ of a Dieudonn\'e $F_x$-module is the Newton slope of the determinant of $(V_x,\varphi_x)$.

\begin{lem}
  \label{lemma-Lau-bound-for-degree-of-D-varphi-space-outside-split-places}
  Let $m\geq0$ be an integer. Let $(V,\varphi,\iota)$ be a \emph{simple} $(D,\varphi)$-space over $\cl{\FF}_q$ with $\dim_{\breve F}V = d \cdot m$. We let $[-]_{\QQ} : \QQ/\ZZ \to \QQ\cap [0,1)$ denote a section of $\QQ\twoheadrightarrow \QQ/\ZZ$. Then for any closed point $x\in X$, we have
    \[
      \frac{\deg(V_x,\varphi_x)}{d} \equiv [m \inv_x(D)]_{\QQ} \bmod\ZZ
    \]
\end{lem}

\begin{proof}
  We set $(L,\Pi)\coloneqq (L_{(V,\varphi)},\Pi_{(V,\varphi)})$. Then by Proposition \ref{proposition-properties-Dieudonne-D-x-module} we have
  \begin{equation}\label{eq:deg-V-y}
      \deg(V_y,\varphi_y) = \mu_y\cdot \dim_{\breve F_x}(V_y) = d\cdot d(\Delta)\cdot \deg_y(\Pi),
  \end{equation}
  for any place $y$ of $L$, where $(V_y,\varphi_y)$ is defined in \eqref{eq:V-y} and $\Delta\coloneqq \End_{(D,\varphi)}( V,\varphi,\iota)$.

  Let $L'/L$ be a minimal extension that splits $\Delta$, so we have $[L':L]=d(\Delta)$. Since $\dim_{\breve F}(V) = dm$, Corollary~\ref{cor:properties-D-varphi-space} implies that
  \begin{equation}\label{eq:m}
      m = [L:F]\cdot d(\Delta) = [L:F]\cdot[L':L] = [L':F].
  \end{equation}
  
  Let $\Pi' \in (L')^{\times}\otimes_{\ZZ}\QQ$ be the image of $\Pi$. Then for any place $y'|y$ of $L'$ we have
  \begin{align}\label{eq:deg-y-Pi}
      \deg_{y'}(\Pi') &= [L'_{y'}:L_y]\cdot \deg_y(\Pi)\\
      &\equiv [\inv_{y'}(D\otimes_F L')]_\QQ \bmod\ZZ \notag
  \end{align}
  by Corollary~\ref{cor:properties-D-varphi-space}, since $\Delta\otimes_L L'_{y'}$ splits and 
  \[
    \inv_y(D\otimes_F L)\cdot [L_{y'}':L_y] = \inv_{y'}(D\otimes_F L')\in \QQ/\ZZ.
  \]

  Combining the above estimates, we obtain
  \begin{align*}
    \frac{\deg(V_x,\varphi_x)}{d}
    & = [L':L]\cdot \sum_{y|x,\text{ place of }L}  \deg_y(\Pi) =  \sum_{y'|x,\text{ place of }L'} \deg_{y'}(\Pi')  \\
    & \equiv  \sum_{y'|x,\text{ place of }L'} [\inv_{y'}(D\otimes_F L')]_{\QQ} \bmod\ZZ \\ 
    &\equiv \big[[L':F]\cdot\inv_x(D)\big]_\QQ \equiv [m\inv_x(D)]_\QQ\bmod\ZZ , 
  \end{align*}
  where the first equality follows from \eqref{eq:deg-V-y} and $[L':L]=d(\Delta)$, the  last congruence from \eqref{eq:m}, and the rest from \eqref{eq:deg-y-Pi}. This concludes the proof.  %This shows that the equality
%  \[
%   \frac{\deg(V_x,\varphi_x)}{d} = [m \inv_x(D)]_{\QQ} + \left \lfloor \frac{\deg(V_x,\varphi_x)}{d} \right \rfloor
%  \]
%  holds whenever $(V,\varphi,\iota)$ is a simple $(D,\varphi)$-space. 
\end{proof}

\section{Valuative criterion}\label{sec-valcrit}
% !TEX root = ./degeneration.tex

%In this section, we apply the results from \S\ref{sec-isosht} to show that $\Sht^{\leqslant\blambda}_{\Dcal^\times,I_\bullet}/a^\ZZ$ satisfies the valuative criterion for properness if a certain explicit inequality involving $D$ and $\blambda$ is satisfied. In particular, we show the sufficient condition for properness when $\Sht^{\leqslant\blambda}_{\Dcal^\times,I_\bullet}/a^\ZZ$  is known to be of finite type; \emph{cf.} Proposition~\ref{prop:quasi-compactness-of-Sht-D-mod-a}. 
In this section, we apply the results from \S\ref{sec-isosht} to obtain a criterion for properness of $\Sht^{\leqslant\blambda}_{\Dcal^\times,I_\bullet}/a^\ZZ$ (\emph{cf.} Theorem~\ref{thm:from-inequality-to-properness}).

\begin{notation}  \label{notation-D-shtukas}
The properness of  $\Sht^{\leqslant\blambda}_{\Dcal^\times,I_\bullet}/a^{\ZZ}$ depends only on $D$, $I$ and $\blambda$, but not on the choice of the partition $I_\bullet$ of $I$  (\emph{cf.} Corollary~\ref{cor:joining-modifications}) and the choice of the hereditary order $\Dcal$ of $D$ (\emph{cf.} Theorem~\ref{th:change-parahorics}). Therefore, we will make the following choice of $\Dcal$ and $I_\bullet$ for the rest of the paper.
%For the rest of the paper, we will adopt the following setting, which is more restrictive than the previous sections.

We fix a central division algebra $D$ over $F$ with dimension $d^2$ as before, and choose a \emph{maximal} order $\Dcal$ over $X$ once and for all. We shall apply the previous discussions to $G=D^\times$ and its integral model $\Gcal = \Dcal^\times$.

Let $I\coloneqq\{1,\cdots,r\}$ for some $r\geq1$, and set $I_\bullet = (\{j\})_{j=1,\cdots,r}$ unless stated otherwise. We also choose an $I$-tuple of dominant coweights $\blambda\in (X_\ast(T)_+)^I$; note that we identify $X_\ast(T)_+$ as length-$d$ sequences of decreasing integers.

Set $U\coloneqq X\setminus\Ram(D)$, which is also the maximal open subscheme where $\Dcal|_U$ is an Azumaya algebra (by maximality of $\Dcal$). Let $\breve U\subset \breve X$ denote the pullback of $U$.
\end{notation}

% statements corresponding to Lau
\begin{constr}\label{constr:adm-lattice}
A $(D,\varphi)$-space over $\cl{\FF}_q$ whose $\breve F$-dimension is less than $d^2$ may arise from a degenerating family of $\Dcal^{\times}$-shtukas, as described in \cite[\S3]{Lau:Degeneration}. We will review the construction and explain how to extend it when the legs are allowed to meet the ramification locus of $D$.

%Let us review the construction of a $(D,\varphi)$-space from a degeneration of $\Dcal^{\times}$-shtuka following \cite[\S3]{Lau:Degeneration}, whose method applies in the same way for $\Dcal^\times$-shtukas having ramified legs. % \yg{Lau only treated the case for shtukas with legs inside $U$, but the following construction works and is required for shtukas with legs possibly outside $U$. Should I mention this?}
Suppose that for a complete discrete valuation ring $A$ with residue field $\cl{\FF}_q$ and fraction field $K$, we are given a $\Dcal^\times$-shtuka over $K$ 
\[
  (\Eucal,\alpha) = ( (x_i)_{i=1,\cdots,r}, (\mathcal E_j)_{j=0,\cdots, r}, (\varphi_j)_{j=1,\cdots, r} , \alpha ) \in \operatorname{Sht}_{\mathcal{D}^\times ,I_{\bullet}}^{\leqslant {\blambda}}(K).
\]

We let $A'$ denote the local ring of $\breve X_A$ at the generic point of $\breve X_{A/\mathfrak m_A}$, and write $K'\coloneqq\Frac(A')$. Let $V_{K'}$ denote the generic fibre of $\mathcal E_0$, which is a right $(D\otimes_F K')$-module free of rank~$1$ equipped with the following $\tau$-linear morphism
\[
  \varphi_{K'}\colon
  \begin{tikzcd}[column sep=large]
      V_{K'} \arrow[r, "x\mapsto 1\otimes x"]&
      \ltau V_{K'} \arrow[r, "\cong"]&
      V_{K'}
  \end{tikzcd},
\]
where the second arrow is induced by $(\varphi_1)^{-1}\circ\cdots\circ (\varphi_r)^{-1} \circ \alpha$ as in \eqref{eq:varphi}. Note that the $D^\times$-isoshtuka over $K$ associated to $(\Eucal,\alpha)$ in \S\ref{ssec-isoshtukas} corresponds to $(V_{K'},\varphi_{K'},\iota_{K'})$, where $\iota_{K'}\colon D^{\op}\to\End(V_{K'},\varphi_{K'})$ denotes the map given by the $D$-action on $V_{K'}$.

For a finite extension $L/K$, let
\[
  (\underline{\mathcal E}_L,\alpha_L) =  \big(\big((x_{i,L})_{i\in I}, ( \mathcal E_{j,L} )_{j=0,\cdots, r}, (\varphi_{j,L})_{j=1,\cdots, r}\big),\alpha_L \big)\in\Sht^{\leqslant\blambda}_{\Dcal^\times,I_\bullet}(L)
\]
denote the pullback of $(\Eucal,\alpha)$ to $L$. Then, its associated isoshtuka is $(V_{L'},\varphi_{L'},\iota_{L'})$, where $V_{L'}\coloneqq V_{K'}\otimes_{K'}L'$, $\varphi_{L'}\coloneqq\varphi_{K'}\otimes\tau$, and $\iota_{L'}$ defined obviously.

Let $B$ denote the (discrete) valuation ring of $L$. The $L$-point $(x_{i,L})_{i\in I}$ of $X^I$ extends to a unique $B$-point $(\widetilde x_{i})_{i\in I}$ of $X^I$ by the valuative criterion. We write $B'$ for the local ring of $\breve X_B$ at the generic point of $\breve X_{B/\mathfrak m_B}$, and set $L'\coloneqq \Frac(B')$. Recall the following result of Drinfeld concerning the category of locally free sheaves on $X_B$.
\begin{prop}[{\cite[Proposition~3.1]{Drinfeld:CpctModuli}}]\label{prop:locally-free-and-D-lattices}
    The category of locally free sheaves on $X_B$ is equivalent to the category of pairs $(\mathcal B,M_{\mathcal B})$, where $\mathcal B$ is a locally free sheaf on $X_L$ and $M_{\mathcal B}$ is a $B'$-lattice in $\mathcal B|_{X_{L'}}$.
\end{prop}
In particular, this induces an equivalence between the category of $(D\otimes_F B')$-lattices in $V_{L'}$ and the category of extensions of $\mathcal E_{0,L}$ over $X_L$ to a locally free (right) $\Dcal_B$-module over $X_B$, by \cite[Lemma~1.16]{Lau:Degeneration}. 

We choose a $(D\otimes_F B')$-lattice $M_{B'}$ of $V_{L'}$.
%  then we obtain a $B$-point
% \begin{equation}\label{eq:E-tilde}
% \widetilde\Eucal\coloneqq ( (\widetilde x_i)_{i\in I}, (\widetilde{\mathcal E}_j)_{j=0,\cdots, r}, (\widetilde{\varphi}_j)_{j=1,\cdots, r} )\in\Hck_{\mathcal{D}^\times,I_{\bullet}}^{\leqslant {\blambda}}(B)
% \end{equation}
% extending $\underline{\mathcal E}_L\in\Hck_{\mathcal{D}^\times,I_{\bullet}}^{\leqslant {\blambda}}(L)$.
Since $V_{L'}$ is the generic fibre of $\mathcal E_{0,L}$ on $X_L$, the lattice $M_{B'}$ defines a locally free (right) $\mathcal D_B$-module $\widetilde{\mathcal E}_0$ on $X_B$ by Proposition~\ref{prop:locally-free-and-D-lattices}. Then $\widetilde{\mathcal E}_0$ and $\underline {\mathcal E}_L$ induce the uniquely determined 
\begin{equation}\label{eq:E-tilde}
\widetilde\Eucal\coloneqq ( (\widetilde x_i)_{i\in I}, (\widetilde{\mathcal E}_j)_{j=0,\cdots, r}, (\widetilde{\varphi}_j)_{j=1,\cdots, r} )\in\Hck_{\mathcal{D}^\times,I_{\bullet}}^{\leqslant {\blambda}}(B),
\end{equation}
since $\pr_0 \colon \Hck_{\mathcal D^{\times}, I_{\bullet}}^{\leqslant \blambda} \to \Bun_{\mathcal D^{\times}}$ is projective. In particular, $\widetilde {\varphi}_j$ is an isomorphism outside the graph of $\widetilde x_j$, and is bounded by $\lambda_j$. % The isomorphisms $\varphi_{j,L}$ over $L$ identify the vector space $V_{L'}$ with the generic fibre of $\mathcal E_{j,L}$ for every $j=0,\cdots, r$. The $\widetilde {\varphi}_j$ is an isomorphism outside the graph of $\widetilde{x}_i$ and is bounded by $\lambda_j$, since its generic fibre $\varphi_{j,L}$ over $L$ satisfies these conditions. These conditions are then inherited by the reduction $\cl {\varphi}_j$ over $\cl {\mathbb F}_q$.

If $M_{B'}$ is $\varphi_{L'}$-stable, then $\alpha_L$ restricts to an \emph{injective} morphism
\begin{equation}\label{eq:alpha-tilde}
    \widetilde{\alpha} \colon  
    \begin{tikzcd}
        \ltau \widetilde{\Ecal}_0 \arrow[r, hookrightarrow] & \widetilde{\Ecal}_r
    \end{tikzcd}.
\end{equation}

By \cite[Proposition~3.2]{Drinfeld:CpctModuli}, there is a (unique) maximal $\varphi_{L'}$-stable $B'$-lattice $M_{B'}\subset V_{L'}$, and it is easy to see that $M_{B'}$ is a $(D\otimes_F B')$-lattice. Furthermore,  \emph{loc.~cit.} also shows that there exists a finite extension $L/K$ such that $\varphi_{L'}$ on $M_{B'}$ is \emph{not} topologically nilpotent. (Such a lattice is called \emph{admissible}.) We assume that $M_{B'}$ is maximal and admissible.

Now, we obtain a $(D,\varphi)$-space $(N,\varphi_N,\iota_N)$ over $\cl{\FF}_q$ given by
\begin{equation}\label{eq:D-varphi-space-from-degeneration}
  N \coloneqq \bigcap_{n \ge 1}
\frac{\varphi_{L'}^{\,n}(M_{B'}) + \mathfrak m_B M_{B'}} {\mathfrak m_B M_{B'}} \subseteq M_{B'}/ \mfr_B M_{B'}
    % N\coloneqq \bigcap_{n\geq1} ({\varphi}_{L'}\mod {\mfr_{B}M_{B'}})^n(M_{B'}/\mfr_{B}M_{B'})
\end{equation}
with $\varphi_N$ induced by $\varphi_{L'}$. Clearly, $N$ is non-zero as $\varphi_{L'}$ is not topologically nilpotent on $M_{B'}$, so $(\dim_{\breve F}N)/d$ is a positive integer that is at most $d$.

We let 
\[
\big((\cl x_i)_{i=1,\cdots,r}, (\overline\Ecal_j)_{j=0,\cdots, r}, (\overline\varphi_j)_{j=1,\cdots, r} \big) \in \Hck^{\leqslant\blambda}_{\Dcal^\times,I_\bullet}(\cl\FF_q)
\]
denote the restriction of $\widetilde\Eucal$ to the closed point of $\Spec B$. Note that $\overline{\varphi}_j$ is an isomorphism outside $\cl x_j$, and is bounded by $\lambda_j$. Let $(N',\varphi_{N'},\iota_{N'})$ be any $(D,\varphi)$-subspace of $(N,\varphi_N,\iota_N)$ together with an identification of $N'$ with a $\varphi$-stable $\breve D$-submodule of the generic fibre of $\overline\Ecal_0$. Similar to the procedure defining \eqref{eq:E-tilde}, we obtain subbundles $\Fcal'_j\subset \overline\Ecal_j$ for $j=0,\cdots,r$ with generic fibres $N'$, equipped with the following diagram
\begin{equation}\label{eq:Fcal-modif}
    \begin{tikzcd}[column sep = large]
        \ltau\Fcal'_0 \arrow[r, hookrightarrow, "\beta'"] & \Fcal_r' & \arrow[l,dashrightarrow,"\overline\varphi_r|_{\Fcal'_{r-1}}"'] \ \cdots\  & \arrow[l,dashrightarrow,"\overline\varphi_2|_{\Fcal'_{1}}"'] \Fcal'_1 & \arrow[l,dashrightarrow,"\overline\varphi_1|_{\Fcal'_{0}}"']\Fcal'_0
    \end{tikzcd},
\end{equation}
where $\beta'\coloneqq\overline\alpha|_{\ltau\Fcal'_0}$, and the dashed arrows are isomorphisms 
\[
  \overline{\varphi}_{j}|_{\mathcal F'_{j-1}} \colon \mathcal F'_{j-1}|_{\breve X \smallsetminus \{\cl x_j\}}\riso \mathcal F'_{j}|_{\breve X \smallsetminus \{\cl x_j\} }.
\]
The latter are no longer bounded by $\lambda_j$, and there will be only the estimate \eqref{eq:Rel-pos-bounds}. It can even happen that $\overline{\varphi}_{j}|_{\mathcal F'_{j-1}}$ is an isomorphism on all of $\breve X$.
\end{constr}

\begin{rmk}
    Although each modification $\overline{\varphi}_j|_{\Fcal'_{j-1}}$ is an isomorphism away from $\cl x_j$, the support of $\coker(\beta')$ need not be contained in the set of legs. We illustrate this phenomenon with an example, suggested by the referee, involving degenerations of $\GL_2$-shtukas rather than $\Dcal^\times$-shtukas. Similar examples can be constructed for maximal orders in suitable quaternion division algebras, but the \(\GL_2\)-example is simpler to explain.

    Let $X\coloneqq \PP^1_{\FF_q}$ and set $\FF_q[t] = \Gamma(X\setminus\infty,\Ocal_X)$. For $B\coloneqq \cl\FF_q\pot{u}$ with fraction field $L\coloneqq\Frac(B)$, we set $x_1,x_2\in X(B)$ so that $x_2 = \infty$ and $x_1$ is given by
    \[x_1^\ast\,\colon\, \FF_q[t]\to B, \qquad t\mapsto 1+u.\]
    Now we consider the following $\GL_2$-shtuka over $L$:
    \[
    (\Eucal,\alpha) \coloneqq \bigl(
        (x_1,x_2),(\Ecal_j)_{j=0}^{2}, (\varphi_1,\varphi_2),\alpha=\id
    \bigr),
    \]
    where $\Ecal_0 = \Ecal_2 = \Ocal_{X_L}^{\oplus 2}$, $\Ecal_1 = \Ecal_0(-\infty)$, $\varphi_2$ is the natural inclusion, and $\varphi_1^{-1}$ is given by
    \[
    \begin{pmatrix}
        0 & t-(1+u) \\
        t-(1+u) & u^{1-q}t
    \end{pmatrix}.
    \]
    Here, under the natural identifications of the generic fibres of $\Ecal_0$ and $\Ecal_1$ with $L(t)^{\oplus 2}$, the displayed matrix defines an injective morphism $\varphi_1^{-1}\colon\Ecal_1\hookrightarrow\Ecal_0$ over $X_L$ with cokernel supported at $x_1$.
    Then a direct computation gives $M_{B'} = (uB')^{\oplus 2}$.
    The space $N\subset M_{B'}/uM_{B'}$ is spanned by $e_2$, the class of $(0,u)$, and the induced Frobenius is given by $\varphi_N(e_2) = te_2$. Therefore, by taking $N' = N$, we have $\coker(\beta') \cong \cl\FF_q[t]/(t)$, which is supported at $0\in X(\cl \FF_q)$, disjoint from the legs $\cl x_1 = 1$ and $\cl x_2 = \infty$. This example is essentially worked out in \cite[\S5.3, proof of Proposition~5.11]{Gardeyn:Analytic-tau-sheaves} in the language of $\tau$-sheaves.
\end{rmk}

    Although $\coker(\beta')$ need not be supported at the legs, one can still bound its length. The following lemma generalises \cite[\S3]{Lau:Degeneration} to the case where the legs are not necessarily disjoint from $\Ram(D)$.

\begin{lem}\label{lem:adm-lattice}
    In the setting of \S\ref{constr:adm-lattice}, the following properties hold.
    \begin{enumerate}
        \item\label{lem:adm-lattice:good-redn} 
        For $N$ as in \eqref{eq:D-varphi-space-from-degeneration}, we have $\dim_{\breve F}(N) = d^2$ if and only if $\widetilde\alpha$ defined in \eqref{eq:alpha-tilde} is an isomorphism and $(\widetilde\Eucal,\widetilde\alpha)$ defines a $B$-point of $\Sht^{\leqslant\blambda}_{\Dcal^\times,I_\bullet}$. In this case, the diagram in \eqref{eq:Fcal-modif} for $N'=N$ defines a $\Dcal^\times$-shtuka over $\cl\FF_q$ isomorphic to the special fibre of $(\widetilde\Eucal,\widetilde\alpha)$.
        \item\label{lem:adm-lattice:ineq} Let $(N',\varphi_{N'},\iota_{N'})$ be a non-zero $(D,\varphi)$-subspace of $(N,\varphi_N,\iota_N)$.
        If $\dim_{\breve F}(N') = dm$, then we have an inequality
        \[
       \dim_{\cl\FF_q}\big(\coker(\beta'\colon\ltau\Fcal'_0\hookrightarrow\Fcal'_r) \big) \leq d\cdot\sum_{i=1}^r\sum_{j=1}^{d-m}\lambda_{i,j}.
        \]
    \end{enumerate}
\end{lem}
In Claim~(\ref{lem:adm-lattice:ineq}), if $m=d$ then both sides of the inequality reduce to $0$.
%Claim~(\ref{lem:adm-lattice:ineq}) was proved in \cite[Lemma~3.1]{Lau:Degeneration} when none of $\cl x_i$'s lie over $\Ram(D)$. We extend the proof of \emph{loc.~cit.} allowing $\cl x_i$ to be any geometric point.
\begin{proof}
    Claim~(\ref{lem:adm-lattice:good-redn}) is clear from the observation that $\dim_{\breve F}(N) = d^2$ if and only if $\varphi_{L'}(M_{B'})$ generates $M_{B'}$. To show Claim~(\ref{lem:adm-lattice:ineq}), we consider the forgetful morphism
    \[
      \Hck^{\leqslant\lambda_i}_{\Dcal^\times,(\{i\})}\to \Hck^{\leqslant\lambda_{i,\GL}}_{\GL_{d^2},(\{i\})},
    \]
     which is well-defined by \eqref{eq:functoriality-BD-Gr}. (Here,
 $\lambda_{i,\GL}$ is defined in \eqref{eq:lambda-GL}.) By applying the forget morphism to the object
    \begin{equation}\label{eq:Rel-Position}
        \big((\cl x_i), \begin{tikzcd}[column sep = 20pt]
        \overline\Ecal_{i-1} \arrow[r,dashrightarrow,"\overline\varphi_i"] &\overline\Ecal_{i}
    \end{tikzcd} \big) \in\Hck^{\leqslant\lambda_i}_{\Dcal^\times,(\{i\})}(\cl\FF_q),
    \end{equation}
    it follows that the relative position of the formal stalks $(\widehat{\overline\Ecal}_{i-1})_{\cl x_i}$ and  $(\widehat{\overline\Ecal}_{i})_{\cl x_i}$ with respect to  $\overline\varphi_i$ is bounded above by $\lambda_{i,\GL}$. Now we may repeat the  proof of \cite[Lemma~3.1]{Lau:Degeneration} to prove (\ref{lem:adm-lattice:ineq}). In fact, we have
    \begin{align*}
        \dim_{\cl\FF_q}\big(\coker(\beta') \big) 
        & = \deg (\Fcal'_r) - \deg (\ltau\Fcal'_0)\\
        & = \deg (\Fcal'_r) - \deg (\Fcal'_0)\\
        &= \sum_{i=1}^r \big( \deg(\Fcal'_i) - \deg(\Fcal'_{i-1})\big).
        % &\leq d\sum_{i=1}^r\sum_{j=1}^{d-m}\lambda_{i,j},
    \end{align*}

    Now, given two $\Ocal_{\cl x_i}$-lattices $\Lambda,\Lambda'$ in a common finite-dimensional $F_{\cl x_i}$-vector space, we define the relative index as follows
    \[
    [\Lambda:\Lambda'] \coloneqq \dim_{\cl\FF_q}\big( \Lambda/ (\Lambda\cap\Lambda') \big) - \dim_{\cl\FF_q}\big( \Lambda' / (\Lambda\cap\Lambda') \big).
    \]
    Since $\cl {\varphi}_i | _{\mathcal F_{i-1}'}$ is an isomorphism outside $\cl {x}_i$, the difference $\deg(\mathcal F'_i)-\deg(\mathcal F'_{i-1})$ is equal to the relative index $[(\widehat{\Fcal}'_{i})_{\cl x_i} : (\widehat{\Fcal}'_{i-1})_{\cl x_i}]$ of the formal stalks $(\widehat{\mathcal F}'_{i-1})_{\cl x_i}$ and $(\widehat{\mathcal F'_i})_{\cl x_i}$ of $\mathcal F'_{i-1}$ and $\mathcal F'_i$, respectively, at $\cl x_i$. 
    Then from \eqref{eq:Rel-Position} we obtain the following bounds 
    \begin{equation}\label{eq:Rel-pos-bounds}
        %d\cdot\sum_{j=d-m+1}^d\lambda_{i,j} \leq [(\widehat\Fcal_{i})_{\cl x_i} : (\widehat\Fcal_{i-1})_{\cl x_i}]\leq d\cdot\sum_{j=1}^m \lambda_{i,j}
        -d\cdot\sum_{j=1}^m \lambda_{i,j} \leq 
        [(\widehat\Fcal_{i}')_{\cl x_i} : (\widehat\Fcal_{i-1}')_{\cl x_i}]\leq -d\cdot\sum_{j=d-m+1}^d\lambda_{i,j},
    \end{equation}
    Since $\sum_{i=1}^r\deg(\lambda_i)=0$, we have
    \[
    \dim_{\cl\FF_q}\big(\coker(\beta') \big) \leq -d\cdot \sum_{i=1}^r\sum_{j=d-m+1}^d\lambda_{i,j} = d\cdot\sum_{i=1}^r\sum_{j=1}^{d-m}\lambda_{i,j},
    \]
    which proves Claim~(\ref{lem:adm-lattice:ineq}).
\end{proof}

\begin{lem}\label{lem:adm-lattice-simple} 
In the setting of \S\ref{constr:adm-lattice}, let $(N',\varphi_{N'},\iota_{N'})$ be a \emph{simple} $(D,\varphi)$-subspace of $(N,\varphi_N,\iota_N)$ with $\dim_{\breve F}(N') = dm$. Then, the following properties hold.
    \begin{enumerate}
        \item\label{lem:adm-lattice:lower} For a closed point $x\in X$, let $I_x\subset I$ be the subset of $i\in I$ such that $\cl x_i$ lies over $x$. Then we have
        \[ \frac{\deg(N'_x,\varphi_{N',x})}{d} \geq [m\inv_x(D)]_\QQ +\sum_{i\in I_x}\sum_{j=d-m+1}^d \lambda_{i,j}.\]
%        \begin{align*}
%            \frac{\deg(N_x,\varphi_{N,x})}{d} &\geq [m\inv_x(D)]_\QQ +\sum_{i\in I_x}\sum_{j=d-m+1}^d \lambda_{i,j},\quad \text{and}\\
%            \frac{\deg(N_x,\varphi_{N,x})}{d} &\leq -[-m\inv_x(D)]_\QQ +\sum_{i\in I_x}\sum_{j=1}^m \lambda_{i,j}.
%        \end{align*}
%        \[
%        [m\inv_x(D)]_\QQ +\sum_{i\in I_x}\sum_{j=d-m+1}^d \lambda_{i,j}\leq  \frac{\deg(N_x,\varphi_{N,x})}{d} \leq-[-m\inv_x(D)]_\QQ +\sum_{i\in I_x}\sum_{j=1}^m \lambda_{i,j}.
%        \]
        Furthermore, if $I_x=\emptyset$ then all slopes of $(N'_x,\varphi_{N',x})$ are non-negative.
%        \item\label{lem:adm-lattice:upper}  For any finite non-empty subset $Y\subseteq X$, we have the following inequality
%        \begin{align*}
%            \sum_{x\in Y}\frac{\deg(N_x,\varphi_{N,x})}{d} &\leq  \sum_{x\in Y}\sum_{i\in I_x}\sum_{j=1}^m \lambda_{i,j} +\frac{1}{d}\dim_{\cl\FF_q}\big(\coker(\beta\colon\ltau\Fcal_0\hookrightarrow\Fcal_r) \big)
%            %\\&\leq \sum_{x\in S}\sum_{i\in I_x}\sum_{j=1}^m \lambda_{i,j} +\sum_{i=1}^r\sum_{j=1}^{d-m}\lambda_{i,j}.
%        \end{align*}
        \item\label{lem:adm-lattice:coker-identity} For a closed point $x\in X$, we have
        \[
            \sum_{\cl x\mid x} \dim_{\cl\FF_q}\big(\coker(\beta')_{\cl x}\big) = \deg(N'_x,\varphi_{N',x}) + \sum_{i\in I_x}[(\widehat{\Fcal'_i})_{\cl x_i}:(\widehat{\Fcal'_{i-1}})_{\cl x_i}],
        \]
        where the sum is over all geometric points $\cl x$ over $x$.
    \end{enumerate}
\end{lem}
\begin{proof}
        To prove Claim~(\ref{lem:adm-lattice:lower}), fix a closed point $x\in X$. If $I_x = \emptyset$ then all slopes of $(N'_x,\varphi_{N',x})$ are non-negative since the formal stalk $(\widehat\Fcal'_0)_{\cl x}$ at a geometric point $\cl x$ over $x$ defines a $\varphi_{N',x}$-stable lattice of $N'_x$; in fact, all maps in \eqref{eq:Fcal-modif} induce isomorphisms of the formal stalks at $\cl x$ possibly except $\beta'$. To obtain the inequality in (\ref{lem:adm-lattice:lower}), note that we have
    \begin{equation}\label{eq:deg-index}
        \deg(N'_x,\varphi_{N',x}) = %[\varphi\big(\prod_{\cl x}\ltau\widehat{\Fcal_0})_{\cl x}\big):\prod_{\cl x}(\widehat{\Fcal_0})_{\cl x} ],
        \sum_{\cl x | x}\big[(\widehat{\Fcal'_0})_{\cl x}:(\ltau\widehat{\Fcal'_0})_{\cl x}\big],
    \end{equation}
     where the sum is over all geometric points $\cl x$ over $x$. Here, the relative index is for $\Ocal_{\cl x}$-lattices in $N'\otimes_{\breve F}F_{\cl x}$ via the embeddings $(\widehat{\Fcal_0'})_{\cl x} \subset  N'\otimes_{\breve F}F_{\cl x}$ and
    \[(\ltau\widehat{\Fcal'_0})_{\cl x} \subset 
    \begin{tikzcd}
        (\ltau N')\otimes_{\breve F} F_{\cl x} \arrow[r, ""', "\cong"] & N'\otimes_{\breve F}F_{\cl x}
    \end{tikzcd}\]
    where the isomorphism is obtained by tensoring with $F_{\cl x}$ the linearisation of the $\tau$-linear map $\varphi_{N'}$.
    
     Now, for any $\cl x\in X(\cl\FF_q)$ we have
    \begin{equation}\label{eq:index-at-geom-pt}
        \big[(\widehat{\Fcal'_0})_{\cl x}:(\ltau\widehat{\Fcal'_0})_{\cl x}\big]  = \dim_{\cl\FF_q}\big((\coker\beta')_{\cl x}\big) - \sum_{i=1}^r [(\widehat{\Fcal'_i})_{\cl x}:(\widehat{\Fcal'_{i-1}})_{\cl x}]
    \end{equation}
    
    For each $\cl x \mid x$, since $\overline {\varphi}_i | _{\mathcal F_{i-1}'}$ is an isomorphism outside $\cl{x}_i$, the relative index $[(\widehat{\Fcal}'_{i})_{\cl x}:(\widehat{\Fcal'_{i-1}})_{\cl x}]$ can be nonzero only if $\cl x=\cl{x}_i$, and hence only if $i\in I_x$. This implies that
    \begin{equation}\label{eq:sum-only-I-x}
      \sum_{\cl x \mid x} \sum_{i=1}^r \big[(\widehat{\Fcal'_i})_{\cl x}:(\widehat{\Fcal'_{i-1}})_{\cl x}\big] = \sum_{i\in I_x} \big[(\widehat{\Fcal'_i})_{\cl x}:(\widehat{\Fcal'_{i-1}})_{\cl x_i}\big].
    \end{equation}
    Since obviously $\dim_{\cl\FF_q}\big((\coker\beta')_{\cl x}\big)\geq0$, we deduce the following inequality  
    \begin{align}\label{eq:Newton-bounds-at-x}
    %    \sum_{i\in I_x}\sum_{j=d-m+1}^d \lambda_{i,j} \leq \frac{\deg(N_x,\varphi_{N,x})}{d} \leq  \sum_{i\in I_x}\sum_{j=1}^m \lambda_{i,j} + \frac{1}{d}\sum_{\cl x | x} \dim_{\cl\FF_q}\big(\coker(\beta)_{\cl x} \big)
    \frac{\deg(N'_x,\varphi_{N',x})}{d} 
    \overset{\eqref{eq:deg-index}} = & \sum_{\cl x | x}\frac{\big[(\widehat{\Fcal'_0})_{\cl x}:(\ltau\widehat{\Fcal'_0})_{\cl x}\big]}{d} 
    \overset{\eqref{eq:index-at-geom-pt}}\geq -\sum_{\cl x \mid x}\sum_{i=1}^r \frac{[(\widehat{\Fcal'_i})_{\cl x}:(\widehat{\Fcal'_{i-1}})_{\cl x}]}{d} \\
    \overset{\eqref{eq:sum-only-I-x}}= & - \sum_{i\in I_x} \frac{\big[(\widehat{\Fcal'_i})_{\cl x}:(\widehat{\Fcal'_{i-1}})_{\cl x_i}\big]}{d} 
    \overset{\eqref{eq:Rel-pos-bounds}}\geq \sum_{i\in I_x}\sum_{j=d-m+1}^d \lambda_{i,j}. \notag
    \end{align}
    Since the final term of the inequality is an integer and $(N',\varphi_{N'},\iota_{N'})$ is assumed to be a simple $(D,\varphi)$-space, we may add $[m\inv_x(D)]_\QQ$ to the lower bound by Lemma~\ref{lemma-Lau-bound-for-degree-of-D-varphi-space-outside-split-places}, and thus the inequality in (\ref{lem:adm-lattice:lower}) immediately follows. 
    
    Claim~(\ref{lem:adm-lattice:coker-identity}) follows directly from \eqref{eq:deg-index}, \eqref{eq:index-at-geom-pt} and \eqref{eq:sum-only-I-x}.
\end{proof}

\begin{rmk}\label{rmk:stupid-sanity}
    Let us make some remarks on the case when $\dim_{\breve F}N = d^2$. By Lemma~\ref{lem:adm-lattice}(\ref{lem:adm-lattice:good-redn}), this implies that  $\coker(\beta) = 0$ and $(N_x,\varphi_{N,x})$ is a Dieudonn\'e $D_x$-module free of rank-$1$ over $\breve D_x$ for any $x\in |X|$. On the other hand, summing up the identity in Lemma~\ref{lem:adm-lattice-simple}(\ref{lem:adm-lattice:coker-identity}) over all simple $(D,\varphi)$-subspaces of $(N,\varphi_N,\iota_N)$ we obtain
\begin{equation}\label{eq:Mazur-degeneration}
0=\frac{1}{d}\dim_{\cl\FF_q}\big(\coker(\beta)\big) = \sum_{x\in |X|}\frac{\deg(N_x,\varphi_{N,x})}{d} - \sum_{i\in I}\deg(\lambda_i).
\end{equation}
Furthermore, Lemma~\ref{lem:adm-lattice}(\ref{lem:adm-lattice:good-redn}) ensures the existence of a $\Dcal^\times$-shtuka $(\Fucal,\beta)\in\Sht^{\leqslant\blambda}_{\Dcal^\times,I_\bullet}(\cl\FF_q)$ whose associated $D^\times$-isoshtuka is $(N,\varphi_N)$. In this context, the above equation reduces to the well-known consequence of Mazur's inequality (Proposition~\ref{prop:Adm}): \emph{the Newton and Hodge polygons share the same endpoint}. 

To explain, we adopt the notation from \S\ref{ssec:fixed-legs}, particularly \eqref{eq:ybf}. For a closed point $y\in |X|$ lying under a leg of $(\Fucal,\beta)$, the isomorphism class of the Dieudonn\'e $D_y$-module associated to $(\Fucal,\beta)$ is given by 
\[
b_y\in B(D_y^\times,\lambda_y),\quad \text{where}\ \lambda_y=\sum_{i\in I_y}\lambda_i;
\]
(See Proposition~\ref{prop:Adm}; note that $\ltau\lambda_i=\lambda_i$ for any $\lambda_i\in X_\ast(T)_+$, as $D_y^\times$ is an inner form of $\GL_d$.) Now, by embedding $B(D_y^\times,\lambda_y)$ into $B(\GL_d,\lambda_y)$ as explained in \cite[\S6]{Kottwitz:Gisoc2} and comparing the $y$-coordinates of the endpoints of the Hodge and Newton polygons (or alternatively, by applying the natural map $B(D^\times_y,\lambda_y)\to B(\GG_m,\deg\lambda_y)$ induced by the reduced norm), we obtain
\[
\frac{\deg(N_y,\varphi_{N,y})}{d} = \deg(\lambda_y) = \sum_{i\in I_y}\deg(\lambda_i).
\]
Summing this equality over all closed points recovers \eqref{eq:Mazur-degeneration}.
%%%%%%Old version
%which is compatible with Proposition~\ref{prop:Adm} in the following sense. Let $Y_N\coloneq \{ y\in |X| \mid \cl x_i \text{ lies over } y\text{ for some }i\}$. For each $y\in Y_N$, the summand $\deg(N_y,\varphi_{N,y})/d$ is determined by $[b_y]\in B(D_y)$. On the other hand, since we have $\deg({}^{\tau^{a}}\lambda_i)=\deg(\lambda_i)$ for every $0\leq a\leq \deg(y)-1$, we have
%\[
%\sum_{i\in I}\deg(\lambda_i)=\sum_{y\in Y_N}\deg(\lambda_y).
%\]
%with $\lambda_{y}$ defined as in Proposition~\ref{prop:Adm}. 
%
%%%%%%%%%% Older version
%In this sense, \eqref{eq:Mazur-degeneration} is the consequence of the admissibility conditions $[b_y]\in B(G_y,\lambda_y)$ of Proposition~\ref{prop:Adm} after passing to degree. Thus, Lemma~\ref{lem:adm-lattice-simple} can loosely be thought of as the \emph{degeneration} of Mazur's inequality.

Therefore,  Lemma~\ref{lem:adm-lattice-simple} can be thought of the \emph{degeneration} of \eqref{eq:Mazur-degeneration}; or more loosely, the \emph{degeneration of Mazur's inequality}.
\end{rmk}

We can now prove the following criterion for properness. %(with respect to complete discrete valuation rings with residue field $\cl\FF_q$

\begin{thm}
  \label{thm:from-inequality-to-properness}
Assume that $\Sht_{\Dcal^\times,I_\bullet}^{\leqslant\blambda}$ is non-empty (\emph{cf.} Lemma~\ref{lem:non-empty-Sht}), and let $a\in \AA_F^{\times}$ be any id\`ele of positive degree. Then the stack $\Sht_{\Dcal^\times,I_\bullet}^{\leqslant\blambda} / a^{\ZZ}$ is proper over $\cl\FF_q$ if and only if, for every integer $0< m <d$, we have
  \begin{equation}  \label{eq:inequality-for-properness}  
      \sum_{x\in \Ram(D)} [m \inv_x (D)]_{\QQ}  > \sum_{i\in I}\sum_{j=1}^{d-m} \lambda_{i,j} .
  \end{equation}
\end{thm}

In the proof of the theorem, we directly show the quasi-compactness of $\Sht_{\Dcal^\times,I_\bullet}^{\leqslant\blambda} / a^{\ZZ}$ without using Proposition~\ref{prop:quasi-compactness-of-Sht-D-mod-a}(\ref{prop:quasi-compactness-of-Sht-D-mod-a:no-restriction}). In fact, it is unclear if the assumption in the theorem implies the assumption for Proposition~\ref{prop:quasi-compactness-of-Sht-D-mod-a}(\ref{prop:quasi-compactness-of-Sht-D-mod-a:no-restriction}) unless $D_x$ is a division algebra for any $x\in \Ram(D)$.

\begin{proof}
  Suppose that $\Sht_{\Dcal^\times,I_\bullet}^{\leqslant\blambda}/a^{\ZZ}$ is proper over $\cl\FF_q$. Its restriction to $(X-\Ram(D))^I$ is then proper over $(X-\Ram(D))^I$, so \cite[Theorem~A]{Lau:Degeneration} yields the displayed inequalities.
  
  Now we prove the converse direction. Recall that $\Sht_{\Dcal^\times,I_\bullet}^{\leqslant\blambda} / a^{\ZZ}$ is separated and locally of finite type over $\FF_q$ \cite[Theorem~3.15]{ArastehRad-Hartl:Uniformizing}, so we have the uniqueness part of the valuative criterion. Therefore, we need to show that the inequality~\eqref{eq:inequality-for-properness} implies quasi-compactness and the existence part of the valuative criterion \cite[\href{https://stacks.math.columbia.edu/tag/0CQM}{Tag~0CQM}]{AlgStackProj} for $\Sht_{\Dcal^\times,I_\bullet}^{\leqslant\blambda} / a^{\ZZ}$. Note that if quasi-compactness is already known, then to verify properness it suffices to check the valuative criterion for discrete valuation rings that are essentially of finite type over $\cl\FF_q$, using \cite[Lemma~A.11]{AlperHalpenLeistnerHeinloth:ModuliForAlgebraicStacks}. In particular, we may restrict the test rings for the valuative criterion to complete discrete valuation rings over $\cl\FF_q$ with residue field $\cl\FF_q$.

  Let $A$ be a complete discrete valuation ring over $\cl\FF_q$ with residue field $\cl\FF_q$, and set $K\coloneqq\Frac(A)$. Choose a $\Dcal^{\times}$-shtuka
  \[
    ( \Eucal , \alpha ) = ( (x_i)_{i\in I}, (\mathcal E_j)_{j}, (\varphi_j)_{j} , \alpha )\in \Sht^{\leqslant\blambda}_{\Dcal^\times, I_\bullet} ( K ),
  \]
  and construct a $(D,\varphi)$-space $(N,\varphi_N,\iota_N)$ with respect to some finite extension $L/K$ as in \eqref{eq:D-varphi-space-from-degeneration}. %Let $(N',\varphi_{N'},\iota_{N'})$ be a simple $(D,\varphi)$-subspace of $(N,\varphi_N,\iota_N)$. We claim that
  Under the assumption of the theorem, we claim that
  \begin{equation}\label{claim:inequality-for-properness}
      \textit{$(N,\varphi_N,\iota_N)$ is a simple $(D,\varphi)$-space with }\dim_{\breve F}N = d^2.
  \end{equation}
  To show the claim, let $(N',\varphi_{N'},\iota_{N'})$ be a simple $(D,\varphi)$-subspace and set $ \dim_{\breve F} N' = dm$. Then Lemma~\ref{lem:adm-lattice-simple}\eqref{lem:adm-lattice:lower} says
  \begin{equation}
    \frac{\deg(N'_x,\varphi_{N',x})}{d} \geq [m\inv_x(D)]_\QQ +\sum_{i\in I_x}\sum_{j=d-m+1}^d \lambda_{i,j}
  \end{equation}
  for \emph{every} closed point $x$ of $X$.
  Summing over all closed points $x$, we obtain
  \begin{equation}
      0 = \sum_{x\in |X|}\frac{\deg(N'_x,\varphi_{N',x})}{d} \geq \sum_{x\in \Ram(D)} [m\inv_x(D)]_\QQ +\sum_{i\in I}\sum_{j=d-m+1}^d \lambda_{i,j},
  \end{equation}
    where the first equality holds since, by \eqref{eq:V-y} and \eqref{eq:deg-V-y}, the sum $\sum_{x}\deg(N'_x,\varphi_{N',x})$ is a multiple of $\sum_{y}\deg_y(\Pi_{(N',\varphi_{N'})})$, with $y$ running over all places of $L_{(N',\varphi_{N'})}$, and the latter sum vanishes by the product formula. The inequality can be rewritten as
  \begin{equation}
      \sum_{x\in \Ram(D)} [m\inv_x(D)]_\QQ \leq \sum_{i\in I}\sum_{j=1}^{d-m} \lambda_{i,j},
  \end{equation}
  since $\sum_{i\in I}\deg(\lambda_i) = 0$ by Lemma~\ref{lem:non-empty-Sht}.
  This forces $m=d$, and hence we have $N'=N$. This proves Claim~\eqref{claim:inequality-for-properness}.

  Let us now deduce the theorem. Suppose that $(\Eucal,\alpha)$ is the pullback to $K$ of $(\Eucal_{\cl\FF_q},\alpha_{\cl\FF_q})\in \Sht^{\leqslant\blambda}_{\Dcal^\times, I_\bullet} (\cl\FF_q)$. Then $(\Eucal_{\cl\FF_q},\alpha_{\cl\FF_q})$ is \emph{irreducible} in the sense of Definition~\ref{def:D-subshtukas}, since its associated $(D,\varphi)$-space $(N,\varphi_N,\iota_N)$ is simple by \eqref{claim:inequality-for-properness}. By repeating the proof of Proposition~\ref{prop:quasi-compactness-of-Sht-D-mod-a} using Lemma~\ref{lem:bounding-HN-parameter-of-Sht-D}, it follows that the image of $\Sht^{\leqslant\blambda}_{\Dcal^\times,I_\bullet}(\cl\FF_q)$ in $\Bun_{\GL_{d^2}}/a^\ZZ$ is contained in some quasi-compact open substack. This shows that $\Sht_{\Dcal^\times,I_\bullet}^{\leqslant\blambda} / a^{\ZZ}$ is of finite type over $\FF_q$, as it is already locally of finite type over $\FF_q$.

  For an arbitrary $( \Eucal , \alpha )\in \Sht^{\leqslant\blambda}_{\Dcal^\times, I_\bullet} ( K )$, Claim~\eqref{claim:inequality-for-properness} together with Lemma \ref{lem:adm-lattice}(\ref{lem:adm-lattice:good-redn}) implies that the valuative criterion for properness is satisfied for $\Sht_{\Dcal^\times,I_\bullet}^{\leqslant\blambda} / a^{\ZZ}$ with respect to complete discrete valuation rings over $\cl\FF_q$ with residue field $\cl\FF_q$. This now implies the desired properness as explained in the second paragraph.
\end{proof}

\begin{rmk}\label{rmk:previous-properness-criterion}
The previous version of this paper proved Theorem~\ref{main-thm-previous} by comparing the following bounds on the length of $\coker(\beta')$ for a simple $(D,\varphi)$-subspace $N'\subseteq N$ with $\dim_{\breve F}N'=dm$:
\[
  \sum_{x\in Y'}[m\inv_x(D)]_{\QQ} \leq \frac{1}{d}\dim_{\cl\FF_q}(\coker(\beta')) \leq \sum_{i\in I}\sum_{j=1}^{d-m}\lambda_{i,j},
\]
where $Y'\subseteq\Ram(D)$ is the set of ramified places hit by no legs or only by central legs. Here, the upper bound is Lemma~\ref{lem:adm-lattice}(\ref{lem:adm-lattice:ineq}), and the lower bound essentially follows from Lemma~\ref{lem:adm-lattice-simple}. Since there are $|I^{\nc}|$ non-central legs, at most $|I^{\nc}|$ places in $\Ram(D)$ lie outside $Y'$, so these bounds contradict inequality~\eqref{eq:inequality-for-properness-previous} if $m<d$.
\end{rmk}

\begin{exa}\label{exa:ramified-legs-quaternion}
We compare the argument in Remark~\ref{rmk:previous-properness-criterion} with the proof of Theorem~\ref{thm:from-inequality-to-properness} in the case where $D$ is a quaternion division algebra (so $d=2$) and $\blambda=((1,0),(0,-1))$.

Let $(N',\varphi_{N'},\iota_{N'})$ be a simple $(D,\varphi)$-subspace of $(N,\varphi_N,\iota_N)$ with $\dim_{\breve F}N' = 2$, i.e., $m=1$. For a closed point $x\in X$, the local defect measuring the failure of $\mathcal F'_{\bullet}$ to be a genuine shtuka of rank $dm$ with minimal Hodge contribution (i.e., whose modifications at $x$ have relative position $(\lambda_{i,2},\lambda_{i,2})$ for each $i\in I_x$) can be written as
\[
  c_x\coloneqq{}\sum_{\cl x\mid x}\dim_{\cl\FF_q}\left(\coker(\beta')_{\cl x}\right)+\sum_{i\in I_x}\left(-2\lambda_{i,2}-[(\widehat{\Fcal'_i})_{\cl x_i}:(\widehat{\Fcal'_{i-1}})_{\cl x_i}]\right)
\]
We have $c_x\geq0$ since $[(\widehat{\Fcal'_i})_{\cl x_i}:(\widehat{\Fcal'_{i-1}})_{\cl x_i}] \leq -2\lambda_{i,2}$ by \eqref{eq:Rel-pos-bounds}. Moreover, Lemma~\ref{lem:adm-lattice-simple}(\ref{lem:adm-lattice:lower}) and~(\ref{lem:adm-lattice:coker-identity}) imply that
\[
  c_x = \deg(N'_x,\varphi_{N',x})-2\sum_{i\in I_x}\lambda_{i,2} \geq 2[\inv_x(D)]_\QQ= \begin{cases}
    1 & \text{if }x\in\Ram(D), \\
    0 & \text{if }x\notin\Ram(D).
  \end{cases}
\]
Now we compare the cases $x\in\Ram(D)$ and $x\notin\Ram(D)$. The lower bound for $c_x$ is larger by $1$ in the ramified case. If no leg lies over $x$ as in the setting of \cite{Lau:Degeneration}, this increase by $1$ is contributed by the term
\[
    \sum_{\cl x\mid x}\dim_{\cl\FF_q}\left(\coker(\beta')_{\cl x}\right).
\]
In contrast, if there is a leg lying over $x$, the increase by $1$ is possibly contributed by the term
\[
    \sum_{i\in I_x}\left(-2\lambda_{i,2}-[(\widehat{\Fcal'_i})_{\cl x_i}:(\widehat{\Fcal'_{i-1}})_{\cl x_i}]\right)
\]
as well. The argument in Remark~\ref{rmk:previous-properness-criterion} considers only the cokernel term in $c_x$ and ignores the latter.

Summing over all closed points of $X$, we obtain
\[
    \sum_{x\in|X|}c_x = \sum_{x\in|X|}\deg(N'_x,\varphi_{N',x}) - 2\sum_{i=1}^2\lambda_{i,2} =2,
\]
and the previous inequality shows that
\[
    2= \sum_{x\in|X|}c_x \geq \sum_{x\in |X|}2[\inv_x(D)]_{\QQ} = |\Ram(D)|.
\]
Thus, if $|\Ram(D)|>2$, such a simple $(D,\varphi)$-subspace $N'$ with $\dim_{\breve F}N'=2$ does not exist regardless of the positions of the legs. As in the proof of Theorem~\ref{thm:from-inequality-to-properness}, this implies properness of $\Sht_{\Dcal^\times,I_\bullet}^{\leqslant\blambda}/a^\ZZ$ over $\cl\FF_q$.
\end{exa}

\subsection*{Acknowledgement}
We thank Tuan Ngo Dac and Arghya Sadhukhan for the helpful comments and suggestions. We are grateful to the anonymous referee for a careful and thorough reading of the manuscript and for many helpful comments.
The first named author was supported by the National Research Foundation of Korea (NRF) grant funded by the Korea government (MSIT) (RS-2026-25508638).
The second named author was supported by the National Research Foundation of Korea (NRF) grant funded by the Korea government (MSIT) (RS-2026-25469397).
The last named author was supported by Basic Science Research Program through the
National Research Foundation of Korea (NRF) funded by the Ministry of Education (RS-2024-00449679) and the National Research Foundation of Korea (NRF) grant funded by the Korea government (MSIT) (RS-2024-00415601).
The first named author thanks Anthropic for providing access to Claude through the Claude Team plan for scientists.

\subsection*{Data availability}
This paper is purely based on a theoretical approach. No data sets were used for the preparation.

\bibliographystyle{amsalpha}
\bibliography{bib}

\providecommand{\bysame}{\leavevmode\hbox to3em{\hrulefill}\thinspace}
\providecommand{\MR}{\relax\ifhmode\unskip\space\fi MR }
\providecommand{\MRhref}[2]{%
  \href{http://www.ams.org/mathscinet-getitem?mr=#1}{#2}
}
\providecommand{\href}[2]{#2}
\begin{thebibliography}{AHLH23}

\bibitem[AHLH23]{AlperHalpenLeistnerHeinloth:ModuliForAlgebraicStacks}
Jarod Alper, Daniel Halpern-Leistner, and Jochen Heinloth, \emph{Existence of moduli spaces for algebraic stacks}, Invent. Math. \textbf{234} (2023), no.~3, 949--1038. \MR{4665776}

\bibitem[ARH14]{ArastehRad-Hartl:LocGlShtuka}
Esmail Arasteh~Rad and Urs Hartl, \emph{Local {$\mathbb{P}$}-shtukas and their relation to global {$\mathfrak{G}$}-shtukas}, M{\"u}nster J. Math. \textbf{7} (2014), no.~2, 623--670. \MR{3426233}

\bibitem[ARH21]{ArastehRad-Hartl:Uniformizing}
\bysame, \emph{Uniformizing the moduli stacks of global {$G$}-shtukas}, Int. Math. Res. Not. IMRN (2021), no.~21, 16121--16192. \MR{4338216}

\bibitem[Bie24]{Bieker:IntModels}
Patrick Bieker, \emph{Integral models of moduli spaces of {S}htukas with deep {B}ruhat-{T}its level structures}, Int. Math. Res. Not. IMRN (2024), no.~14, 10559--10596. \MR{4775669}

\bibitem[BL95]{Beauville-Laszlo:DescentLemma}
Arnaud Beauville and Yves Laszlo, \emph{Un lemme de descente}, C. R. Acad. Sci. Paris S\'er. I Math. \textbf{320} (1995), no.~3, 335--340. \MR{MR1320381 (96a:14049)}

\bibitem[Bre18]{Breutmann:Thesis}
Paul Breutmann, \emph{Functoriality and stratifications of moduli spaces of global {$\mathbb{G}$}-shtukas}, PhD Thesis--Westf\"{a}lischen Wilhelms-Universit\"{a}t M\"{u}nster (2018).

\bibitem[BT72]{BruhatTits:RedGp1}
F.~Bruhat and J.~Tits, \emph{Groupes r{\'e}ductifs sur un corps local. {I}. {D}onn{\'e}es radicielles valu{\'e}es}, Inst. Hautes {\'E}tudes Sci. Publ. Math. (1972), no.~41, 5--251. \MR{0327923}

\bibitem[BT84]{BruhatTits:RedGp2}
\bysame, \emph{Groupes r{\'e}ductifs sur un corps local. {II}. {S}ch{\'e}mas en groupes. {E}xistence d'une donn{\'e}e radicielle valu{\'e}e}, Inst. Hautes {\'E}tudes Sci. Publ. Math. (1984), no.~60, 197--376. \MR{756316 (86c:20042)}

\bibitem[Dri87]{Drinfeld:CpctModuli}
Vladimir~G. Drinfel{$'$}d, \emph{Cohomology of compactified moduli varieties of {$F$}-sheaves of rank {$2$}}, Zap. Nauchn. Sem. Leningrad. Otdel. Mat. Inst. Steklov. (LOMI) \textbf{162} (1987), no.~Avtomorfn. Funkts. i Teor. Chisel. III, 107--158, 189. \MR{918745}

\bibitem[Dri88]{Drinfeld:PeterssonsConjecture}
\bysame, \emph{Proof of the {P}etersson conjecture for {${\rm GL}(2)$} over a global field of characteristic {$p$}}, Funktsional. Anal. i Prilozhen. \textbf{22} (1988), no.~1, 28--43. \MR{936697}

\bibitem[Gar03]{Gardeyn:Analytic-tau-sheaves}
Francis Gardeyn, \emph{The structure of analytic {$\tau$}-sheaves}, J. Number Theory \textbf{100} (2003), no.~2, 332--362. \MR{1978461}

\bibitem[He14]{He:GeomADL}
Xuhua He, \emph{Geometric and homological properties of affine {D}eligne-{L}usztig varieties}, Ann. of Math. (2) \textbf{179} (2014), no.~1, 367--404. \MR{3126571}

\bibitem[He16]{He:KottwitzRapoportConj}
\bysame, \emph{Kottwitz-{R}apoport conjecture on unions of affine {D}eligne-{L}usztig varieties}, Ann. Sci. \'Ec. Norm. Sup\'er. (4) \textbf{49} (2016), no.~5, 1125--1141. \MR{3581812}

\bibitem[HK21]{HamacherKim:Gisoc}
Paul Hamacher and Wansu Kim, \emph{On {$\mathsf{G}$}-isoshtukas over function fields}, Selecta Math. (N.S.) \textbf{27} (2021), no.~4, Paper No. 75, 34. \MR{4292785}

\bibitem[HK25]{HamacherKim:Igusa}
\bysame, \emph{Point counting on {I}gusa varieties for function fields}, Adv. Math. \textbf{480} (2025), Paper No. 110517, 88. \MR{4957546}

\bibitem[HR08]{HainesRapoport:Parahoric}
Thomas~J. Haines and Michael Rapoport, \emph{Appendix: On parahoric subgroups}, Advances in Mathematics \textbf{219} (2008), no.~1, 188--198.

\bibitem[HR17]{HeRapoport:EKOR}
X.~He and M.~Rapoport, \emph{Stratifications in the reduction of {S}himura varieties}, Manuscripta Math. \textbf{152} (2017), no.~3-4, 317--343. \MR{3608295}

\bibitem[HR20]{HainesRicharz:TestFtnRes}
Thomas~J. Haines and Timo Richarz, \emph{The test function conjecture for local models of {W}eil-restricted groups}, Compos. Math. \textbf{156} (2020), no.~7, 1348--1404. \MR{4120166}

\bibitem[HR21]{HainesRicharz:TestFtnParahoric}
\bysame, \emph{The test function conjecture for parahoric local models}, J. Amer. Math. Soc. \textbf{34} (2021), no.~1, 135--218. \MR{4188816}

\bibitem[HX]{HartlXu:UniformizingII}
Urs Hartl and Yujie Xu, \emph{Uniformizing the moduli stacks of global {$G$}-shtukas {II}}.

\bibitem[Kot85]{Kottwitz:Gisoc1}
Robert~E. Kottwitz, \emph{Isocrystals with additional structure}, Compositio Math. \textbf{56} (1985), no.~2, 201--220. \MR{809866 (87i:14040)}

\bibitem[Kot97]{Kottwitz:Gisoc2}
\bysame, \emph{Isocrystals with additional structure. {II}}, Compositio Math. \textbf{109} (1997), no.~3, 255--339. \MR{1485921 (99e:20061)}

\bibitem[KR03]{KottwitzRapoport:Existence}
Robert Kottwitz and Michael Rapoport, \emph{On the existence of {$F$}-crystals}, Comment. Math. Helv. \textbf{78} (2003), no.~1, 153--184. \MR{1966756}

\bibitem[Laf97]{LafforgueL:Ramanujan}
Laurent Lafforgue, \emph{Chtoucas de {D}rinfeld et conjecture de {R}amanujan-{P}etersson}, Ast{\'e}risque (1997), no.~243, ii+329. \MR{1600006 (99c:11072)}

\bibitem[Laf02]{LafforgueL:GlobalLanglands}
\bysame, \emph{Chtoucas de {D}rinfeld et correspondance de {L}anglands}, Invent. Math. \textbf{147} (2002), no.~1, 1--241. \MR{1875184 (2002m:11039)}

\bibitem[Lau97]{Laumon:DrinfeldShtukas}
G.~Laumon, \emph{Drinfeld shtukas}, Vector bundles on curves---new directions ({C}etraro, 1995), Lecture Notes in Math., vol. 1649, Springer, Berlin, 1997, pp.~50--109. \MR{1605028}

\bibitem[Lau04]{Lau:Thesis}
Eike Lau, \emph{On generalised {$D$}-shtukas}, Bonner Mathematische Schriften [Bonn Mathematical Publications], vol. 369, Universit\"{a}t Bonn, Mathematisches Institut, Bonn, 2004, Dissertation, Rheinische Friedrich-Wilhelms-Universit\"{a}t Bonn, Bonn, 2004. \MR{2206061}

\bibitem[Lau07]{Lau:Degeneration}
\bysame, \emph{On degenerations of {$\mathscr{D}$}-shtukas}, Duke Math. J. \textbf{140} (2007), no.~2, 351--389. \MR{2359823 (2009a:11236)}

\bibitem[LRS93]{LaumonRapoportStuhler}
G{{\'e}}rard Laumon, Michael Rapoport, and Ulrich Stuhler, \emph{{$\mathscr{D}$}-elliptic sheaves and the {L}anglands correspondence}, Invent. Math. \textbf{113} (1993), no.~2, 217--338. \MR{1228127 (96e:11077)}

\bibitem[Mad19]{MadapusiPera:Compactification}
Keerthi Madapusi, \emph{Toroidal compactifications of integral models of {S}himura varieties of {H}odge type}, Ann. Sci. \'Ec. Norm. Sup\'er. (4) \textbf{52} (2019), no.~2, 393--514. \MR{3948111}

\bibitem[Mil92]{Milne:ModpPointsGoodRed}
James~S. Milne, \emph{The points on a {S}himura variety modulo a prime of good reduction}, The zeta functions of {P}icard modular surfaces, Univ. Montr{\'e}al, Montreal, QC, 1992, pp.~151--253. \MR{1155229 (94g:11041)}

\bibitem[Ng{\^{o}}06]{NgoBC:DSht}
Bao~Ch\^{a}u Ng{\^{o}}, \emph{{$\mathscr{D}$}-chtoucas de {D}rinfeld {\`a} modifications sym\'{e}triques et identit\'{e} de changement de base}, Ann. Sci. \'{E}cole Norm. Sup. (4) \textbf{39} (2006), no.~2, 197--243. \MR{2245532}

\bibitem[Pap23]{Pappas:Canonical}
Georgios Pappas, \emph{On integral models of {S}himura varieties}, Math. Ann. \textbf{385} (2023), no.~3-4, 2037--2097. \MR{4566689}

\bibitem[PR24]{PappasRapoport:ShtukaCanonical}
Georgios Pappas and Michael Rapoport, \emph{{$p$}-adic shtukas and the theory of global and local {S}himura varieties}, Camb. J. Math. \textbf{12} (2024), no.~1, 1--164. \MR{4701491}

\bibitem[Rei03]{Reiner:MaxOrders}
I.~Reiner, \emph{Maximal orders}, London Mathematical Society Monographs. New Series, vol.~28, The Clarendon Press, Oxford University Press, Oxford, 2003, Corrected reprint of the 1975 original, With a foreword by M. J. Taylor. \MR{1972204}

\bibitem[Ric16]{Richarz:AffGr}
Timo Richarz, \emph{Affine {G}rassmannians and geometric {S}atake equivalences}, Int. Math. Res. Not. IMRN (2016), no.~12, 3717--3767. \MR{3544618}

\bibitem[RR96]{RapoportRichartz:Gisoc}
Michael Rapoport and Melanie Richartz, \emph{On the classification and specialization of {$F$}-isocrystals with additional structure}, Compositio Math. \textbf{103} (1996), no.~2, 153--181. \MR{1411570}

\bibitem[Sch15]{Schieder:HNstr}
Simon Schieder, \emph{The {H}arder-{N}arasimhan stratification of the moduli stack of {$G$}-bundles via {D}rinfeld's compactifications}, Selecta Math. (N.S.) \textbf{21} (2015), no.~3, 763--831. \MR{3366920}

\bibitem[{Sta}]{AlgStackProj}
The {Stacks project authors}, \emph{Stacks project}.

\bibitem[Var04]{Varshavsky:shtuka}
Yakov Varshavsky, \emph{Moduli spaces of principal {$F$}-bundles}, Selecta Math. (N.S.) \textbf{10} (2004), no.~1, 131--166. \MR{2061225}

\bibitem[YZ17]{YunZhang:ShtukaTaylorI}
Zhiwei Yun and Wei Zhang, \emph{Shtukas and the {T}aylor expansion of {$L$}-functions}, Ann. of Math. (2) \textbf{186} (2017), no.~3, 767--911. \MR{3702678}

\bibitem[Zhu14]{Zhu:CoherenceConj}
Xinwen Zhu, \emph{On the coherence conjecture of {P}appas and {R}apoport}, Annals of Mathematics \textbf{180} (2014), no.~1, 1--85.

\end{thebibliography}
\end{document}